\documentclass[11pt]{amsart}
\usepackage{amsmath,amssymb,amsthm,mathrsfs,enumerate,bm,xcolor,multirow,pbox,url}
\usepackage{graphicx,color,framed,tikz,enumitem}
\usetikzlibrary{shapes.geometric}

\usepackage{subfigure,float}

\usepackage[normalem]{ulem} % for strikethrough text

\setlist{leftmargin=5mm}
\allowdisplaybreaks[4]
\numberwithin{equation}{section}
\newcommand{\N}{\mathbb{N}}
\newcommand{\R}{\mathbb{R}}

\newcommand{\E}{\mathbb{E}}
\newcommand{\Prob}{\mathbb{P}}
\newcommand{\G}{\mathbb{G}}

\newcommand{\pnorm}[2]{\lVert#1\rVert_{#2}}

\newcommand{\abs}[1]{\lvert#1\rvert}
\newcommand{\bigabs}[1]{\big\lvert#1\big\rvert}
\newcommand{\biggabs}[1]{\bigg\lvert#1\bigg\rvert}
\newcommand{\iprod}[2]{\langle#1,#2\rangle}

\renewcommand{\epsilon}{\varepsilon}

\renewcommand{\d}[1]{\mathrm{d}#1}

\newcommand{\floor}[1]{\left\lfloor #1 \right\rfloor}

\renewcommand{\hat}{\widehat}
\renewcommand{\tilde}{\widetilde}
\DeclareMathOperator*{\argmin}{argmin}

\newcommand{\magenta}[1]{{\color{magenta}#1}} 
\newcommand{\red}[1]{{\color{red}#1}}
\newcommand{\blue}[1]{{\color{blue}#1}}

\theoremstyle{definition}
\theoremstyle{remark}\newtheorem{assumption}{Assumption}

\theoremstyle{remark}\newtheorem{remark}{Remark}
\theoremstyle{definition}
\theoremstyle{plain}
\theoremstyle{plain}\newtheorem{theorem}{Theorem}
\theoremstyle{plain}\newtheorem{lemma}{Lemma}
\theoremstyle{plain}\newtheorem{proposition}{Proposition}
\theoremstyle{plain}\newtheorem{corollary}{Corollary}
\theoremstyle{plain}\newtheorem{conjecture}{Conjecture}

\AtBeginDocument{%
	\def\MR#1{}
}

\begin{document}

\title[Confidence intervals for multiple isotonic regression]{Confidence intervals for multiple isotonic regression and other monotone models}
\thanks{The research of H. Deng is partially supported by DMS-1454817. The research of Q. Han is partially supported by DMS-1916221. The research of C.-H. Zhang is partially supported by DMS-1513378, DMS-1721495, IIS-1741390 and CCF-1934924. }

\author[H. Deng]{Hang Deng}

\address[H. Deng]{
	Department of Statistics, Rutgers University, Piscataway, NJ 08854, USA.
}
\email{hdeng@stat.rutgers.edu}

\author[Q. Han]{Qiyang Han}

\address[Q. Han]{
Department of Statistics, Rutgers University, Piscataway, NJ 08854, USA.
}
\email{qh85@stat.rutgers.edu}

\author[C.-H. Zhang]{Cun-Hui Zhang}

\address[C.-H. Zhang]{
Department of Statistics, Rutgers University, Piscataway, NJ 08854, USA.
}
\email{czhang@stat.rutgers.edu}

\date{\today}
\keywords{limit distribution theory, confidence interval, multiple isotonic regression, Gaussian process, shape constraints}
\subjclass[2000]{60F17, 62E17}
\maketitle

\begin{abstract}
We consider the problem of constructing pointwise confidence intervals in the multiple isotonic regression model. Recently, Han and Zhang \cite{han2019limit} obtained a pointwise limit distribution theory for the so-called block max-min and min-max estimators \cite{fokianos2017integrated,deng2018isotonic} in this model, but inference remains a difficult problem due to the nuisance parameter in the limit distribution that involves multiple unknown partial derivatives of the true regression function.

In this paper, we show that this difficult nuisance parameter can be effectively eliminated by taking advantage of information beyond point estimates in the block max-min and min-max estimators. 
Formally, let $\hat{u}(x_0)$ (resp. $\hat{v}(x_0)$) be the maximizing lower-left (resp. minimizing upper-right) vertex in the block max-min (resp. min-max) estimator, and $\hat{f}_n$ be the average of the block max-min and min-max estimators. If all (first-order) partial derivatives of $f_0$ are non-vanishing at $x_0$, then the following pivotal limit distribution theory holds:
\begin{align*}
\textstyle \sqrt{n_{\hat{u},\hat{v}}(x_0)} \big(\hat{f}_n(x_0)-f_0(x_0)\big)\rightsquigarrow \sigma\cdot \mathbb{L}_{\bm{1}_d}.
\end{align*}
Here $n_{\hat{u},\hat{v}}(x_0)$ is the number of design points in the block $[\hat{u}(x_0),\hat{v}(x_0)]$, $\sigma$ is the standard deviation of the errors, and $\mathbb{L}_{\bm{1}_d}$ is a universal limit distribution free of nuisance parameters. This immediately yields confidence intervals for $f_0(x_0)$ with asymptotically exact confidence level and oracle length. Notably, the construction of the confidence intervals, even new in the univariate setting, requires no more efforts than performing an isotonic regression once using the block max-min and min-max estimators, and can be easily adapted to other common monotone models including, e.g., (i) monotone density estimation, (ii) interval censoring model with current status data, (iii) counting process model with panel count data, and (iv) generalized linear models. Extensive simulations are carried out to support our theory.

\end{abstract}

%\tableofcontents

\setcounter{theorem}{-1}

\section{Introduction}

\subsection{Overview}

The field of estimation and inference under shape constraints has undergone rapid development in recent years, mostly notably in the direction of estimation theory of multi-dimensional shape constrained models. We briefly give some review of the history and some recent progress:
\begin{itemize}
	\item (\emph{Univariate shape constraints}) Starting from the seminal work of \cite{grenander1956theory,rao1969estimation,rao1970estimation}, estimation of a univariate monotone density or regression function has received much attention, cf. \cite{groeneboom1985estimating,groeneboom1989brownian,groeneboom1995isotonic,zhang2002risk,brunk1970estimation,wright1981asymptotic,chatterjee2015risk,bellec2018sharp,han2019berry}. Estimation of a univariate convex density or regression function is a more challenging task, but considerable progress has been made through the efforts of many authors, cf. \cite{hildreth1954point,hanson1976consistency,groeneboom2001canonical,groeneboom2001estimation,mammen1991nonparametric,guntuboyina2013global,chatterjee2015risk,bellec2018sharp}. Recent years also witnessed much progress in further understanding the behavior of the maximum likelihood estimator (MLE) of a univariate log-concave density, cf. \cite{dumbgen2009maximum,dumbgen2011approximation,balabdaoui2009limit,kim2016global,kim2016adaptation,doss2013global}. Other topics include estimation of unimodal regression functions, cf. \cite{chatterjee2015adaptive,bellec2018sharp}, estimation of a concave bathtub-shaped hazard function, cf. \cite{jankowski2009nonparametric}, and estimation of a $k$-monotone density, cf. \cite{balabdaoui2007estimation}.
	\item (\emph{Multi-dimensional monotonicity constraints}) \cite{chatterjee2018matrix} initiated a study of risk bounds for the least squares estimator (LSE) of a bivariate coordinate-wise non-decreasing regression function.
    \cite{han2017isotonic} extends the results of \cite{chatterjee2018matrix} to the more challenging case $d\geq 3$.
    See also \cite{han2019} for some further improvements. \cite{deng2018isotonic} studied 
    block max-min and min-max estimators originally proposed in \cite{fokianos2017integrated}.
    See also a recent work \cite{fang2019multivariate} for a different notion of multi-dimensional monotonicity.
	\item (\emph{Multi-dimensional convexity constraints}) Convex/concave regression in multi-dimensional settings is initiated in \cite{kuosmanen2008representation}. Consistency of the LSEs is proved in \cite{seijo2011nonparametric,lim2012consistency}.  \cite{han2016multivariate} studied global and adaptive risk bounds for convex bounded LSEs in a random design for $d\leq 3$. \cite{kim2016global} studied global risk bounds for log-concave MLEs, and \cite{feng2018adaptation} studied their adaptation properties, both for $d\leq 3$. \cite{xu2019high} studied log-concave density estimation in high dimensions. 
	Other topics include estimation of $s$-concave densities, cf. \cite{seregin2010nonparametric,koenker2010quasi,han2015approximation,han2019}, and additive modeling, cf. \cite{chen2016generalized}.
\end{itemize}

Despite these remarkable progress in the \emph{estimation} theory of shape constrained models in multivariate settings for various tuning-free estimators, little next to nothing is known about how these merits can be actually useful in making \emph{inference} for the multi-dimensional shape constrained function of interest. The purpose of this paper is to start to fill this gap, within the context of multiple isotonic regression \cite{han2017isotonic,deng2018isotonic}. 

Here is our setup. Consider the regression model 
\begin{align}\label{model:regression}
Y_i = f_0(X_i)+\xi_i,\quad i=1,\ldots,n,
\end{align}
where $X_1,\ldots,X_n$ are design points in $[0,1]^d$ which can be either fixed or random, and $\xi_1,\ldots,\xi_n$ are independent mean-zero errors. The true regression function $f_0$ is assumed to belong to the class of coordinate-wise nondecreasing functions on $[0,1]^d$:
\begin{align*}
f_0 \in \mathcal{F}_d\equiv \{f:[0,1]^d \to \R, f(x)\leq f(y) \textrm{ if }x_i\leq y_i\textrm{ for all }i=1,\ldots,d\}.
\end{align*}

The hope of making some real progress in the inference aspect of this model, beyond purely the estimation theory, is spurred by the recent work of the second and third authors \cite{han2019limit}, who obtained a \emph{pointwise limit distribution theory} for the block max-min and min-max estimators originally proposed in \cite{fokianos2017integrated} and rigorously defined in \cite{deng2018isotonic}. For \emph{any} $x_0 \in [0,1]^d$, let the block max-min and min-max estimators, $\hat{f}_n^{-}$ and $\hat{f}_n^{+}$, be defined as
\begin{align}\label{def:max_min_estimator}
\hat{f}_n^{-}(x_0) &\equiv \max_{u\leq x_0} \min_{ \substack{v\geq x_0\\ [u,v]\cap \{X_i\}\neq \emptyset}} \frac{1}{\abs{ \{i: u\leq X_i\leq v\}}}\sum_{{i: u\leq X_i\leq v}} Y_i\\
&\equiv \max_{u\leq x_0} \min_{ \substack{v\geq x_0\\ [u,v]\cap \{X_i\}\neq \emptyset}} \bar{Y}|_{[u,v]}, \hbox{ and }
\nonumber\\
\hat{f}_n^{+}(x_0) &\equiv \min_{v\geq x_0} \max_{ \substack{u\leq x_0\\ [u,v]\cap \{X_i\}\neq \emptyset}} \bar{Y}|_{[u,v]}.\nonumber
\end{align}
Note that in the univariate case ($d=1$), the block max-min estimator $\hat{f}_n^{-}$ and the block min-max estimator $\hat{f}_n^{+}$ are the same and coincide with the isotonic least squares estimator (LSE) at design points $\{X_i\}$. However, $\hat{f}_n^{-}$ and $\hat{f}_n^{+}$ are in general different and $\hat{f}_n^{-} \le \hat{f}_n^{+}$ is only guaranteed at design points in $d\geq 2$; see \cite{deng2018isotonic} for an explicit example in which the two estimators differ.

If the errors $\xi_i$'s are i.i.d. mean-zero with variance $\sigma^2$, \cite{han2019limit} showed that 
\begin{align}\label{eqn:limit_distribution}
\omega_n^{-1}(\bm{\alpha})\big(\hat{f}_n^{\mp}(x_0)-f_0(x_0)\big)\rightsquigarrow r(\sigma)\cdot K(f_0,x_0)\cdot \mathbb{D}^{\mp}_{\bm{\alpha}}.
\end{align}
Here $\omega_n(\bm{\alpha})$ is the local rate of convergence of $\hat{f}_n^{\mp}$, depending on the `local smoothness' level $\bm{\alpha}$ of $f_0$ at $x_0$ (the precise meaning of this will be clarified in Section \ref{section:CI_isotonic_regression}) 
and the design of the covariates in a fairly complicated way, $r(\sigma)$ is a constant depending on the noise level of the errors, and $K(f_0,x_0)$ is a constant depending on the unknown information concerning the derivatives of $f_0$ at $x_0$. \cite{han2019limit} also showed that the limit distribution theory (\ref{eqn:limit_distribution}) is optimal in a local asymptotic minimax sense.

One may naturally wish to use (\ref{eqn:limit_distribution}) for construction of confidence intervals (CIs) for $f_0(x_0)$, but unfortunately, the complications for using directly the above limit theory for inference are multi-fold:
\begin{enumerate}
	\item The constants $K(f_0,x_0), r(\sigma)$ depend on the unknown information of derivatives of $f_0$ at $x_0$ and the noise level $\sigma$;
	\item The local rate of convergence $\omega_n^{-1}(\bm{\alpha})$ depends on the unknown local smoothness level $\bm{\alpha}$ of $f_0$.
\end{enumerate}
Even one could be content with the knowledge of
the local smoothness level of $f_0$ at $x_0$, for instance assuming all first-order partial derivatives are non-vanishing, the problem of getting a consistent estimate of the nuisance parameter $K(f_0,x_0)$, which involves \emph{many} derivatives of $f_0$ at $x_0$ is already very challenging. Since one of the main features of shape-constrained methods is the avoidance of tuning parameters---which is particularly important in multi-dimensional settings---we would ideally want to avoid estimation of derivatives to begin with. 

A popular tuning-free testing approach for inference in the univariate monotone-response models, put forward in \cite{banerjee2001likelihood,banerjee2007likelihood}, proposes the use of a log likelihood ratio test. The strength of this method lies in the fact that the limit distribution of the log likelihood ratio statistic is \emph{pivotal}, i.e., not depending on nuisance parameters, in particular the derivative of the monotone function of interest, provided it is non-vanishing at the point of interest. Using the quantiles for the pivotal limit distribution, one can then obtain CIs by inverting a family of log likelihood ratio tests. The same idea is further exploited in \cite{doss2016inference,doss2019concave} in the contexts of inference for the mode of a log-concave density and for the value of a concave regression function. 

It is natural to wonder if a similar program, based on likelihood methods, can be extended to multi-dimensional settings, for instance in the multiple isotonic regression model (\ref{model:regression}) we study here. Apart from the apparent lack of any limit distribution theory for the LSE, i.e., the maximum likelihood estimator under Gaussian likelihood, the more fundamental problem is that the LSE does exhibit some undesirable sub-optimal behavior. In particular, as have been clear from the work \cite{han2017isotonic}, the LSE does not adapt to constant functions at the near optimal parametric rate, while the block max-min and min-max estimators (\ref{def:max_min_estimator}) do \cite{deng2018isotonic,han2019limit}. This strongly hints that a limit distribution theory of type (\ref{eqn:limit_distribution}) does not hold for the LSE, or at most can only hold for a very restrictive range of $\bm{\alpha}$, since (\ref{eqn:limit_distribution}) already recovers the parametric rate for constant signals.

Another common approach for avoiding estimation of nuisance parameters in limit distributions is the bootstrap. However, as shown in \cite{kosorok2008bootstrapping,sen2010inconsistency,seijo2011change}, standard bootstrap methods in non-standard problems, in particular those with cube-root asymptotics and non-normal limit distributions, typically lead to \emph{inconsistent} estimates. Although it is in principle possible to develop consistent bootstrap procedures, e.g., $m$-out-of-$n$ bootstrap, or bootstrap with smoothing, cf. \cite{sen2010inconsistency,seijo2011change}, these procedures involve one or more tuning parameters that need to be  carefully calibrated in practice, which unfortunately demerits the tuning-free advantages of shape-constrained methods.

The conceptual and practical difficulties in the likelihood and bootstrap methods lead us to a completely different approach for making inference of $f_0(x_0)$. Our proposal for the construction of the CI for $f_0(x_0)$, as will be detailed in \eqref{def:CI} below, \emph{requires essentially no more efforts than performing an isotonic regression once using the block max-min and min-max estimators}. The key idea for our proposal is to use information beyond point estimates in isotonic regression to directly estimate the scaled magnitude 
$\omega_n^\ast \equiv \omega_n(\bm{\alpha}) r(\sigma) K(f_0,x_0)/ \sigma$ of the 
error of estimating $f_0(x_0)$ in \eqref{eqn:limit_distribution} and therefore to bypass the difficult problem of estimating the
nuisance parameter $K(f_0,x_0)$. More important, the implementation of this idea 
does not require consistent estimation of the scaled magnitude: Given estimates $\hat{f}_n(x_0)$ and $\hat\omega_n^\ast$, 
it requires only the convergence in distribution of the product of $(\hat{f}_n(x_0)-f_0(x_0)\}/\omega_n^\ast$ and the ratio $\omega_n^\ast/\hat\omega_n^\ast$ to a known distribution. See Theorems \ref{thm:pivotal_limit_distribution} and \ref{thm:CI_exact} in the next section and their proofs.

Formally, let $(\hat{u}(x_0),\hat{v}(x_0))$ be any pair such that
\begin{align}\label{def:block_estiamtors}
\hat{f}_n^{-}(x_0) 
&\equiv \max_{u\leq x_0} \min_{ \substack{v\geq x_0\\ [u,v]\cap \{X_i\}\neq \emptyset}} \bar{Y}|_{[u,v]} =  \min_{ \substack{v\geq x_0\\ [u,v]\cap \{X_i\}\neq \emptyset}} \bar{Y}|_{[\hat{u}(x_0),v]},\\
\hat{f}_n^{+}(x_0) 
&\equiv \min_{v\geq x_0} \max_{ \substack{u\leq x_0\\ [u,v]\cap \{X_i\}\neq \emptyset}} \bar{Y}|_{[u,v]} =  \max_{ \substack{u\leq x_0\\ [u,v]\cap \{X_i\}\neq \emptyset}} \bar{Y}|_{[u,\hat{v}(x_0)]}. \nonumber
\end{align}
Let the average of the two estimators $\hat{f}_n^{\mp}$ in (\ref{def:max_min_estimator}) be the \emph{block average estimator} $\hat{f}_n(x_0)$, i.e.,
\begin{align}\label{def:block_avg}
\hat{f}_n(x_0)\equiv \frac{1}{2} \big(\hat{f}_n^{-}(x_0)+\hat{f}_n^{+}(x_0)\big),
\end{align}
and let $n_{\hat{u},\hat{v}}(x_0)$ be the number of design points in the block $[\hat{u}(x_0),\hat{v}(x_0)]$, i.e.,
\begin{align}\label{def:n_uv}
\textstyle n_{\hat{u},\hat{v}}(x_0) \equiv \sum_{i} \bm{1}_{X_i \in [\hat{u}(x_0), \hat{v}(x_0)]}.
\end{align}
\begin{figure}
	% \centering
	% \includegraphics[width=7.5cm]{plots/n_uv.pdf}
	\centering
\begin{tikzpicture}[x=6.5cm, y = 5cm]
\draw [dashed, gray, opacity = 0.7] (0,-0.02) -- (0,1.02);
\draw [dashed, gray, opacity = 0.7] (-0.02,0) -- (1.02,0);
\draw [dashed, gray, opacity = 0.7] (1,-0.02) -- (1,1.02);
\draw [dashed, gray, opacity = 0.7] (-0.02,1) -- (1.02,1);

\foreach \x/\y/\z in {0.03/0.33/1, 0.10/0.13/2, 0.18/0.89/3, 0.26/0.30/4, 0.39/0.19/5, 0.55/0.53/6, 0.62/0.21/7, 0.70/0.42/8, 0.76/0.66/9, 0.93/0.06/10}{
	\draw [dashed, gray, opacity = 0.7] (\x, -0.02) -- (\x, 1.02);
	\draw [dashed, gray, opacity = 0.7] (-0.02, \y) -- (1.02, \y);
	\filldraw [black] (\x,\y) circle (1.5pt);
	\coordinate[label={[label distance=-3pt]135: \scalebox{.5}{$X_{\z}$} } ] ($X_{\z}$) at (\x, \y);
}

\coordinate[label={[label distance=-2pt]270: \scalebox{.7}{\magenta{$\hat{u}(x_0) = \hat{u}^-$}} } ] (\detokenize{$\hat{u}(x_0)$}) at (0.10, 0.13);
\coordinate[label={[label distance=-2pt]45: \scalebox{.7}{\magenta{$\hat{v}^-$}} } ] (\detokenize{$\hat{v}(x_0)$}) at (0.76, 0.89);

\coordinate[label={[label distance=-2pt]270: \scalebox{.7}{\blue{$\hat{u}^+$}} } ] (\detokenize{$\hat{u}(x_0)^+$}) at (0.26, 0.06);
\coordinate[label={[label distance=-2pt]45: \scalebox{.7}{\blue{$\hat{v}(x_0) = \hat{v}^+$}} } ] (\detokenize{$\hat{v}(x_0)$}) at (0.93, 0.66);

\node[star,star points=5, star point ratio=2.25, fill=red, minimum size = 3mm, inner sep=0pt] at (0.5,0.5) {};
\coordinate[label={[label distance=-3pt]-135: \red{\tiny{$x_0$}} } ] ($x_0$) at (0.5, 0.5);

\draw [dash pattern = on 5pt off 5pt, magenta, line width=0.5mm, opacity = 0.5] (0.10, 0.13) -- (0.76, 0.13);
\draw [dash pattern = on 5pt off 5pt, magenta, line width=0.5mm, opacity = 0.5] (0.10, 0.13) -- (0.10, 0.89);
\draw [dash pattern = on 5pt off 5pt, magenta, line width=0.5mm, opacity = 0.5] (0.76, 0.13) -- (0.76, 0.89);
\draw [dash pattern = on 5pt off 5pt, magenta, line width=0.5mm, opacity = 0.5] (0.10, 0.89) -- (0.76, 0.89);

\draw [dash pattern = on 5pt off 5pt, blue, line width=0.5mm, opacity = 0.5] (0.26, 0.06) -- (0.93, 0.06);
\draw [dash pattern = on 5pt off 5pt, blue, line width=0.5mm, opacity = 0.5] (0.26, 0.06) -- (0.26, 0.66);
\draw [dash pattern = on 5pt off 5pt, blue, line width=0.5mm, opacity = 0.5] (0.93, 0.06) -- (0.93, 0.66);
\draw [dash pattern = on 5pt off 5pt, blue, line width=0.5mm, opacity = 0.5] (0.26, 0.66) -- (0.93, 0.66);

\fill [green, opacity = 0.2] (0.10,0.13) rectangle (0.93,0.66);

\end{tikzpicture}
	\caption{Figure illustration of the specification of $\hat{u}(x_0)$ and $\hat{v}(x_0)$.}
	\label{fig:n_uv}
\end{figure}
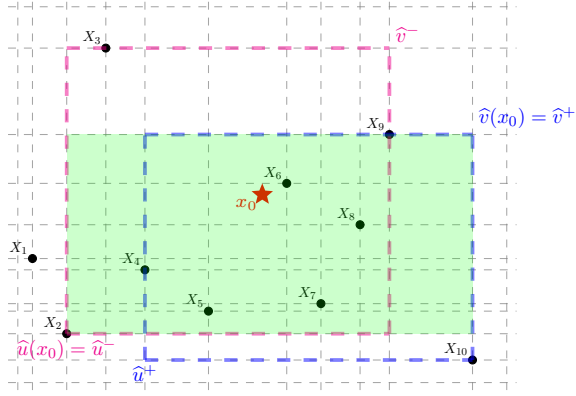
When $(Y_1,\ldots,Y_n)$ are in general positions (i.e., $A\neq B$ implies $\bar{Y}|_A \neq \bar{Y}|_B$ for nonempty $A$ and $B$), the set of design points in the rectangle $[\hat{u}^-,\hat{v}^-]$ giving $\hat{f}^-_n(x_0)$ in (\ref{def:block_estiamtors}) is unique. In this case, $\hat{u}(x_0)=\hat{u}^-$ is unique if we confine our choice to the rectangle $[\hat{u}^-,\hat{v}^-]$ with at least one design point in each of its $2^d$ sides. Similarly, $\hat{v}(x_0)$ is also unique when the solution $[\hat{u}^+,\hat{v}^+]$ for $\hat{f}_n^+(x_0)$ in (\ref{def:block_estiamtors}) is required to have a design point on each side. For such specification of $(\hat{u}(x_0),\hat{v}(x_0))$, $n_{\hat{u},\hat{v}}(x_0)$ defined above is uniquely specified. See Figure \ref{fig:n_uv} for an illustration. In any cases, our theoretical results hold for all feasible pairs $(\hat{u}(x_0),\hat{v}(x_0))$ in (\ref{def:block_estiamtors}).

We propose the following form of CI:
\begin{align}\label{def:CI}
\mathcal{I}_n(c_\delta)\equiv \bigg[\hat{f}_n(x_0) - \frac{c_\delta\cdot \hat{\sigma}}{ \sqrt{n_{\hat{u},\hat{v}}(x_0)}  },\hat{f}_n(x_0) + \frac{c_\delta\cdot \hat{\sigma}}{\sqrt{n_{\hat{u},\hat{v}}(x_0) } } \bigg],
\end{align}
where $\hat{\sigma}$ is the square root of an estimator $\hat{\sigma}^2$ of the variance $\sigma^2$, and $c_\delta>0$ is a critical value chosen by the user that depends only on the confidence level $\delta>0$. 

The crux of our proposal (\ref{def:CI}) is the following \emph{pivotal limit distribution theory}: Under the same conditions as in the limit distribution theory (\ref{eqn:limit_distribution}),
\begin{align}\label{eqn:pivotal_limit_distribution}
\sqrt{n_{\hat{u},\hat{v}}(x_0)}
\big(\hat{f}_n(x_0)-f_0(x_0)\big)\rightsquigarrow \sigma\cdot  \mathbb{L}_{\bm{\alpha}},
\end{align}
where the distribution of $\mathbb{L}_{\bm{\alpha}}$ \emph{does not depend on the nuisance parameter $K(f_0,x_0)$}. Hence, given a consistent variance estimator $\hat{\sigma}^2$, the CI (\ref{def:CI}) can both achieve asymptotically \emph{exact} confidence level $1-\delta$ and shrink at the optimal length on the order of $\omega_n(\bm{\alpha})\cdot r(\sigma) \cdot K(f_0, x_0)$, provided that the local smoothness $\bm{\alpha}$ is known. 
This is the case, e.g., if one assumes that all first-order partial derivatives are non-vanishing, much as in \cite{banerjee2001likelihood,banerjee2007likelihood} in the univariate case that assumes a non-vanishing first derivative for the monotone function at the point of interest. By relaxing the requirement of asymptotically \emph{exact} confidence level, it is also possible, by \emph{calibrating the critical value $c_\delta$ alone}, to construct conservative CIs (\ref{def:CI}) that adapt to 
any given range of local smoothness levels $\bm{\alpha}$, while maintaining the optimal order of the length as in the limit theory (\ref{eqn:limit_distribution}). One natural question here is that whether the likelihood approach of \cite{banerjee2001likelihood,banerjee2007likelihood} in the much simpler univariate setting, wins over our proposal (\ref{def:CI}) in terms of adaptation to unknown local smoothness $\bm{\alpha}$. As will be clear in Section \ref{section:CI_isotonic_regression}, the limit distributions of log likelihood ratio tests also depend on $\bm{\alpha}$, so indeed the likelihood approach of \cite{banerjee2001likelihood,banerjee2007likelihood} by itself does not offer a stronger degree of universality from the perspective of adaptation.

At a deeper level, the viewpoint of our construction for the CI for $f_0(x_0)$ is markedly different from that of \cite{banerjee2001likelihood,banerjee2007likelihood}. We treat the nuisance parameters $K(f_0,x_0)$ and $r(\sigma)$ differently in the limit theory (\ref{eqn:limit_distribution}), with a particular view that \emph{it is $K(f_0,x_0)$ that constitutes the main difficulty of using (\ref{eqn:limit_distribution}) for inference, while the problem of $r(\sigma)$ is relatively minor. } The underlying reason for this is that $K(f_0,x_0)$ involves the information for derivatives of $f_0$ at $x_0$, which cannot be obtained in a simple way from point estimates that take local averages, while $r(\sigma)$ can be relatively easily estimated using known methods (e.g., difference estimators \cite{rice1984bandwidth,hall1991estimation,munk2005difference}), or large samples (if available) in the data-driven local block $[\hat{u}(x_0),\hat{v}(x_0)]$.

The idea for the construction of the proposed CI (\ref{def:CI}) has a much broader scope of applications beyond the isotonic regression model (\ref{model:regression}). In Section \ref{section:more_examples} similar constructions of CIs are exploited in a number of other models with monotonicity shape constraints, including (i) monotone density estimation \cite{rao1970estimation,groeneboom1985estimating,groeneboom1989brownian}, (ii) interval censoring model with current status data \cite{groeneboom1992information}, (iii) estimation of the mean function of a counting process with panel count data \cite{wellner2000two}, and more generally, (iv) generalized linear models with monotonicity. These new CIs share the same general scheme that the constructions utilize the local information encoded by the analogue of $\hat{u}(x_0),\hat{v}(x_0)$ as in the regression setting, and require essentially no more efforts than performing the (maximum likelihood) estimation procedure once.

The results of this paper, in particular the pivotal limit distribution theory (\ref{eqn:pivotal_limit_distribution}) and the resulting CI (\ref{def:CI}), make a significant further step in developing practical inference methods using the block estimators (\ref{def:max_min_estimator}) beyond the limit theory (\ref{eqn:limit_distribution}) developed in \cite{han2019limit}, especially in view of the dependence of the constant factor $K(f_0,x_0)$ on the partial derivatives of the unknown $f_0$. However, the techniques used in proving (\ref{eqn:limit_distribution}) in \cite{han2019limit} serve as the foundation for establishing the limit theory in (\ref{eqn:pivotal_limit_distribution}) in this paper. In addition, as a key technical ingredient in proving (\ref{eqn:pivotal_limit_distribution}), we show that the limiting Gaussian white noise versions of properly rescaled $\hat{u}(x_0),\hat{v}(x_0)$ are almost surely well-defined (see Lemma \ref{lem:uniqueness_supinf}), so the limit distribution $\mathbb{L}_{\bm{\alpha}}$ in (\ref{eqn:pivotal_limit_distribution}) is indeed well-defined.

The rest of the article is organized as follows. In Section \ref{section:CI_isotonic_regression} we give a review of the limit distribution theory (\ref{eqn:limit_distribution}) developed in \cite{han2019limit}, study the proposed CI \eqref{def:CI}, and present the pivotal limit distribution theory (\ref{eqn:pivotal_limit_distribution}). 
Some comparisons with the Banerjee-Wellner likelihood based inference method are also detailed in Section \ref{section:CI_isotonic_regression}. In Section \ref{section:more_examples}, we illustrate the generality of our method of constructing CIs in the four models mentioned above. Section \ref{section:simulation} contains extensive simulation results that demonstrate the accuracy of the coverage probability of our proposed CIs, along with a detailed numerical comparison with the Banerjee-Wellner CIs \cite{banerjee2001likelihood,banerjee2007likelihood}. For clarity of presentation, proofs are deferred to the Appendix.

\subsection{Notation}\label{section:notation}
For the simplicity of presentation, we write the CI $[\widehat{\theta} - c_0, \widehat{\theta} + c_0]$ which is symmetric around $\widehat{\theta}$ as $\mathcal{I} = [ \widehat{\theta} \pm c_0]$. 

For a real-valued measurable function $f$ defined on $(\mathcal{X},\mathcal{A},P)$, $\pnorm{f}{L_p(P)}\equiv \pnorm{f}{P,p}\equiv \big(P\abs{f}^p)^{1/p}$ denotes the usual $L_p$-norm under $P$, and $\pnorm{f}{\infty}\equiv \sup_{x \in \mathcal{X}} \abs{f(x)}$. Let $(\mathcal{F},\pnorm{\cdot}{})$ be a subset of the normed space of real functions $f:\mathcal{X}\to \R$. For $\epsilon>0$ let $\mathcal{N}(\epsilon,\mathcal{F},\pnorm{\cdot}{})$ be the $\epsilon$-covering number of $\mathcal{F}$; see \cite[pp. 83]{van1996weak} for more details.

For two real numbers $a,b$, $a\vee b\equiv \max\{a,b\}$ and $a\wedge b\equiv\min\{a,b\}$. For $x \in \R^d$, let $\pnorm{x}{p}$ denote its $p$-norm $(0\leq p\leq \infty)$. For any $x, y \in \R^d$, $x\leq y$ if and only if $x_i\leq y_i$ for all $1\leq i\leq d$. Let $[x,y] \equiv \prod_{k=1}^d [x_k\wedge y_k,x_k\vee y_k]$, $xy\equiv (x_ky_k)_{k=1}^d$, and $x\wedge (\vee) y \equiv (x_k \wedge (\vee) y_k)_{k=1}^d$. 
Let $\bm{1}_d=(1,\ldots,1)\in\R^d$. For $\ell_1, \ell_2 \in \{1,\ldots,d\}$, we let $\bm{1}_{[\ell_1:\ell_2]} \in \R^d$ be such that $(\bm{1}_{[\ell_1:\ell_2]})_k = \bm{1}_{\ell_1\leq k\leq \ell_2}$. We use $C_{x}$ to denote a generic constant that depends only on $x$, whose numeric value may change from line to line unless otherwise specified. $a\lesssim_{x} b$ and $a\gtrsim_x b$ mean $a\leq C_x b$ and $a\geq C_x b$ respectively, and $a\asymp_x b$ means $a\lesssim_{x} b$ and $a\gtrsim_x b$ [$a\lesssim b$ means $a\leq Cb$ for some absolute constant $C$]. $\mathcal{O}_{\mathbf{P}}$ and $\mathfrak{o}_{\mathbf{P}}$ denote the usual big and small O notation in probability. $\rightsquigarrow$ is reserved for weak convergence. For two integers $k_1>k_2$, we interpret $\sum_{k=k_1}^{k_2}\equiv 0, \prod_{k=k_1}^{k_2}\equiv 1$. We also interpret $(\infty)^{-1}\equiv 0, 0/0\equiv 0$.

\section{Confidence interval: Isotonic regression}\label{section:CI_isotonic_regression}

\subsection{Limit distribution theory in \cite{han2019limit}: a review}
Let us now describe the setting under which limit distribution theory for the block max-min and min-max estimators (\ref{def:max_min_estimator}) is developed in \cite{han2019limit}. The exposition below largely follows \cite{han2019limit}.

First some further notation. For $f: \R^d \to \R$, and $k \in \{1,\ldots,d\}$, $\alpha_k \in \mathbb{Z}_{\geq 1}$, let $\partial_k^{\alpha_k} f(x)\equiv \frac{d^{\alpha_k}}{d x_k^{\alpha_k}}f(x)$. For a multi-index $\bm{j}=(j_1,\ldots,j_d) \in \mathbb{Z}_{\geq 0}^d$, let $\partial^{\bm{j}} \equiv \partial_{1}^{j_1}\cdots \partial_d^{j_d}$, and $\bm{j}! \equiv j_1!\cdots j_d!$ and $x^{\bm{j}} \equiv x_1^{j_1}\ldots x_d^{j_d}$ for $x \in \R^d$. For $\bm{\alpha}=(\alpha_1,\ldots,\alpha_d) \in \mathbb{Z}_{\geq 1}^d$ in Assumption \ref{assump:smoothness} below, i.e., for some $0\leq s\leq d$, $1\leq \alpha_1,\ldots,\alpha_s<\infty = \alpha_{s+1}=\ldots=\alpha_d$, let $J(\bm{\alpha})$ (resp. $J_\ast(\bm{\alpha})$) be the set of all $\bm{j}=(j_1,\ldots,j_d) \in \mathbb{Z}_{\geq 0}^d$ satisfying $0<\sum_{k=1}^s j_k/\alpha_k \leq 1$ (resp. $\sum_{k=1}^s j_k/\alpha_k = 1$) and $j_k = 0$ for $s+1\leq k\leq d$, and let $J_0(\bm{\alpha})\equiv J(\bm{\alpha})\cup \{\bm{0}\}$. We often write $J=J(\bm{\alpha})$, $J_\ast = J_\ast(\bm{\alpha})$ and $J_0 = J_0(\bm{\alpha})$ if no confusion arises.

\begin{assumption}\label{assump:smoothness}
	$f_0$ is coordinate-wise nondecreasing (i.e., $f_0 \in \mathcal{F}_d$), and is $\bm{\alpha}$-smooth at $x_0$ with intrinsic dimension $s$, $\bm{\alpha}=(\alpha_1,\ldots,\alpha_d)$ with integers $1\leq \alpha_1,\ldots,\alpha_s<\infty = \alpha_{s+1}=\ldots=\alpha_d$, $0\leq s\leq d$, in the sense that $\partial_k^{j_k} f_0(x_0)=0$ for $1\leq j_k\leq \alpha_k-1$ and $\partial_k^{\alpha_k} f_0(x_0)\neq 0$, $1\leq k\leq s$, and in rectangles of the form $\cap_{k=1}^d \{\abs{(x-x_0)_k}\leq L_0\cdot(r_n)_k\}$, $r_n = (\omega_n^{1/\alpha_1},\ldots,\omega_n^{1/\alpha_d})$ with $\omega_n>0$, the Taylor expansion of $f_0$ satisfies for all $L_0>0$,
	\begin{align*}
		\lim_{\omega_n \downarrow 0} \omega_n^{-1} \sup_{ \substack{x \in[0,1]^d,\\ \abs{(x-x_0)_k}\leq L_0 \cdot (r_n)_k, \\1\leq k\leq d} } \biggabs{f_0(x)- \sum_{\bm{j} \in J_0} \frac{\partial^{\bm{j}}f_0(x_0)}{\bm{j}!}(x-x_0)^{\bm{j}} }= 0.
	\end{align*}
\end{assumption}

The above assumption will be satisfied if $f_0$ depends only on its first $s$ coordinates, and is locally $C^{\max_{1\leq k\leq s}\alpha_k}$ at $x_0$, with $\partial_k^{j_k} f_0(x_0)=0$ for $1\leq j_k\leq \alpha_k-1$ and $\partial_k^{\alpha_k} f_0(x_0)\neq 0$, $1\leq k\leq s$.

\begin{assumption}\label{assump:design}
	The design points $\{X_i\}_{i=1}^n$ satisfy either of the following:
	\begin{itemize}
		\item (\emph{Fixed design}) $\{X_i\}$'s follow a $\bm{\beta}$-fixed lattice design: there exist some $\{\beta_1,\ldots,\beta_d\}\subset (0,1)$ with $\sum_{k=1}^d \beta_k=1$ such that $x_0 \in \{X_i\}_{i=1}^n = \prod_{k=1}^d \{x_{1,k},\ldots,x_{n_k,k}\}$, where $\{x_{1,k},\ldots,x_{n_k,k}\}$ are equally spaced in $[0,1]$ (i.e., $\abs{x_{j,k}-x_{j+1,k}}=1/n_k$ for all $j=1,\ldots,n_k-1$) and $n_k = \floor{n^{\beta_k}}$. 
		\item (\emph{Random design}) 
		$\{X_i\}$'s follow i.i.d. uniform random design in $[0,1]^d$ and are also independent of $\{\xi_i\}$'s.
	\end{itemize}
\end{assumption}

In $\bm{\beta}$-fixed lattice design,
we assume without loss of generality
\begin{align}\label{def:ordering}
0\leq \alpha_1 \beta_1 \leq \ldots \leq \alpha_s\beta_s\leq \ldots \leq \alpha_d \beta_d \leq \infty.
\end{align}
Otherwise we may find a permutation of $\{1,\ldots,d\}$ to satisfy the above condition and the theory below will be carried over for the permuted indices.

In the random design, we assume for simplicity that the law $P$ of $X_i$ is uniform on $[0,1]^d$; the forthcoming Theorem \ref{thm:limit_distribution_pointwise} holds with minor changes when $P$ is relaxed to have Lebesgue density $\pi$ that is bounded away from $0$ and $\infty$ on $[0,1]^d$ and is continuous over an open set containing the region $\big\{\big((x_0)_1,\ldots,(x_0)_s,x_{s+1},\ldots,x_d\big): 0\leq x_k\leq 1, s+1\leq k\leq d\big\}$. More discussion on the above assumptions is referred to \cite{han2019limit}. 
The following limit distribution theory is obtained by \cite{han2019limit}.

\begin{theorem}\label{thm:limit_distribution_pointwise}
	Let $x_0 \in (0,1)^d$. Suppose Assumptions \ref{assump:smoothness} and \ref{assump:design} hold, and the errors $\{\xi_i\}$ are i.i.d. mean-zero with finite variance $\E \xi_1^2=\sigma^2<\infty$ (and are independent of $\{X_i\}$ in the random design case). Let $\kappa_\ast,n_\ast$ be defined by
	\vspace{0.5ex}
	\setlength{\tabcolsep}{8pt} 
	\renewcommand{\arraystretch}{1.2} 
	\begin{center}
		\begin{tabular}{|c||c|c|}
			\hline 
			& $\bm{\beta}$-fixed lattice design & random design\\
			\hline\hline
			$\kappa_\ast$ & $\arg\max\limits_{1\leq \ell \leq d} \frac{\sum_{k=\ell}^d \beta_k}{2+\sum_{k=\ell}^s \alpha_k^{-1}}$ & $1$\\
			\hline
			$n_\ast$ & $n^{\sum_{k=\kappa_\ast}^d \beta_k}$ & $n$ \\
			\hline
		\end{tabular}
	\end{center}
	\vspace{0.5ex}
	If $\kappa_\ast$ is uniquely defined, then for some finite random variables $\mathbb{C}^{\mp }(f_0,x_0)$, 
	\begin{align*}
	(n_\ast/\sigma^2)^{\frac{1}{2+\sum_{k=\kappa_\ast}^s \alpha_k^{-1}}}\big(\hat{f}_n^{\mp}(x_0)-f_0(x_0)\big) \rightsquigarrow \mathbb{C}^{\mp}(f_0,x_0).
	\end{align*}
	Furthermore, if either $\{\alpha_k\}_{k=1}^s$ is a set of relative primes, i.e., the greatest common divisor of $\{\alpha_{k_1},\alpha_{k_2}\}$ is $1$ for all $1\leq k_1<k_2\leq s$, or all mixed derivatives $\partial^{\bm{j}} f_0$ of $f_0$ vanish at $x_0$ for all $\bm{j}\in J_\ast$, then 
	\begin{align*}
	\mathbb{C}^{\mp}(f_0,x_0) =_d K(f_0,x_0)\cdot \mathbb{D}^{\mp}_{\bm{\alpha}},
	\end{align*} 
	where $K(f_0,x_0) = \big\{\prod_{k=\kappa_\ast}^s\big({\partial_k^{\alpha_k} f_0(x_0)}/{(\alpha_k+1)!}\big)^{1/\alpha_k}\big\}^{\frac{1}{2+\sum_{k=\kappa_\ast}^s \alpha_k^{-1}}}$, and $\mathbb{D}^{\mp}_{\bm{\alpha}}$ are given by
	\begin{align*}
	\mathbb{D}^{-}_{\bm{\alpha}} &\equiv \sup_{g_1 \in \mathscr{G}_1}\inf_{g_2 \in \mathscr{G}_2} \mathbb{V}_{\bm{\alpha}}(g_1,g_2), \quad	\mathbb{D}^{+}_{\bm{\alpha}} \equiv \inf_{g_2 \in \mathscr{G}_2} \sup_{g_1 \in \mathscr{G}_1}\mathbb{V}_{\bm{\alpha}}(g_1,g_2).
	\end{align*}
	Here 
	\begin{align*}
	&\mathscr{G}_1\equiv \{g_1 \in \R^d: g_1>0, (g_1)_k\leq (x_0)_k, s+1\leq k\leq d\},\\
	&\mathscr{G}_2\equiv \{g_2 \in \R^d: g_2>0, (g_2)_k\leq (1-x_0)_k, s+1\leq k\leq d\},\\
		&\mathbb{V}_{\bm{\alpha}}(g_1,g_2)\equiv \frac{ \G(g_1,g_2)}{\prod_{k=\kappa_\ast}^d \big((g_1)_k+(g_2)_k\big)}+\sum_{k=\kappa_\ast}^s \frac{ (g_2)_k^{\alpha_k+1}-(g_1)_k^{\alpha_k+1}}{(g_2)_k+(g_1)_k},
	\end{align*}
	where $\G$ is a Gaussian process defined on $\R^{d}_{\geq 0}\times \R^{d}_{\geq 0}$ with the following covariance structure: for any $(g_1,g_2), (g_1',g_2')$,
	\begin{align*}
	&\mathrm{Cov}\big(\G(g_1,g_2),\G(g_1',g_2')\big)= \prod_{k=\kappa_\ast}^d \big( (g_1)_k\wedge (g_1')_k+ (g_2)_k\wedge (g_2')_k\big).
	\end{align*}
\end{theorem}

Strictly speaking, \cite{han2019limit} proved the case for the block max-min estimator $\hat{f}_n^{-}$, but the case for the block min-max estimator $\hat{f}_n^{+}$ follows from the same proofs. We refer the readers to \cite{han2019limit} for detailed discussion and some concrete examples for $\kappa_\ast,n_\ast$. The merit of using the block average estimator $\hat{f}_n$ is discussed in Section \ref{subsec:blocl-average}.

Theorem \ref{thm:limit_distribution_pointwise} is comprehensive under the general Assumptions \ref{assump:smoothness} and \ref{assump:design}. To capture the essence, we consider a simple yet important case in the following Corollary \ref{coro:simple}.

\begin{corollary}\label{coro:simple}
Suppose $f_0$ is locally $C^1$ at $x_0 \in (0,1)^d$ with $\partial_k f_0(x_0)>0$ for all $1\leq k\leq d$, and the design points $\{X_i\}$ either (i) form a balanced fixed lattice design with $\beta_1=\cdots=\beta_d = 1/d$, or (ii) follow a uniform random design on $[0,1]^d$. Suppose the errors $\{\xi_i\}$ are i.i.d. mean-zero with finite variance $\E \xi_1^2=\sigma^2<\infty$ (and are independent of $\{X_i\}$ in the random design case). Then,
\begin{align*}
(n/\sigma^2)^{1/(2+d)} \big(\hat{f}_n^{\mp}(x_0)-f_0(x_0)\big) \rightsquigarrow \bigg\{\prod_{k=1}^d \big(\partial_k f_0(x_0)/2\big)\bigg\}^{1/(2+d)}\cdot \mathbb{D}_{\bm{1}_d}^\mp.
\end{align*}
\end{corollary}

\begin{remark}\label{remark:error_independence}
	Theorem \ref{thm:limit_distribution_pointwise} (and Corollary \ref{coro:simple}) in the random design case requires the independence of the errors and the random design points due to the precise form of the limit distribution. Although this independence assumption is not most desirable, it is common in limit distribution theories for other one-dimensional shape-constrained regression estimators under random design settings, for instance in the convex regression model \cite{ghosal2017univariate}. It is an interesting open question to see if the theory carries over to the more general settings, for instance $\E[\xi_i|X_i]=0$ and $\E[\xi_i^2|X_i=x]=\sigma^2(x)$ for some smooth enough function $\sigma$.
\end{remark}

\subsection{Pivotal limit distribution theory}

In this subsection, we formally establish the pivotal limit distribution theory (\ref{eqn:pivotal_limit_distribution}). The main idea of the pivotal limit distribution theory (\ref{eqn:pivotal_limit_distribution}) is that the information for $K(f_0,x_0)$ is already encoded in $\hat{u}(x_0),\hat{v}(x_0)$, so after proper scaling, we may naturally view $\sqrt{n_{\hat{u},\hat{v}}(x_0)/\sigma^2}$ as an estimator for 
$\{\omega_n(\bm{\alpha})r(\sigma)K(f_0,x_0)\}^{-1}
= (n_\ast/\sigma^2)^{1/{(2+\sum_{k=\kappa_\ast}^s \alpha_k^{-1})}}/K(f_0,x_0)$.
Indeed, we have the following:

\begin{theorem}\label{thm:pivotal_limit_distribution}
Let $x_0 \in (0,1)^d$. Suppose Assumptions \ref{assump:smoothness} and \ref{assump:design} hold, and the errors $\{\xi_i\}$ are i.i.d. mean-zero with finite variance $\E \xi_1^2=\sigma^2<\infty$ (and are independent of $\{X_i\}$ in the random design case). If the limit distribution in Theorem \ref{thm:limit_distribution_pointwise} is of the explicit form $\mathbb{C}^{\mp} (f_0,x_0) = K(f_0,x_0)\cdot \mathbb{D}^{\mp}_{\bm{\alpha}}$, then with $\hat{f}_n(x_0)$ and $n_{\hat{u},\hat{v}}(x_0)$ defined respectively in (\ref{def:block_avg}) and (\ref{def:n_uv})
\begin{align*}
\sqrt{n_{\hat{u},\hat{v}}(x_0)}
\big(\hat{f}_n(x_0)-f_0(x_0)\big)\rightsquigarrow \sigma\cdot \mathbb{L}_{\bm{\alpha}}.
\end{align*}
Here $\mathbb{L}_{\bm{\alpha}}$ is a finite random variable defined by
\begin{align*}
\mathbb{L}_{\bm{\alpha}}&\equiv  \mathbb{S}_{\bm{\alpha}}(g_{1,\bm{\alpha}}^\ast,g_{2,\bm{\alpha}}^\ast)\cdot\frac{1}{2}\bigg(\sup_{g_1 \in \mathscr{G}_1}\inf_{g_2 \in \mathscr{G}_2} \mathbb{V}_{\bm{\alpha}}(g_{1},g_{2})+\inf_{g_2 \in \mathscr{G}_2} \sup_{g_1 \in \mathscr{G}_1 }\mathbb{V}_{\bm{\alpha}}(g_{1},g_{2})\bigg),
\end{align*}
where $g_{1,\bm{\alpha}}^\ast$ and $g_{2,\bm{\alpha}}^\ast$ are almost surely uniquely determined via
\begin{align*}
\inf_{g_2 \in \mathscr{G}_2} \mathbb{V}_{\bm{\alpha}}(g_{1,\bm{\alpha}}^\ast,g_{2}) &= \sup_{g_1 \in \mathscr{G}_1}\inf_{g_2 \in \mathscr{G}_2} \mathbb{V}_{\bm{\alpha}}(g_{1},g_{2}), \\
\sup_{ g_1 \in \mathscr{G}_1} \mathbb{V}_{\bm{\alpha}}(g_1,g_{2,\bm{\alpha}}^\ast ) &= \inf_{g_2 \in \mathscr{G}_2} \sup_{g_1 \in \mathscr{G}_1 }\mathbb{V}_{\bm{\alpha}}(g_{1},g_{2}),
\end{align*}
and
\begin{align*}
\mathbb{S}_{\bm{\alpha}}(g_{1,\bm{\alpha}}^\ast,g_{2,\bm{\alpha}}^\ast) \equiv \textstyle \sqrt{\prod_{k=\kappa_\ast}^s (g_{2,\bm{\alpha}}^\ast+g_{1,\bm{\alpha}}^\ast)_k}, 
\end{align*}
with the $\mathscr{G}_i(i=1,2)$, $\mathbb{V}_{\bm{\alpha}}$ and $\kappa_\ast$ in Theorem \ref{thm:limit_distribution_pointwise}. In particular, $\mathbb{L}_{\bm{\alpha}}$ does not depend on $K(f_0,x_0)$.
\end{theorem}

Theorem \ref{thm:pivotal_limit_distribution} provides a \emph{pivotal} limit distribution theory in the sense that the limit distribution of certain statistic concerning $\hat{f}_n(x_0)-f_0(x_0)$ does not depend on the \emph{difficult nuisance parameter $K(f_0,x_0)$}. 

This pivotal phenomenon can be understood from an oracle perspective. As an illustration, we focus on the leading case $\bm{\alpha}=(1,\ldots,1)$ and $\sigma=1$ in the setting of balanced fixed lattice design. In this case, $n_\ast =n$ and $\kappa_\ast =1$. The oracle bandwidth vector $h^\ast =h^\ast(x_0)$ balances the bias and variance:
\begin{align*}
h_\ell^\ast\partial_\ell f_0(x_0) \approx \frac{1}{\sqrt{\prod_{k=1}^d \big(n^{1/d} h_{k}^\ast\big)}}, \hbox{ for } 1\le \ell \le d,
\end{align*}
which yields 
\begin{align*}
\textstyle h_\ell^\ast\approx \big(\partial_\ell f_0(x_0)\big)^{-1} \cdot n^{-1/(2+d)}\big(\prod_{k=1}^d \partial_{k} f_0(x_0)\big)^{1/(2+d)}.
\end{align*}
Hence, with $u^\ast=u^\ast(x_0)= x_0- h^\ast/2$ and $v^\ast=v^\ast(x_0)= x_0 + h^\ast/2$,
\begin{align*}
&
\bigabs{
\textstyle \sqrt{n_{u^\ast,v^\ast}(x_0)}
\big(\hat{f}_n(x_0)-f_0(x_0)\big)}\\
&\approx \textstyle \sqrt{n \prod_{k=1}^d (v_k^\ast-u_k^\ast) } \cdot n^{-1/(d+2)}K(f_0,x_0) \cdot \abs{\mathcal{O}_{\mathbf{P}}(1)} \\
& = \textstyle \sqrt{n^{2/(2+d)} \big(\prod_{k=1}^d \partial_k f_0(x_0)\big)^{-2/(2+d)} } \cdot n^{-1/(d+2)}K(f_0,x_0) \cdot \abs{\mathcal{O}_{\mathbf{P}}(1)}  \\
&= \mathrm{const.}\times \abs{\mathcal{O}_{\mathbf{P}}(1)},
\end{align*}
where $\abs{\mathcal{O}_{\mathbf{P}}(1)}$ denotes a universal random variable. Theorem \ref{thm:pivotal_limit_distribution} can then be understood as the data driven bandwidth vectors $\hat{u}(x_0),\hat{v}(x_0)$ mimic the above oracle vectors $u^\ast(x_0),v^\ast(x_0)$ in achieving a pivotal limiting behavior. We note that the asymptotic variance of $\hat{f}_n(x_0)$ is 
proportional to 
$\big(n^{-1}\prod_{k=1}^d \partial_k f_0(x_0)\big)^{2/(2+d)}$.

\subsection{Confidence interval}\label{subsection:CI} 
The pivotal limit distribution theory in Theorem \ref{thm:pivotal_limit_distribution} naturally implies the tuning free CI \eqref{def:CI}. In this subsection, we study
\eqref{def:CI} under fixed lattice and uniform random designs as in Assumption \ref{assump:design}. The non-uniform random design case will be discussed at the end.

To construct the CI, it remains to find a good estimate for the variance $\sigma^2$. Estimation of variance in the nonparametric regression model (\ref{model:regression}) is a well studied topic in the literature; for instance, we may use the class of difference estimators (\cite{rice1984bandwidth,hall1991estimation,munk2005difference}).  Below we present a `principled' estimator that shows the reason why $\sigma^2$ is much easier to estimate than $K(f_0,x_0)$---point estimates already contain enough information for the variance, as long as a law of large numbers is satisfied. Formally, let
\begin{align}\label{def:var_est}
\sigma_{\hat{u},\hat{v}}^2\equiv \frac{1}{n_{\hat{u},\hat{v}}(x_0)} \sum_{X_i \in [\hat{u}(x_0),\hat{v}(x_0)]} (Y_i-\hat{f}_n(x_0))^2.
\end{align} 
Note that $\sigma_{\hat{u},\hat{v}}^2$ only requires information of the observed $\{Y_i\}$ in the data driven neighborhood $[\hat{u}(x_0),\hat{v}(x_0)]$ of $x_0$ and the fitted value $\hat{f}_n(x_0)$. Intuitively, since we have large samples in $[\hat{u}(x_0),\hat{v}(x_0)]$, it is natural to expect good performance of $\sigma_{\hat{u},\hat{v}}^2$ in the large sample limit. In fact, we have:

\begin{proposition}\label{prop:consistent_var_est}
Under the conditions of Theorem 1, $\sigma_{\hat{u},\hat{v}}^2\to_p \sigma^2$.
\end{proposition}

One theoretical advantage of (\ref{def:var_est}) compared with the difference estimators is that $\sigma_{\hat{u},\hat{v}}^2$ takes local average around $x_0$, and therefore may \emph{in principle} estimate the variance even in the heteroscedastic regression setting,  
e.g., when the variance of $\bar{Y}|_{[u,v]}$ is given by $\sigma^2_{u,v}$ with a 
certain strictly positive and continuous $\sigma_{u,v}$ defined on the entire $[0,1]^d\times [0,1]^d$. The practical issue, however, is that the effective sample size in $[\hat{u}(x_0),\hat{v}(x_0)]$ is relatively small so (\ref{def:var_est}) typically requires very large samples to achieve accurate variance estimation, in particular for $d\geq 2$.

With a consistent variance estimator, Theorem \ref{thm:pivotal_limit_distribution} can then be used to justify the use of the CI of the form defined in (\ref{def:CI}).

We first consider the leading case $\bm{\alpha}=\bm{1}_d$, where it is possible to construct \emph{asymptotically exact CIs}. 

\begin{theorem}\label{thm:CI_exact}
Let $c_\delta>0$ be a continuity point of the d.f. of $\abs{\mathbb{L}_{\bm{1}_d}}$ such that 
\begin{align}\label{cond:critical_value}
\Prob(\abs{\mathbb{L}_{\bm{1}_d}}>c_\delta) = \delta.
\end{align}
For any consistent variance estimator $\hat{\sigma}^2$, i.e., $\hat{\sigma}^2\to_p \sigma^2$, the CI $\mathcal{I}_n(c_\delta)$ defined in (\ref{def:CI}) satisfies
\begin{align}\label{eqn:coverage}
\lim_{n \to \infty} \Prob_{f_0}\big(f_0(x_0) \in \mathcal{I}_n(c_\delta)\big) = 1-\delta.
\end{align}
Furthermore, with $K(f_0,x_0)=\big(\prod_{k=1}^d (\partial_k f_0(x_0)/2)\big)^{1/(d+2)}$ as in 
Theorem~\ref{thm:limit_distribution_pointwise}, for any $\epsilon>0$,
\begin{align}\label{eqn:length}
\liminf_{n \to \infty} \Prob_{f_0} \bigg(\abs{\mathcal{I}_n(c_\delta)} 
< 2c_\delta \mathfrak{g}_\epsilon\cdot  (\sigma^2/n)^{1/(d+2)} K(f_0, x_0)
\bigg) \geq 1-\epsilon.
\end{align}
Here
$\mathfrak{g}_\epsilon \in (0,\infty)$ is such that 
\begin{align}
\Prob\big(\mathbb{S}_{\bm{1}_d}^{-1}\geq\mathfrak{g}_\epsilon\big)\leq\epsilon.
\end{align}
\end{theorem}

Theorem \ref{thm:CI_exact} shows that if the critical value $c_\delta$ is chosen according to (\ref{cond:critical_value}), then the CI $\mathcal{I}_n(c_\delta)$ achieves the asymptotic exact confidence level $1-\delta$. Furthermore, the CI $\mathcal{I}_n(c_\delta)$ shrinks at the optimal length, being automatically adaptive to the unknown information on the derivatives of $f_0$ at $x_0$, i.e., $K(f_0,x_0)$.

The choice of the critical value $c_\delta$ depends on the distribution of $\mathbb{L}_{\bm{1}_d}$. Since $\mathbb{L}_{\bm{1}_d}$ does not depend on the unknown regression function $f_0$, it is possible to simulate $c_\delta$ for different values of confidence levels $\delta>0$. See Section \ref{section:simulation} for more details on simulated critical values of $\mathbb{L}_{\bm{1}_d}$ for $d=1,2,3$.

Note that the CI in Theorem \ref{thm:CI_exact}, although enjoying the advantage of achieving asymptotically exact confidence level, is not adaptive to the unknown local smoothness $\bm{\alpha}$ of the isotonic regression function $f_0$ at $x_0$. By relaxing the requirement of \emph{exact} CI and calibrating the critical value $c_\delta$ only, the CI $\mathcal{I}_n(c_\delta)$ may adapt to all local smoothness level $\bm{\alpha}$ as well as the unknown information of $f_0$ at $x_0$ expressed by $K(f_0,x_0)$. More concretely, we have:
 
\begin{theorem}\label{thm:adaptive_confidence_interval}
Let $c_\delta>0$ be chosen such that 
\begin{align}\label{cond:critical_value_adaptive}
\sup_{\bm{\alpha}}\Prob(\abs{\mathbb{L}_{\bm{\alpha}}}>c_\delta) \leq \delta.
\end{align}
For any consistent variance estimator $\hat{\sigma}^2$, i.e., $\hat{\sigma}^2\to_p \sigma^2$, the CI $\mathcal{I}_n(c_\delta)$ defined in (\ref{def:CI}) satisfies
	\begin{align}\label{eqn:coverage_adaptive}
	\liminf_{n \to \infty} \Prob_{f_0}\big(f_0(x_0) \in \mathcal{I}_n(c_\delta)\big) \geq 1-\delta.
	\end{align}
	Furthermore, with $\kappa_\ast,n_\ast$ 
	and $K(f_0, x_0)$ defined in Theorem \ref{thm:limit_distribution_pointwise}, for any $\epsilon>0$,
	\begin{align}\label{eqn:length_adaptive}
	\liminf_{n \to \infty} \Prob_{f_0} \bigg(\abs{\mathcal{I}_n(c_\delta)} 
	< 2c_\delta \mathfrak{g}_{\epsilon,\bm{\alpha}}(\sigma^2/n_\ast)^{\frac{1}{2+\sum_{k=\kappa_\ast}^s \alpha_k^{-1}}} K(f_0,x_0)\bigg) \geq 1-\epsilon.
	\end{align}
	Here $\mathfrak{g}_{\epsilon,\bm{\alpha}} \in (0,\infty)$ is such that 
	\begin{align}
	\Prob\big(\mathbb{S}_{\bm{\alpha}}^{-1}(g_{1,\bm{\alpha}}^\ast,g_{2,\bm{\alpha}}^\ast)
	\geq \mathfrak{g}_{\epsilon,\bm{\alpha}}\big)\leq \epsilon.
	\end{align}
\end{theorem}

To make the above theorem useful for construction of 
$\bm{\alpha}$-adaptive CIs, it is crucial to choose the critical value $c_\delta>0$ such that (\ref{cond:critical_value_adaptive}) is satisfied. The proposition below shows that this is indeed possible.

\begin{proposition}\label{prop:uniform_tail_L}
The following holds for some constant $L_0>0$ depending only on $d,x_0$: for any $t\geq 1$,
\begin{align*}
\sup_{\bm{\alpha}} \Prob\big(\abs{\mathbb{L}_{\bm{\alpha}}}>t\big)\leq L_0 \exp(-t^{4/(d+2)}/L_0).
\end{align*}
\end{proposition}

Hence it suffices to choose $c_\delta \asymp_{d,x_0} \log^{(d+2)/4}(1/\delta)$ to satisfy (\ref{cond:critical_value_adaptive}).

\medskip
\noindent \textbf{Non-uniform random design.} So far we have assumed that the design distribution $P$ is uniform on $[0,1]^d$ in the random design case. The situation will be more complicated for general design distributions $P$, as the limit distribution in Theorem \ref{thm:limit_distribution_pointwise} can depend on $P$ in a rather complicated and non-local way. Note that for general $P$, Theorem \ref{thm:limit_distribution_pointwise} requires the Lebesgue density $\pi$ of $P$ to be bounded away from $0$ and $\infty$ on $[0,1]^d$ and to be continuous in the neighborhood of $x_0$. In the special case where $s=d$ (i.e., $f_0(x)$ depends on all elements of $x$ at $x = x_0$), the effect of $P$ is local and can be factored out in the limit distribution theory in Theorem \ref{thm:limit_distribution_pointwise}; see \cite[Remark 1 (5)]{han2019limit} for detailed discussion. In this case, using similar arguments as in the proof of Theorem \ref{thm:pivotal_limit_distribution}, we still have
\begin{align*}
\sqrt{n_{\hat{u}, \hat{v}}(x_0)}
\big(\hat{f}_n(x_0)-f_0(x_0)\big)\rightsquigarrow \sigma \cdot \mathbb{L}_{\bm{1}_d}.
\end{align*}
Hence for any consistent variance estimator $\hat{\sigma}^2$, (\ref{def:CI}) continues to be an asymptotic $1- \delta$ CI of $f_0(x_0)$ which shrinks at the optimal length in a similar sense to the statements of Theorem \ref{thm:CI_exact}.

\subsection{Comparison with the approach of \cite{banerjee2001likelihood} in $d=1$}\label{section:compare_BW_theory}

\cite{banerjee2001likelihood,banerjee2007confidence} developed an inference procedure using the log likelihood ratio test for various monotone-response models in the univariate case $d=1$. In the regression setting, this idea is best illustrated in the random design, with $P$ being the uniform distribution on $[0,1]$ and the errors $\{\xi_i\}$ being i.i.d. $\mathcal{N}(0,1)$. The block max-min and min-max estimators $\hat{f}_n^{\mp}$
defined via (\ref{def:max_min_estimator}) and their average $\hat{f}_n$ all
reduce to the univariate LSE at design points.

We consider testing the hypothesis
\begin{align*}
H_0: f_0(x_0) = m_0\quad \textrm{vs.}\quad H_1: f_0(x_0)\neq m_0.
\end{align*}
To form a likelihood ratio test, let $\hat{f}_n^0$ be the \emph{constrained} least squares estimator defined via
\begin{align*}
\hat{f}_n^0\in \argmin_{f \in \mathcal{F}_1, f(x_0)=m_0} \sum_{i=1}^n (Y_i-f(X_i))^2.
\end{align*}
$\hat{f}_n^0$ is well-defined on the design points, and can be computed by performing two isotonic regressions on $\{(Y_i,f(X_i)): X_i\leq x_0\}$ and $\{(Y_i,f(X_i)): X_i> x_0\}$ followed by thresholding (see \cite[pp. 939]{banerjee2007likelihood}). Under the Gaussian likelihood, the likelihood ratio test statistic is given by
\begin{align}\label{def:log_likelihood_ratio}
2\log \lambda_n(m_0) = -\sum_{i=1}^n (Y_i-\hat{f}_n(X_i))^2 + \sum_{i=1}^n (Y_i-\hat{f}_n^0(X_i))^2.
\end{align}
\cite{banerjee2001likelihood,banerjee2007confidence} showed that if $f_0$ is locally $C^1$ at $x_0$ with $f_0'(x_0)>0$ and $H_0$ holds, then
\begin{align}\label{eqn:banerjee_wellner_limit}
2\log \lambda_n(m_0) \rightsquigarrow \mathbb{K}_1,
\end{align}
where the distribution $\mathbb{K}_1$ is free of the nuisance parameter $f_0'(x_0)$ that would otherwise be present in the limit distribution theory (cf. Theorem \ref{thm:limit_distribution_pointwise} in the simplest case $d=1,\alpha=1$). A CI of $f_0(x_0)$ can now be obtained through inversion of (\ref{eqn:banerjee_wellner_limit}): Let $\mathcal{I}^{\mathrm{BW}}_n(d_\delta)\equiv \{m_0: 2\log \lambda_n(m_0)\leq d_\delta\}$ (BW refers to Banerjee-Wellner) with $\Prob\big(\mathbb{K}_1>d_\delta\big)=\delta$. Then
\begin{align*}
\Prob_{f_0}\big(f_0(x_0) \in \mathcal{I}^{\mathrm{BW}}_n(d_\delta) \big) \to \Prob\big(\mathbb{K}_1\leq d_\delta\big)=1-\delta.
\end{align*}

\begin{remark}
For a fixed $m_0$, the likelihood ratio test requires two isotonic regressions to calculate the test statistic (\ref{def:log_likelihood_ratio}), which can be computed efficiently thanks to the fast PAVA algorithms. With carefully written algorithms the inversion of (\ref{eqn:banerjee_wellner_limit}) may not add too much computational burden to obtain the likelihood ratio test based CIs, see e.g., \cite{groeneboom2015rcpp} for a fast algorithm in the related current status model. However, the proposed procedure \eqref{def:CI} is still computationally simpler and more straightforward.
\end{remark}	

The validity of the Banerjee-Wellner CI crucially relies on the assumption $f_0'(x_0)>0$. Compared with our procedure, it is natural to wonder if the likelihood ratio approach offers a stronger degree of universality in terms of adaptation to unknown local smoothness of the regression function. As we will show below, the limit distribution of the log likelihood ratio test statistic does depend on the number of vanishing derivatives of $f_0$, and is therefore not adaptive to the local smoothness of $f_0$.

To formally state the result, let $\mathrm{slogcm}(f,I)$ be the left-hand slope of the greatest convex minorant of $f$ restricted to the interval $I$. Write $\mathrm{slogcm}(f)= \mathrm{slogcm}(f,\R)$ for simplicity. Let 
\begin{align*}
\mathrm{slogcm}^0(f) = \big(\mathrm{slogcm}(f,(-\infty,0])\wedge 0\big)\bm{1}_{(-\infty,0]}+\big(\mathrm{slogcm}(f,(0,\infty))\vee 0\big)\bm{1}_{(0,\infty)}.
\end{align*}
Let $\mathbb{B}$ be the standard two-sided Brownian motion started from $0$, and $X_{a,b;\alpha}(t)\equiv a \mathbb{B}(t) + b t^{\alpha+1}$. Let $g_{a,b;\alpha} \equiv \mathrm{slogcm}(X_{a,b;\alpha})$ and $g_{a,b;\alpha}^0\equiv \mathrm{slogcm}^0(X_{a,b;\alpha})$. These quantities are a.s. well-defined for $b>0$ and an odd integer $\alpha\geq 1$ as $X_{a,b;\alpha}(t)$ is of the order $\mathcal{O}_{a.s.}(t^{\alpha+1})$ for $t\to \pm\infty$ and is a.s. bounded on compacta. Informally, $g_{a,b;\alpha}$ is the `isotonic regression for $X_{a,b;\alpha}$' in the Gaussian white noise model $\d{X_{a,b;\alpha}(t)} = b(\alpha+1) t^{\alpha}\,\d{t}+ a\,\d{\mathbb{B}(t)}$, and $g^0_{a,b;\alpha}$ is the `constrained isotonic regression' subject to $g^0_{a,b;\alpha}(0)=0$.

\begin{theorem}\label{thm:LRT}
Consider the above setting. Suppose $f_0$ is nondecreasing and locally $C^\alpha$ at $x_0$ for some $\alpha\geq 1$, with $\partial^{j} f_0(x_0)=0$ for $j=1,\ldots,\alpha-1$ and $\partial^{\alpha} f_0(x_0)\neq 0$. Then under $H_0$,
\begin{align*}
2 \log \lambda_n(m_0) \rightsquigarrow \int_{\R} \big\{(g_{1,1;\alpha}(t))^2-(g_{1,1;\alpha}^0(t) )^2 \big\}\ \d{t}\equiv \mathbb{K}_{\alpha}.
\end{align*}
\end{theorem}

\begin{remark}
By \cite[Lemma 1]{han2019limit}, $\alpha$'s satisfying the assumption of Theorem \ref{thm:LRT} must be odd, and $\partial^\alpha f_0(x_0)>0$.
\end{remark}

It is clear from Theorem \ref{thm:LRT} that the limit distribution for the log likelihood ratio test depends on the unknown local smoothness level $\alpha$ of $f_0$ through the slope processes $g_{1,1;\alpha},g_{1,1;\alpha}^0$. This phenomenon is observed numerically in \cite{doss2016inference} in another related setting:  the limit distributions for the log-likelihood ratio tests for the mode of a log-concave density depend on the number of vanishing derivatives at the mode. 

\begin{remark}\label{rmk:LRT_unknown_var}
If the variance $\sigma^2$ is unknown, then the log-likelihood ratio test statistic involves the unknown $\sigma^2$. By taking (\ref{def:log_likelihood_ratio}) as the definition of the quantity of $2\log \lambda_n(m_0)$, it holds (under the same conditions as in Theorem \ref{thm:LRT}) that $
2 \log \lambda_n(m_0)\rightsquigarrow \sigma^2 \cdot \mathbb{K}_{\alpha}$.
\end{remark}

\section{Beyond isotonic regression: Inference in other monotone models}\label{section:more_examples}

The idea for constructing tuning-free CIs in the previous section has a much broader scope beyond the setting of multiple isotonic regression.  As a proof of concept, in this section we construct CIs for a few further non-parametric models with certain monotonicity shape constraints, adapting essentially the same idea as developed in the previous section.

\subsection{The common scheme}

We briefly outline the common scheme for the construction of CIs in the models to be studied in detail below. Suppose we want to estimate a 
univariate monotone function $f_0$. There is a natural piecewise constant estimator $\hat{f}_n$ (usually the maximum likelihood estimator) for $f_0$ that exhibits a non-standard limit distribution at the point of interest $x_0$, typically at a cube-root rate $\omega_n= n^{-1/3}$ under curvature conditions on $f_0'$ at $x_0$:
\begin{align*}
\omega_n^{-1}\big(\hat{f}_n(x_0)-f_0(x_0)\big)&\rightsquigarrow \sup_{h_1>0}\inf_{h_2>0} \bigg[a \cdot \frac{\G(h_1,h_2)}{h_1+h_2} + b\cdot (h_2-h_1)\bigg]\\
& =_d \inf_{h_2>0} \sup_{h_1>0}\bigg[a \cdot \frac{\G(h_1,h_2)}{h_1+h_2} + b\cdot (h_2-h_1)\bigg] \\
& =_d (a^2 b)^{1/3}\cdot \mathbb{D}_1.
\end{align*}
Recall that $\mathbb{D}_1\equiv \mathbb{D}_1^+=_d \mathbb{D}_1^-$ is defined in Theorem \ref{thm:limit_distribution_pointwise}. In fact, $\mathbb{D}_1/2$  follows the well-known Chernoff distribution (cf. \cite{groeneboom2014nonparametric}). 

Here we have two nuisance parameters, namely $a,b$:
\begin{itemize}
	\item $b$ is a \emph{difficult} nuisance parameter to estimate, which usually involves the derivatives of the monotone function to be estimated. This will be tackled by the analogue of `$n_{\hat{u},\hat{v}}(x_0)$'.
	\item $a$ is typically \emph{easy} to estimate, either via observed samples or via fitted values in the analogue of the `local block $[\hat{u}(x_0),\hat{v}(x_0)]$' of $x_0$.
\end{itemize}
One special feature in the one-dimensional setting is the exchangability of supremum and infimum in the limit distribution, so $\hat{u}(x_0)$ and $\hat{v}(x_0)$ are simply the left and right end-points of the constant piece of $\hat{f}_n$ that contains $x_0$. As $\hat{v}(x_0)-\hat{u}(x_0)$ is of order $\mathcal{O}_{\mathbf{P}}(\omega_n (a/b)^{2/3})$, we would expect the following pivotal limit distribution theory:
\begin{align}\label{eqn:pivot_limit_generic}
\textstyle \sqrt{n(\hat{v}(x_0)-\hat{u}(x_0))}\big(\hat{f}_n(x_0)-f_0(x_0)\big)\rightsquigarrow a\cdot \mathbb{L}_1.
\end{align}

Now given a consistent estimate $\hat{a}_n$ of $a$, we have the following generic CI of $f_0(x_0)$:
\begin{align*}
\mathcal{I}_n^\ast(c_\delta)\equiv \Big[\hat{f}_n(x_0)\pm c_\delta \cdot \hat{a}_n \big/ \textstyle \sqrt{n(\hat{v}(x_0)-\hat{u}(x_0))} \Big].
\end{align*}
Similar to the regression setting, the construction of CIs in specific models to be detailed below is tuning-free, and requires essentially no further efforts beyond a single step of (maximum likelihood) estimation.

\subsection{Monotone density estimation}

Consider the classical problem of estimating a decreasing density $f_0$ on $[0,\infty)$ based on i.i.d. observations $X_1,\ldots,X_n$. The maximum likelihood estimator (MLE) $\hat{f}_n$, known as the \emph{Grenander estimator}, is the left derivative of the least concave majorant of the empirical distribution function $\mathbb{F}_n$. By the max-min representation, for any $x_0 \in (0,\infty)$, we may write
\begin{align*}
\hat{f}_n(x_0) = \inf_{0<u<x_0}\sup_{v\geq x_0} \frac{\mathbb{F}_n(v)-\mathbb{F}_n(u)}{v-u} = \frac{\mathbb{F}_n(\hat{v}(x_0))-\mathbb{F}_n(\hat{u}(x_0))}{\hat{v}(x_0)-\hat{u}(x_0)},
\end{align*}
where $(\hat{u}(x_0),\hat{v}(x_0))$ is the a.s. uniquely specified pair for which the last equality in the above display holds. It is well known (see e.g., \cite{rao1969estimation,groeneboom1985estimating,groeneboom1989brownian,van1996weak}) that if $f_0$ is locally $C^1$ at $x_0$ with $f_0'(x_0)<0$ and $f_0(x_0)>0$, then
\begin{align*}
n^{1/3}\big(\hat{f}_n(x_0)-f_0(x_0)\big)&\rightsquigarrow \sup_{h_1>0}\inf_{h_2>0}\bigg[\sqrt{f_0(x_0)}\cdot \frac{\G(h_1,h_2)}{h_1+h_2}+\frac{1}{2}\abs{f'_0(x_0)} (h_2-h_1)\bigg]\\
 & =_d \big(f_0(x_0)\abs{f_0'(x_0)}/2\big)^{1/3}\cdot  \mathbb{D}_1.
\end{align*}
To use the above limit theorem to form CI, the difficult nuisance to estimate is $f_0'(x_0)$, while the easy one is $f_0(x_0)$. The inference problem in the density setting is recently tackled in \cite{groeneboom2015nonparametric}, using both the log likelihood ratio test approach similar to \cite{banerjee2001likelihood,banerjee2007likelihood} and a bootstrap assisted approach for the smoothed maximum likelihood estimator. Our proposal for a CI of $f_0(x_0)$ is the following:
\begin{align*}
\textstyle \mathcal{I}_n^{\textrm{den}}(c_\delta)\equiv \Big[\hat{f}_n(x_0)\pm c_{\delta} \cdot \sqrt{\hat{f}_n(x_0)} \big/ \textstyle \sqrt{n(\hat{v}(x_0)-\hat{u}(x_0))} \Big]\cap [0,\infty).
\end{align*}
Let $\mathbb{L}_1$ and $\mathbb{S}_1$ be as in 
Theorem \ref{thm:pivotal_limit_distribution} with $d=1$ and $\bm{\alpha}=1$.

\begin{theorem}\label{thm:CI_grenander}
Suppose $f_0$ is locally $C^1$ at $x_0$ with $f_0'(x_0)<0$ and $f_0(x_0)>0$. Let $c_\delta>0$ be a continuity point of the d.f. of $\abs{\mathbb{L}_{1}}$ such that $
\Prob\big(\abs{\mathbb{L}_1}>c_\delta\big)=\delta$. Then 
\begin{align*}
\lim_{n \to \infty} \Prob_{f_0}\big(f_0(x_0) \in \mathcal{I}_n^{\textrm{den}}(c_\delta)\big) = 1-\delta.
\end{align*}
Furthermore, for any $\epsilon>0$,
\begin{align*}
\liminf_{n \to \infty} \Prob_{f_0} \bigg(\abs{\mathcal{I}_n^{\textrm{den}}(c_\delta)}< 2c_\delta \mathfrak{g}_\epsilon\cdot  n^{-1/3} \big(f_0(x_0)\abs{f_0'(x_0)}/2\big)^{1/3} \bigg) \geq 1-\epsilon.
\end{align*}
Here $\mathfrak{g}_\epsilon\in (0,\infty)$ is such that $
\Prob\big(\mathbb{S}_1^{-1}\geq \mathfrak{g}_\epsilon\big)\leq\epsilon$.
\end{theorem}

It is also possible to consider adaptive CIs by calibrating the critical value $c_\delta$ similarly as in Theorem \ref{thm:adaptive_confidence_interval}. We omit the details.

\subsection{Current status data: interval censoring model}

Let $X_1,\ldots,X_n$ and $T_1,\ldots,T_n$ be independent i.i.d. samples from distribution functions $F_0$ and $G_0$ supported on $[0,\infty)$. Let $\Delta_i \equiv \bm{1}_{X_i\leq T_i}$. We observe $(\Delta_1,T_1),\ldots,(\Delta_n,T_n)$ and want to estimate $F_0$, the distribution of unobserved $X_1,\ldots,X_n$. Consider the maximum likelihood estimator $\hat{F}_n$ that maximizes 
\begin{align}\label{ineq:likelihood_interval_censoring}
F \mapsto \sum_{i=1}^n \big( \Delta_i \log F(T_i) + (1-\Delta_i) \log(1-F(T_i))\big).
\end{align}
Let $T_{(1)}\leq \ldots \leq T_{(n)}$ be the order statistics of $T_1,\ldots,T_n$. It is well-known (see e.g., \cite{groeneboom1992information,van1996weak}) that the solutions $(\hat{F}_n(T_{(1)}),\ldots,\hat{F}_n(T_{(n)}))$ is given by the isotonic regression over $(\Delta_{(i)} = \bm{1}_{X_{(i)}\leq T_{(i)}} )_{i=1}^n$. In other words, for any $t_0 \in (0,\infty)$,
\begin{align*}
\hat{F}_n(t_0) = \max_{i: T_i\leq t_0} \min_{j: T_j\geq t_0} \frac{\sum_{k=i}^j \Delta_{(k)} }{j-i+1} \equiv \max_{u\leq t_0} \min_{v\geq t_0} \bar{\Delta}_{(\cdot)}|_{[u,v]}= \bar{\Delta}_{(\cdot)}|_{[\hat{u}(t_0),\hat{v}(t_0)]},
\end{align*} 
where $(\hat{u}(t_0),\hat{v}(t_0))$ is any pair for which the last equality in the above display holds. It is also well-known (see e.g., \cite{groeneboom1992information,van1996weak}) that if $F_0,G_0$ has positive and locally continuous density $f_0,g_0$ at $t_0$, then
\begin{align*}
&n^{1/3}\big(\hat{F}_n(t_0)-F_0(t_0)\big)\\
&\rightsquigarrow \sup_{h_1>0}\inf_{h_2>0} \bigg[\sqrt{F_0(t_0)(1-F_0(t_0))/g_0(t_0)}\cdot \frac{\G(h_1,h_2)}{h_1+h_2}+\frac{1}{2}f_0(t_0) (h_2-h_1)\bigg]\\
& =_d  \big(F_0(t_0)(1-F_0(t_0))f_0(t_0)/2g_0(t_0)\big)^{1/3}\cdot  \mathbb{D}_1.
\end{align*}
The inference problem in the current status model is investigated in \cite{banerjee2001likelihood,groeneboom2015nonparametric} using likelihood ratio methods; see \cite{banerjee2008estimating} for similar likelihood ratio based inference methods in the context of monotone, uni-modal and U–shaped failure rates under a right–censoring mechanism. Here we take a different approach, similar to our proposal in the regression setting. Note that the difficult nuisance parameter in this problem is $f_0(t_0)$ since $X_1,\ldots,X_n$ are unobserved, while $F_0(t_0)$ and $g_0(t_0)$ are easy to estimate. For instance, we may use $\hat{F}_n(t_0)$ to estimate $F_0(t_0)$, and 
\begin{align*}
\hat{g}_n(t_0) \equiv \sum_i \bm{1}_{T_i \in [\hat{u}(t_0),\hat{v}(t_0)]}\big/\{ n(\hat{v}(t_0)-\hat{u}(t_0)) \}
\end{align*}
to estimate $g_0(t_0)$. Now consider the following CI for $F_0(t_0)$:
\begin{align*}
\mathcal{I}_n^{\textrm{cur}}(c_\delta) 
&\equiv \bigg[ \hat{F}_n(t_0)\pm c_\delta \cdot \sqrt{\hat{F}_n(t_0)(1-\hat{F}_n(t_0))/\hat{g}_n(t_0) } \big/ \sqrt{n(\hat{v}(t_0)-\hat{u}(t_0))} \bigg] \cap [0,1]
\\ \nonumber
&= \textstyle \bigg[ \hat{F}_n(t_0)\pm  c_\delta \cdot \sqrt{\hat{F}_n(t_0)(1-\hat{F}_n(t_0)) }  \Big/ \sqrt{\sum_i \bm{1}_{T_i \in [\hat{u}(t_0),\hat{v}(t_0)]} } \bigg] \cap [0,1].
\end{align*}

\begin{theorem}\label{thm:CI_interval_censoring}
	Suppose $F_0,G_0$ has positive and locally continuous density $f_0,g_0$ at $t_0$. Let $c_\delta>0$ be a continuity point of the d.f. of $\abs{\mathbb{L}_{1}}$ such that $
	\Prob\big(\abs{\mathbb{L}_1}>c_\delta\big)=\delta$. Then 
	\begin{align*}
	\lim_{n \to \infty} \Prob_{F_0,G_0}\big(F_0(t_0) \in \mathcal{I}_n^{\textrm{cur}}(c_\delta)\big) = 1-\delta.
	\end{align*}
	Furthermore, for any $\epsilon>0$,
	\begin{align*}
	&\liminf_{n \to \infty} \Prob_{F_0,G_0} \bigg(\abs{\mathcal{I}_n^{\textrm{cur}}(c_\delta)}\\
	&\qquad \qquad < 2c_\delta \mathfrak{g}_\epsilon\cdot  n^{-1/3}\big(F_0(t_0)(1-F_0(t_0))f_0(t_0)/2g_0(t_0)\big)^{1/3} \bigg) \geq 1-\epsilon.
	\end{align*}
	Here $\mathfrak{g}_\epsilon\in (0,\infty)$ is such that $
	\Prob\big(\mathbb{S}_1^{-1}\geq \mathfrak{g}_\epsilon\big)\leq\epsilon$.
\end{theorem}

\subsection{Panel count data: counting process model}

The examples in previous subsections are amongst the `classical' ones in the field of monotonicity-constrained estimation. Below we consider one further example, in the context of \emph{panel count data}, that is less `classical' due to its increased complexity. The inference problem for this model is previous studied in \cite{sen2008pseudo} using likelihood ratio methods.

Here is the setup. We follow the notation in \cite{wellner2000two}. Suppose that $N = \{N(t):t\geq 0\}$ is a counting process with mean function $\Lambda_0(t) = \E N(t)$. Let $K$ be an integer-valued random variable, and $T=\{T_{k,j}: 1\leq j\leq k, k\geq 1\}$ be an triangular array of observation times. We assume that $N$ and $(K,T)$ are independent and $T_{k,j-1}\leq T_{k,j}$. 
Let $X=(N_K,T_K,K)$, where $T_K =(T_{K,1},\ldots,T_{K,K})$ and $N_K = (N(T_{K,1}),\ldots,N(T_{K,K}))$. We observe i.i.d. copies $X_1,\ldots,X_n$ of $X$, where $X_i = (N_{K_i}^{(i)}, T_{K_i}^{(i)},K_i)$. The problem is to estimate $\Lambda_0(t)$.
By building a Poisson model for $N(t)\sim_d \mathrm{Poisson}(\Lambda_0(t))$, and pretending independence of the events 
within each sample $X_i$, we may consider the estimator $\hat{\Lambda}_n$ that 
maximizes the pseudo log-likelihood
\begin{align}\label{ineq:likelihood_panel_count}
\Lambda \mapsto \sum_{i=1}^n \sum_{j=1}^{K_i} \big[N_{K_i,j}^{(i)} \log \Lambda(T_{K_i,j}^{(i)})- \Lambda(T_{K_i,j}^{(i)}) \big].
\end{align}
Let $s_1<s_2<\ldots<s_m$ be the ordered distinct observation time points in the set $\{T_{K_i,j}^{(i)}: 1\leq j\leq K_i, i=1,\ldots,n\}$. For $1\leq \ell\leq m$, define
\begin{align*}
w_\ell \equiv \sum_{i=1}^n \sum_{j=1}^{K_i} \bm{1}_{T_{K_i,j}^{(i)} = s_\ell},\quad \bar{N}_\ell \equiv \frac{1}{w_\ell}\sum_{i=1}^n \sum_{j=1}^{K_i} N_{K_i,j}^{(i)}\bm{1}_{T_{K_i,j}^{(i)} = s_\ell}.
\end{align*}
It is known (see e.g., \cite{sun1995estimation,wellner2000two}) that
\begin{align*}
\hat{\Lambda}_n(t_0) =\max_{s_i\leq t_0}\min_{s_j\geq t_0} \frac{\sum_{p=i}^j w_p \bar{N}_p}{\sum_{p=i}^j w_p} = \frac{\sum_{p: \hat{u}(t_0)\leq s_p\leq \hat{v}(t_0)} w_p \bar{N}_p}{\sum_{p: \hat{u}(t_0)\leq s_p\leq \hat{v}(t_0)} w_p},
\end{align*}
where $(\hat{u}(t_0),\hat{v}(t_0))$ is any pair in $\{s_1,\ldots,s_m\}^2$ such that the right hand side of the above display holds. Under the assumption that $\Lambda_0$ is non-decreasing and locally $C^1$ with $\Lambda'_0(t_0)>0$ and further regularity conditions, \cite{wellner2000two} proved the following limit distribution theory for $\hat{\Lambda}_n(t_0)$:
\begin{align*}
&n^{1/3}\big(\hat{\Lambda}_n(t_0)-\Lambda_0(t_0)\big)\\
&\rightsquigarrow \sup_{h_1>0}\inf_{h_2>0} \bigg[\sqrt{\sigma^2(t_0)/g(t_0)}\cdot \frac{\G(h_1,h_2)}{h_1+h_2}+\frac{1}{2}\Lambda'_0(t_0) (h_2-h_1)\bigg]\\
& =_d  \big(\sigma^2(t_0)\Lambda_0'(t_0)/2g(t_0)\big)^{1/3}\cdot  \mathbb{D}_1.
\end{align*}
Here $\sigma^2(t_0) \equiv \mathrm{Var}(N(t_0))$ and $g(t_0) \equiv \sum_{k=1}^\infty \Prob(K=k) \sum_{j=1}^k g_{k,j}(t_0)$ with $g_{k,j}$ denoting the Lebesgue density of $T_{k,j}$. 

The difficult nuisance parameter in this problem is $\Lambda_0'(t_0)$, and easier ones are $\sigma^2(t_0)$ and $g(t_0)$. For instance, with $n_{\hat{u}, \hat{v}}(t_0)$ being the number of $\{T_{K_i,j}^{(i)}\}$ in the interval $[\hat{u}(t_0),\hat{v}(t_0)]$, i.e.,
\begin{align*}
\textstyle n_{\hat{u}, \hat{v}}(t_0) = \sum_{i=1}^n \sum_{j=1}^{K_i}  \bm{1}_{T_{K_i,j}^{(i)} \in [\hat{u}(t_0),\hat{v}(t_0)] },
\end{align*}
let $\hat{g}_n(t_0) \equiv n_{\hat{u}, \hat{v}}(t_0) \big/ \{ n(\hat{v}(t_0)-\hat{u}(t_0)) \}$ and 
 \begin{align*}
\hat{\sigma}_n^2(t_0) \equiv \frac{1}{n_{\hat{u}, \hat{v}}(t_0)} \sum_{i=1}^n \sum_{j=1}^{K_i} \big(N_{K_i,j}^{(i)}-\hat{\Lambda}_n(t_0)\big)^2 \bm{1}_{T_{K_i,j}^{(i)} \in [\hat{u}(t_0),\hat{v}(t_0)]}.
 \end{align*}
 Consider the following CI for $\Lambda_0(t_0)$:
 \begin{align*}
\mathcal{I}_n^{\textrm{pan}}(c_\delta) 
&\equiv \textstyle  \Big[ \hat{\Lambda}_n(t_0)\pm c_\delta \cdot \sqrt{\hat{\sigma}_n^2(t_0)/\hat{g}_n(t_0)} \big/ \sqrt{n(\hat{v}(t_0)-\hat{u}(t_0))} \Big] \cap [0,\infty)
\\ \nonumber
&= \textstyle  \Big[ \hat{\Lambda}_n(t_0)\pm c_\delta \cdot \hat{\sigma}_n(t_0) \big/ \sqrt{n_{\hat{u}, \hat{v}}(t_0) } \Big] \cap [0,\infty).
 \end{align*}

\begin{theorem}\label{thm:CI_panel_count}
	Suppose $\Lambda_0$ is non-decreasing and locally $C^1$ with $\Lambda'_0(t_0)>0$ and further regularity conditions (as specified in Theorem 4.3 of \cite{wellner2000two}) hold. Let $c_\delta>0$ be a continuity point of the d.f. of $\abs{\mathbb{L}_{1}}$ such that $
	\Prob\big(\abs{\mathbb{L}_1}>c_\delta\big)=\delta$. Then 
	\begin{align*}
	\lim_{n \to \infty} \Prob_{\Lambda_0}\big(\Lambda_0(t_0) \in \mathcal{I}_n^{\textrm{pan}}(c_\delta)\big) = 1-\delta.
	\end{align*}
	Furthermore, for any $\epsilon>0$,
	\begin{align*}
	&\liminf_{n \to \infty} \Prob_{\Lambda_0} \bigg(\abs{\mathcal{I}_n^{\textrm{pan}}(c_\delta)} < 2c_\delta \mathfrak{g}_\epsilon\cdot  n^{-1/3}\big(\sigma^2(t_0)\Lambda_0'(t_0)/2g(t_0)\big)^{1/3} \bigg) \geq 1-\epsilon.
	\end{align*}
	Here $\mathfrak{g}_\epsilon\in (0,\infty)$ is such that $
	\Prob\big(\mathbb{S}_1^{-1}\geq \mathfrak{g}_\epsilon\big)\leq\epsilon$.
\end{theorem}

\begin{remark}
Similar to \cite{wellner2000two}, we do not assume that the Poisson model for the counting process $N$, used in building the pseudo likelihood for the definition $\hat{\Lambda}_n$, need to be true. 
\end{remark}

\subsection{Generalized linear models and the i.n.i.d. (independent, not identically distributed) case}
The likelihoods in (\ref{ineq:likelihood_interval_censoring}) and (\ref{ineq:likelihood_panel_count}) hint that a similar idea could be taken further to the generalized linear models as follows. Suppose real-valued random variables $Y_i$'s ($i=1,\ldots,n$) are independent samples with density $f(\cdot;\theta_{0,i})$ (with respect to a $\sigma$-finite measure $\nu$ on the real line) from a canonical exponential family:
\begin{align}\label{GLM}
f(y;\theta)= \exp\big(y \cdot p(\theta) - q(\theta)\big),\ \theta\in\Theta,
\end{align}
where $\Theta = \{\theta\in \R: \int e^{y\cdot p(\theta)}\nu(\d{y})<\infty\}$ 
and $\nu$ does not put all the mass at a single point $y_0$. We assume that $\theta_{0,i} \equiv \theta_0(x_i)$, where $\theta_0:[0,1]\to \R$ is monotonically non-decreasing. Let $\Theta_0$ be the interior of $\Theta$. 
In general we may assume that the natural parameter $p(\theta)$ 
is a continuously differentiable and strictly increasing function of $\theta$. 
However, as the mean function $\mu(\theta) = \int yf(y;\theta)\nu(\d{y}) = \partial q(\theta)/\partial p(\theta)$ 
is always continuously differentiable and strictly increasing in $p(\theta)$ in $\Theta_0$, 
we consider for simplicity the parametrization $\theta = \mu(\theta)$, as alternative parametrizations can be easily handled by applying the delta-method to our results. 
In this setting, the variance is given by $\int (y-\theta)^2 f(y;\theta)\nu(\d{y}) = 1/p'(\theta)$ in $\Theta_0$. 
We shall assume that the variance is finite at  $\theta=\theta_0(0)$ and $\theta=\theta_0(1)$ even when they are on the boundary of the domain $\Theta$, 
e.g., $p^\prime(\theta_0(0) )=p^\prime(\theta_0(1))=\infty$ when $Y_i\in \{0,1\}$. 

We are interested in estimating $\theta_0$ by the maximum likelihood estimator $\hat{\theta}_n$ that maximizes
\begin{align*}
\theta \mapsto \sum_{i=1}^n \bigg(Y_i \cdot p\big(\theta(x_i)\big) - q\big(\theta(x_i)\big)\bigg)
\end{align*}
over $\theta \in \R^n$ such that $\theta_1\leq \ldots\leq \theta_n$. As $p(\theta)$ and $q(\theta)$ are analytic in $\Theta_0$, by \cite[Theorem 1.5.2]{robertson1988order} 
the solution $\hat{\theta}_n \in \R^n$ is given by the isotonic regression of $(Y_i)_{i=1}^n$. For any $x \in [0,1]$, let
\begin{align}\label{GLM-MLE}
\hat{\theta}_n(x)\equiv \max_{i: x_i\leq x}\min_{j:x_j\geq x} \frac{\sum_{k=i}^j Y_k}{j-i+1} \equiv \max_{u\leq x}\min_{v\geq x}\bar{Y}|_{[u,v]}.
\end{align}
Then we may identify $\hat{\theta}_{n,i} = \hat{\theta}_n(x_i), i=1,\ldots,n$.

From here the analysis of the maximum likelihood isotonic regression 
in the generalized linear model reduces to a special case of the analysis of 
the max-min estimator in the {i.n.i.d.} case where 
\begin{align}\label{inid}
\hbox{$Y_i$ are independent with $\E[Y_i] = \theta_0(x_i)$ and Var$(Y_i) = \sigma^2(x_i)$} 
\end{align}
under a Lindeberg condition on $Y_i$ and a smoothness condition on $\theta_0(x_i)$. 

Formally, let $x_1\leq \ldots\leq x_n$ with 
$x_{i_0-1} \leq x_0\leq x_{i_0}$ for some $i_0\in  \{2,\ldots,n-1\} $, $\alpha > 0$ be fixed and $\omega_n \equiv n^{-\alpha/(2\alpha+1)}$. For $(h_1,h_2)\in \R_{\ge 0}^2$ define 
$S_{n,h_1,h_2} \equiv \{i: i_0 - n\omega_n^{1/\alpha}h_1\le i\le i_0 + n\omega_n^{1/\alpha}h_2\}$. 
We assume that for some $g_0(x_0)>0$ 
\begin{align}\label{inid-cond}
& \bigg|\sum_{i\in S_{n,h_1,h_2}} \frac{\theta_0(x_i)-\theta_0(x_0)}{\omega_n|S_{n,h_1,h_2}|}  
   - g_0(x_0)\frac{h_2^{\alpha+1}-h_1^{\alpha+1}}{h_1+h_2}\bigg| = \mathfrak{o}(1),
\cr & \bigg|\sum_{i\in S_{n,h_1,h_2}} \frac{\sigma^2(x_i)/\sigma^2(x_0)}{|S_{n,h_1,h_2}|} -1\bigg| = \mathfrak{o}(1), 
\\ \nonumber & \sum_{i\in S_{n,h_1,h_2}} \frac{\E\big[(Y_i-\theta_0(x_i))^2 
\bm{1}_{(Y_i-\theta_0(x_i))^2>c^{-1}\sigma^2(x_0) |S_{n,h_1,h_2}|}\big] }{\sigma^2(x_0) |S_{n,h_1,h_2}|} = \mathfrak{o}(1), 
\end{align}
uniformly in $(h_1,h_2)\in [1/c,c]^2$ for every $c>1$. Under these conditions, we have the following:
\begin{theorem}\label{thm:CI_GLM} 
\begin{enumerate}
	\item Suppose \eqref{inid} and \eqref{inid-cond} hold with a non-decreasing function $\theta_0$ 
	and $\max_{1\le i\le n}\sigma^2(x_i)=\mathcal{O}(1)$. 
	Let $\hat{\theta}_n$ be as in \eqref{GLM-MLE}. 
	Then
	\begin{align*}
	& \omega_n^{-1}\big(\hat{\theta}_n(x_0)-\theta_0(x_0)\big)\\
	&\rightsquigarrow \sup_{h_1>0}\inf_{h_2>0}\bigg[
	\sigma(x_0) \cdot \frac{\G(h_1,h_2)}{h_1+h_2}
	+ g_0(x_0)\frac{h_2^{\alpha+1}-h_1^{\alpha+1}}{h_1+h_2}
	\bigg]\\
	&=_d \big((\sigma(x_0))^{2\alpha}g_0(x_0)\big)^{1/(2\alpha+1)}\cdot\mathbb{D}_\alpha.
	\end{align*}
	\item Suppose $\alpha\beta$ is a positive odd integer for some $\beta>0$ and that $\theta_0(\cdot)$ 
	has $\alpha\beta-1$ vanishing derivatives and positive the $(\alpha\beta)$-th derivative at $x_0$.
	Let $\pi(\cdot)$ be a density such that $\pi(x) = (1+\mathfrak{o}(1))\pi_0\beta\cdot |x-x_0|^{\beta-1}$ uniformly in a neighborhood of $x_0$.  
	Then, the first line of \eqref{inid-cond} holds (in probability) when $x_1\le\cdots \le x_n$ are the ordered independent samples 
	from $\pi(\cdot)$. Moreover, in the generalized linear model \eqref{GLM}, 
	the second and third lines of \eqref{inid-cond} and the uniform 
	boundedness condition on the variance always hold. 
\end{enumerate}	
\end{theorem}

In addition to providing a general limit distribution theory in the i.n.i.d. case, 
the above theorem specifies sufficient conditions under which the fast convergence rate with 
$\alpha >1$ can be achieved when more $x_i$ are sampled near $x_0$ than the usual 
$x_i=i/n$, and vice versa.

In Theorem \ref{thm:CI_GLM}, the difficult nuisance parameter is $g_0(x_0)$, and the easier one is $\sigma^2(x_0)$. Let $(\hat{u}(x_0), \hat{v}(x_0))$ be any pair such that $\hat{\theta}_n(x_0)  =\bar{Y}|_{[\hat{u}(x_0), \hat{v}(x_0)]}$. Consider the following CI for $\theta_0(x_0)$:
\begin{align*}
\textstyle
\mathcal{I}_n^{\textrm{GLM}}(c_\delta)
\equiv \Big[\hat{\theta}_n(x_0) \pm c_\delta \cdot \widehat{\sigma}_n \varrho_n  \big/ \sqrt{ n(\hat{v}(x_0)-\hat{u}(x_0)) } \Big],
\end{align*}
where $\hat\sigma_n^2 \equiv 1/p^\prime(\hat{\theta}_n(x_0))$ under \eqref{GLM} 
or $\hat\sigma_n^2 \equiv \sigma_{\hat{u},\hat{v}}^2$ as in \eqref{def:var_est} in general, and $\varrho_n  =1+\mathfrak{o}_{\mathbf{P}}(1)$. If we choose $\varrho_n \equiv \sqrt{n(\hat{v}(x_0)-\hat{u}(x_0))\big/\sum_i \bm{1}_{x_i \in [\hat{u}(x_0),\hat{v}(x_0)]} }$, the CI above reduces to $\big[\hat{\theta}_n(x_0) \pm c_\delta \cdot \hat{\sigma}_n/ \sqrt{\sum_i \bm{1}_{x_i \in [\hat{u}(x_0),\hat{v}(x_0)]}  }\big]$ in the similar form to (\ref{def:CI}).

\begin{theorem}\label{thm:CI_GLM_1}
Assume the same conditions as in Theorem \ref{thm:CI_GLM} with $\alpha=1$. Let $c_\delta>0$ be a continuity point of the d.f. of $\abs{\mathbb{L}_{1}}$ such that $
\Prob\big(\abs{\mathbb{L}_1}>c_\delta\big)=\delta$. Then 
\begin{align*}
\lim_{n \to \infty} \Prob_{\theta_0}\big(\theta_0(x_0) \in \mathcal{I}_n^{\textrm{GLM}}(c_\delta)\big) = 1-\delta.
\end{align*}
Furthermore, for any $\epsilon>0$,
\begin{align*}
&\liminf_{n \to \infty} \Prob_{\theta_0} \bigg(\abs{\mathcal{I}_n^{\textrm{GLM}}(c_\delta)} < 2c_\delta \mathfrak{g}_\epsilon\cdot  n^{-1/3}\big(\sigma^2(x_0)g_0(x_0)\big)^{1/3} \bigg) \geq 1-\epsilon.
\end{align*}
Here $\mathfrak{g}_\epsilon\in (0,\infty)$ is such that $
\Prob\big(\mathbb{S}_1^{-1}\geq \mathfrak{g}_\epsilon\big)\leq\epsilon$.
\end{theorem}

The above theorem covers only the usual case of $\alpha=1$ for the cube-root rate. The critical value $c_\delta$ for general $\alpha$ can be handled as in Theorem \ref{thm:adaptive_confidence_interval}. 

\section{Simulation studies}\label{section:simulation}

\subsection{Critical values $c_\delta$ via simulations}

In this subsection, we discuss (i) simulation methods to approximate critical values $c_{\delta}$ of the pivotal limit distribution theory (with local smoothness $\bm{\alpha}=(1,\ldots,1)$ as in Theorem \ref{thm:CI_exact}), and (ii) data-driven adjustments to edge effect and small sample size.

\subsubsection{Simulated critical values $c_\delta$}

We use the following method to simulate critical values $c_\delta$:

\begin{enumerate}
	\item Specify a true mean function $f_0$ defined on a fixed lattice $\{X_i\}$ in $[0,1]^d$, and generate $B=10^6$ repeated observations $\{ \{Y_{i,b} = f_0(X_i)+\xi_{i,b}, i=1,\ldots,n\}: b =1,\ldots,B\}$ with i.i.d. $\xi_{i,b} \sim \mathcal{N}(0,\sigma^2)$ (we take $\sigma=1$ for this simulation).
	\item At design point $x_0$, we obtain $T(x_0;b)\equiv\{ \sqrt{n_{\hat{u},\hat{v}}(x_0;b) }\bigabs{\hat{f}_n(x_0;b)-f_0(x_0)}/\sigma: b=1,\ldots,B\}$, where $\hat{f}_n(x_0;b)$ and $n_{\hat{u},\hat{v}}(x_0;b)$ are calculated via (\ref{def:block_avg}) and (\ref{def:n_uv}). We note that the block $[\hat{u}(x_0), \hat{v}(x_0)]$ is specified by $\hat{u}(x_0)$ from the block max-min estimator and $\hat{v}(x_0)$ from the block min-max estimator. 
	\item Find critical values $c_\delta$ by the corresponding quantiles of $\{T(x_0;b):  b=1,\ldots,B \}$. 
\end{enumerate}

\noindent \textbf{($\bm{d=1}$).} 
The simulated critical values $c_\delta$ for $d=1$ are summarized in Table \ref{tab:cv-1} below. As the block max-min and min-max estimators (\ref{def:max_min_estimator}) are equivalent to the isotonic least squares estimator (LSE) at design points in $d=1$, we use \verb|isoreg| (based on the PAVA algorithm (\cite{robertson1988order,barlow1972statistical})) built in \verb|R|. Let $n=10^5$, so that $X_i = i/n$ for all $1 \le i \le n$ and $x_0 = 0.5$.

\setlength{\belowcaptionskip}{-10pt}
\setlength{\tabcolsep}{5pt}
\renewcommand{\arraystretch}{1.2}
\begin{table}[htb]
	\begin{tabular}{|c||c|c|c|c|c|c|c|}
			\hline 
			$\delta$ & .01  & .02 & .05 & .1 & .15 &.2 & .5\\
			\hline
			$f_0(x) = 2(x-0.5)$ & 3.04 & 2.65 & 2.11 & 1.68 & 1.42 & 1.23 & 0.59 \\
			\hline
			$f_0(x)=5(x -0.5)$ & 3.04 & 2.65 & 2.11 & 1.68 & 1.42 & 1.23 & 0.59 \\
			\hline
			$f_0(x) =10x^2$ & 3.04 & 2.66 & 2.11 & 1.68 & 1.42 & 1.23 & 0.59\\
			\hline
		\end{tabular}
	\vspace{1ex}
	\caption{Simulated critical values $c_\delta$ for $d=1$.}
	\label{tab:cv-1}
\end{table}

As the sample size ($n=10^5$) for $d=1$ is quite large, the estimates for different $f_0$'s remain the same at least up to two decimal places, with an exception for $c_{0.02}$. Such precision is typically sufficient for the purpose of inference.

\medskip
\noindent \textbf{($\bm{d\geq 2}$).} For multiple isotonic regression, we use brute force to compute the block max-min and min-max estimators, which seems to be the only algorithm readily available. With computational complexity $\mathcal{O}(n^3)$, the brute force algorithm is much more expensive than the linear-time PAVA algorithm specific to the univariate isotonic LSE. Thus, it is not computationally feasible to perform $B=10^6$ simulations on, 
say, a $10^5 \times 10^5$ (i.e., $n=10^{10}$) lattice. 

Nevertheless, we present below in Table \ref{tab:cv-2} ($d=2$) and Table \ref{tab:cv-3} ($d=3$) some encouraging simulation results by brute force computation over relatively small lattices: 
\begin{itemize}
	\item For $d=2$, we use a $50 \times 50$ lattice and take sample critical values at $x_0 = (0.5,0.5)$.
	\item  For $d=3$, we use a $16 \times 16 \times 16$ lattice and take sample critical values at $x_0 = (0.5,0.5,0.5)$.
\end{itemize}

\setlength{\belowcaptionskip}{-10pt}
\setlength{\tabcolsep}{5pt}
\renewcommand{\arraystretch}{1.2}
\begin{table}[hbt]
	\begin{tabular}{|c||c|c|c|c|c|c|c|}
			\hline 
			$\delta$ & .01  & .02 & .05 & .10 & .15 &.20 & .50 \\
			\hline\hline
			$f_0(x) = 2x_1 + 2x_2-2$ & 2.61 & 2.26 & 1.78 & 1.41 & 1.19 & 1.03 & 0.49 \\
			\hline
			$f_0(x) = 2x_1 + 5x_2-3.5$ & 2.63 & 2.27 & 1.80 & 1.43 & 1.21 & 1.04 & 0.50 \\
			\hline
			$f_0(x) = 5x_1 + 5x_2-5$ & 2.64 & 2.29 & 1.81 & 1.43 & 1.21 & 1.05 & 0.50 \\
			\hline
		\end{tabular}
	\vspace{0.5ex}
	\caption{Simulated critical values $c_\delta$ for $d=2$.}
	\label{tab:cv-2}
	% \vspace*{-5mm}
\end{table}

\setlength{\belowcaptionskip}{-10pt}
\setlength{\tabcolsep}{5pt}
\renewcommand{\arraystretch}{1.2}
\begin{table}[hbt]
	\begin{tabular}{|c||c|c|c|c|c|c|c|}
			\hline 
			$\delta$ & .01  & .02 & .05 & .10 & .15 &.20 & .50 \\
			\hline\hline
			$f_0(x) = 2x_1 + 2x_2 + 2x_3 - 3$ & 2.26 & 1.96 & 1.55 & 1.24 & 1.05 & 0.91 & 0.44 \\
			\hline
			$f_0(x) = 2x_1 + 5x_2 + 5x_3 - 6$ & 2.41 & 2.09 & 1.66 & 1.33 & 1.13 & 0.98 & 0.48 \\
			\hline
			$f_0(x) = 5x_1 + 5x_2 + 5x_3 - 7.5$ & 2.41 & 2.10 & 1.67 & 1.34 & 1.14 & 0.99 & 0.49 \\
			\hline
		\end{tabular}
	\vspace{0.5ex}
	\caption{Simulated critical values $c_\delta$ for $d=3$.}
	\label{tab:cv-3}
	% \vspace*{-5mm}
\end{table}

The simulated critical values $c_\delta$ in $d=2,3$, albeit of small sample size in each dimension, already support the pivotal limit distribution theory. Their concrete numeric values are, however, less stable compared with $d=1$ for different mean functions $f_0$, largely due to the curse of dimensionality that requires much larger sample size $n$ to achieve similar accuracy as in $d=1$.  Unfortunately, brute force seems not ideal for this task. It is therefore of great interest to develop fast algorithms for the block max-min and min-max estimators (\ref{def:max_min_estimator}) in view of their theoretically attractive properties.

By taking average, we give a few suggested critical values as follows.

\begin{table}[htb]
	\begin{tabular}{|c||c|c|c|}
			\hline 
			$\delta$ & $d=1$ & $d=2$ & $d=3$ \\
			\hline
			$0.05$ & $2.11$ & $1.80^*$ & $1.63^{*}$ \\
			\hline
			$0.10$ & $1.68$ & $1.42^*$ & $1.30^{*}$ \\
			\hline
		\end{tabular}
	\vspace{0.5ex}
	\caption{Suggested critical values $c_\delta$ ( $^*$: use with caution).} 
	\label{tab:cv-final}
	% \vspace*{-7mm}
\end{table}

\subsubsection{Data-driven adjustments} 

As the pivotal limit distribution theory relies on local smoothness of $f_0$ and a large sample size, it is not surprising that for outskirt design points or when the sample size is relatively small, CIs constructed via (\ref{def:CI}) with the critical values suggested in Table \ref{tab:cv-final} would be less accurate. See for instance the plots given below in subsection 4.2 for demonstration. This is particularly relevant for $d\geq 2$, since a lot more points are present on the outskirts and in practice the sample size in each dimension is usually not as large as in the univariate case. These issues call for critical value adjustments to improve accuracy in inference. 

Our proposal is to adjust the critical values based on the observed $\{Y_i\}$ and the sample size. More specifically, we propose the use of critical values simulated through a smooth proxy $\hat{f}_{smooth}$ of the block average estimator $\hat{f}_n$. In order to match the noise level of $\{Y_i\}$, the variance $\sigma^2$ in simulation can be chosen to be the variance estimate $\hat{\sigma}^2$ of $\{Y_i\}$. A simple smoothing method to get $\hat{f}_{smooth}$ is the isotonization of the \texttt{LOESS} fit $\hat{f}_{loess}$ (with default smoothing parameter built in \texttt{R}) of $\hat{f}_n$, i.e., $\hat{f}_{smooth}=$ the block average estimate for $\{\hat{f}_{loess}(X_i)\}$. 
See \cite{mammen2001general} for more details on constrained smoothing. 
As $\hat{f}_{smooth}$ is expected to be `close' to the true mean function of interest, it is reasonable to expect the simulated critical values for $\hat{f}_{smooth}$ to better mimic those for $f_0$ for design points on the outskirts and when the sample size is relatively small. We call the critical values simulated from $\hat{f}_{smooth}$ the {\it adjusted critical values}.

As will be clear from the next subsection, the {\it adjusted critical values} improve the inference accuracy both for design points on the outskirts and for smaller samples.

\subsection{Numerical performance of the proposed confidence intervals}

In this subsection, we investigate the numerical performance of the proposed CIs, exclusively in the multiple isotonic regression model. More specifically, we construct CIs for $f_0(x)$ at each design point $x \in \{X_i\}$ and compute their corresponding coverage probabilities as follows:

\begin{enumerate}
	\item For each specified mean function $f_0$, generate $B=10^4$ repeated observations $\{ \{Y_{i,b} = f_0(X_i)+\xi_{i,b}, i=1,\ldots,n\}: b =1,\ldots,B\}$ with i.i.d. $\xi_{i,b} \sim \mathcal{N}(0,\sigma^2)$.
	\item For each $b=1,\ldots,B$, construct the CI $\mathcal{I}_n(x;c_\delta,b)$ for each design point $x$ via (\ref{def:CI}). The CIs with the suggested $c_{\delta}$ in Table \ref{tab:cv-final} are referred to as \emph{vanilla CIs}, and the CIs with adjusted $c_{\delta}$ (as described in the proceeding subsection) as \emph{CV-adjusted CIs}. 
	\item Report $B^{-1}\sum_{b=1}^B \bm{1}\big(f_0(x) \in \mathcal{I}_n(x;c_\delta,b) \big)$ as the 
	coverage probability at design point $x$, i.e., the proportion of the CIs $\{\mathcal{I}_n(x;c_\delta,b): b=1,\ldots,B\}$ that successfully cover the truth $f_0(x)$ out of $B=10^4$ repeated observations. We focus on $95\%$ CIs, i.e., $\delta = 0.05$.
\end{enumerate}

For variance estimation, we use the class of difference estimators (\cite{rice1984bandwidth,hall1991estimation,munk2005difference}) rather than the principled estimator in (\ref{def:var_est}), as the latter requires large samples that are computationally expensive for $d\geq 2$. Specifically, we use the following variance estimator $\hat{\sigma}^2$:
\begin{align}\label{def:diff_est}
\hat{\sigma}^2 = 
\begin{cases} 
\sum_{i} \big( 2Y_i - Y_{i-1} - Y_{i+1} \big)^2/(6(n-2)), & d=1, \\
\sum_{i,j} \big( 4 Y_{i,j} - Y_{i-1, j} - Y_{i+1, j}  - Y_{i, j-1} - Y_{i, j+1} \big)^2/\big(20 \times 
\cr
\qquad \qquad  (n_1-2)(n_2-2)\big), & d=2, \\
\sum_{i,j,k} \big( 6 Y_{i,j,k} - Y_{i-1, j,k} - Y_{i+1, j,k}  - Y_{i, j-1,k} - Y_{i, j+1,k}
\cr
\qquad  - Y_{i, j,k-1} - Y_{i, j,k+1} \big)^2/\big(42(n_1-2)(n_2-2)(n_3-2)\big), & d=3,
\end{cases}
\end{align}
where, with slight abuse of notation, the observations are $(Y_i)_{1 \le i \le n}$ for $d=1$, $(Y_{i,j})_{1\le i\le n_1, 1\le j \le n_2}$ for $d=2$, and $(Y_{i,j,k})_{1\le i \le n_1, 1\le j \le n_2, 1\le k \le n_3}$ for $d=3$.

% \subsubsection{$(\bm{d=1})$}
\subsubsection{Coverage probability}

\begin{figure}[hbt]
	\centering
	\subfigure[$\sigma^2 =1$ is known]{
		\label{fig:1d-known} %% label for first subfigure
		\includegraphics[width=0.48\textwidth]{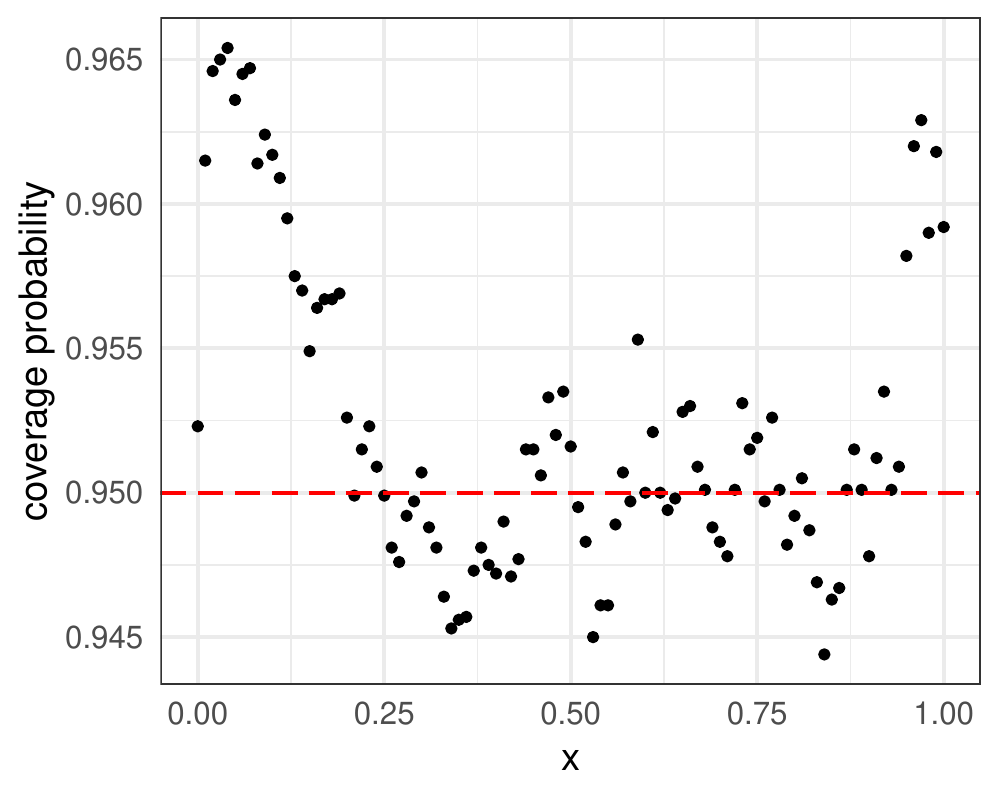}}
	\hspace{0\textwidth}
	\subfigure[$\sigma^2=1$ is unknown]{
		\label{fig:1d-unknown} %% label for second subfigure
		\includegraphics[width=0.48\textwidth]{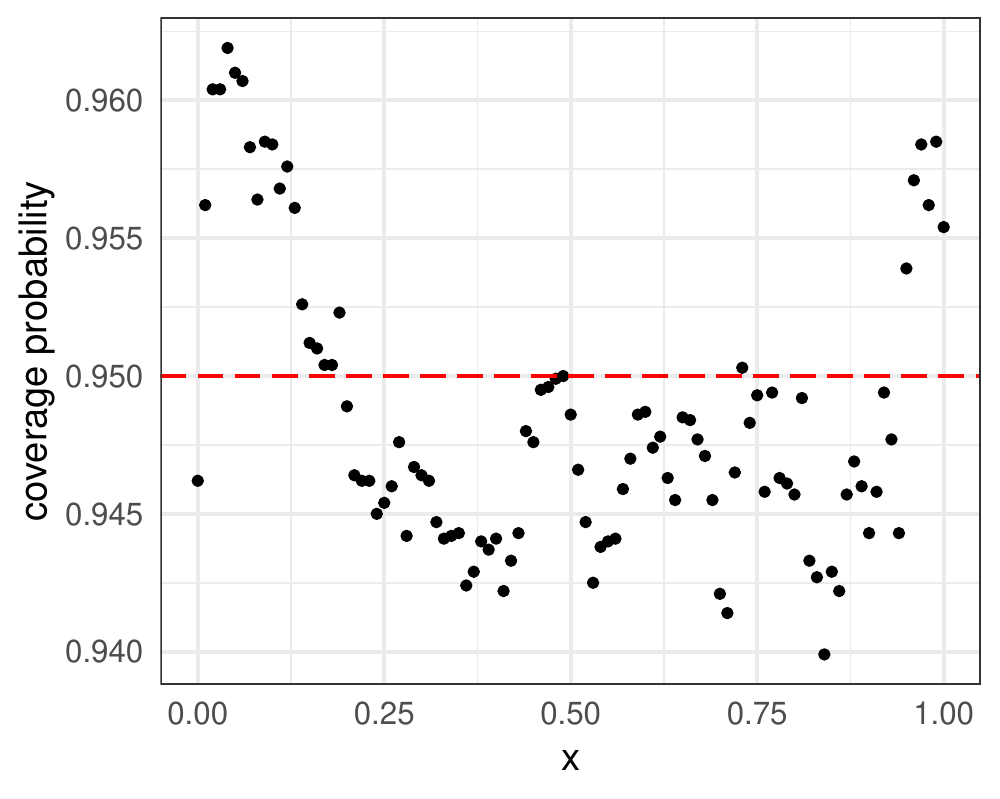}}
	\caption{Scatter plots for the coverage probabilities of the 95\% CIs in $d=1$, where $f_0(x) = e^{2x}$ and $n=100$.}
	\label{fig:1d}
\end{figure}

The scatter plots of coverage probabilities at all design points in $d=1$ are given in Figure \ref{fig:1d}. 
In $d=1$, we consider $f_0(x) = e^{2x}$ and $n=100$, so that $x \in \{0.01, 0.02, \ldots, 1.00\}$. 
We observe slightly larger errors of the coverage probabilities of the vanilla CIs at points near $x=0$ or $x=1$, but overall the coverage errors are small for the small sample size $n=100$. When $\sigma^2$ is unknown and estimated by the difference estimator (\ref{def:diff_est}), the errors are slightly inflated at most of the design points, but are still controlled within $1\%$. The CIs are overall biased slightly towards under-coverage, which is possibly due to the bias in variance estimation: 
the median of $\hat{\sigma}^2$'s from $10^4$ simulations is $0.9803$, slightly smaller than the true $\sigma^2=1$.

% \subsubsection{$(\bm{d=2,3})$}
The scatter plots of coverage probabilities in $d=2$ are given in Figure \ref{fig:2d}. We consider test function $f_0(x) = e^{x_1 + x_2}$ on a $25 \times 25$ lattice on $[0,1]^2$ so that $(x_1, x_2) \in \{(i/25, j/25): i=1, \ldots, 25, j=1, \ldots, 25\}$. We call design points in the inner $17 \times 17$ lattice (i.e., $\{x: 5/25 \le x_1 \le 21/25, 5/25 \le x_2 \le 21/25\}$) inner points, and the rest outskirt points. 
We can clearly identify edge effect in Figure \ref{fig:2d} when using the approximated universal critical value $1.80$ in Table \ref{tab:cv-final}; the coverage probabilities at outskirt points are more biased as shown 
in Figures \ref{fig:2d-known-vanilla} and \ref{fig:2d-unknown-vanilla}.  
The CV-adjusted CIs significantly reduce the edge effect and improve the coverage accuracy, 
as shown in Figures \ref{fig:2d-known-cvadj} and \ref{fig:2d-unknown-cvadj}. Figure \ref{fig:2d-bp} shows the boxplots for the coverage probabilities at all design points using the aforementioned two types of CIs (vanilla and CV-adjusted) and under both known and unknown $\sigma^2$. The CV-adjusted CIs clearly have more accurate coverage.

\begin{figure}[hbt]
	\centering
	\subfigure[Vanilla CIs, $\sigma^2 =1$ is known]{
		\label{fig:2d-known-vanilla} %% label for first subfigure
		\includegraphics[width=0.48\textwidth]{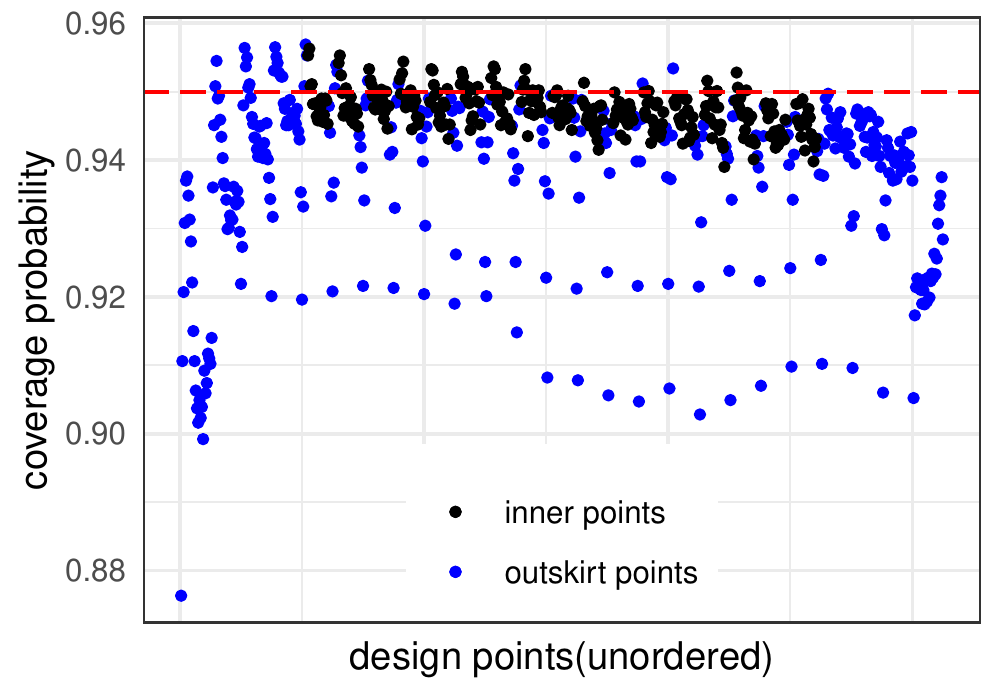}}
	\hspace{0\textwidth}
	\subfigure[CV-adjusted CIs, $\sigma^2=1$ is known]{
		\label{fig:2d-known-cvadj} %% label for second subfigure
		\includegraphics[width=0.48\textwidth]{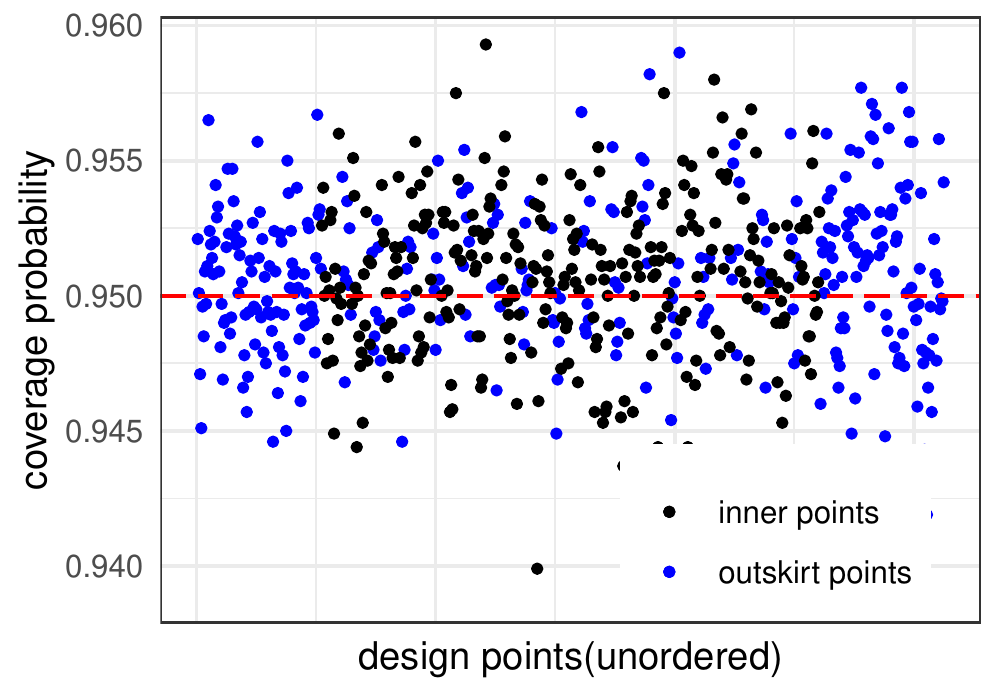}}
	\hspace{0\textwidth}
	\subfigure[Vanilla CIs, $\sigma^2 =1$ is unknown]{
		\label{fig:2d-unknown-vanilla} %% label for first subfigure
		\includegraphics[width=0.48\textwidth]{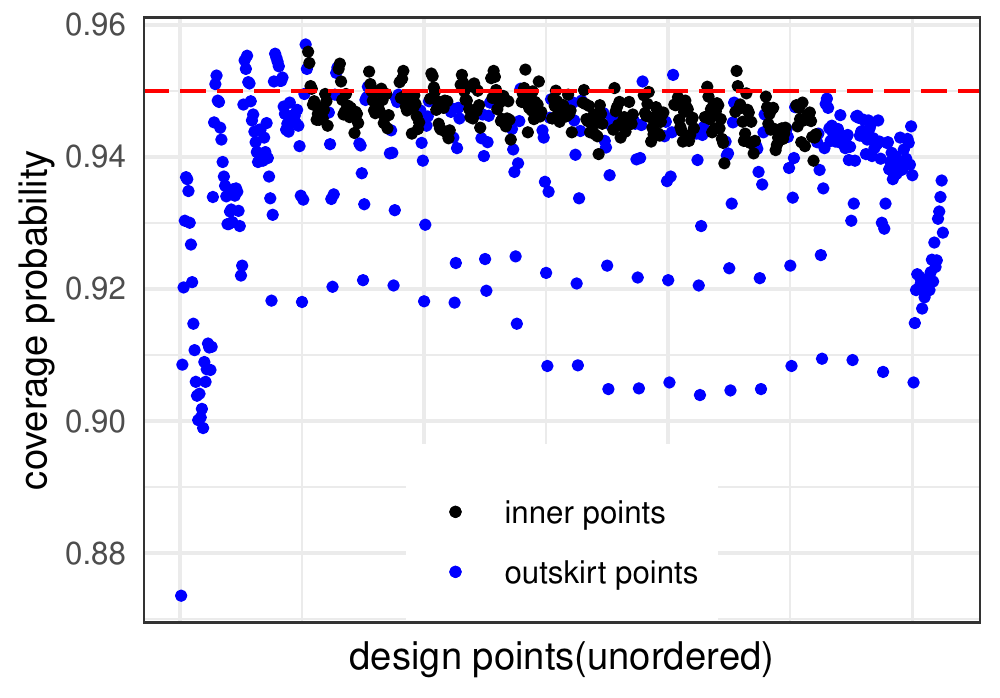}}
	\hspace{0\textwidth}
	\subfigure[CV-adjusted CIs, $\sigma^2=1$ is unknown]{
		\label{fig:2d-unknown-cvadj} %% label for second subfigure
		\includegraphics[width=0.48\textwidth]{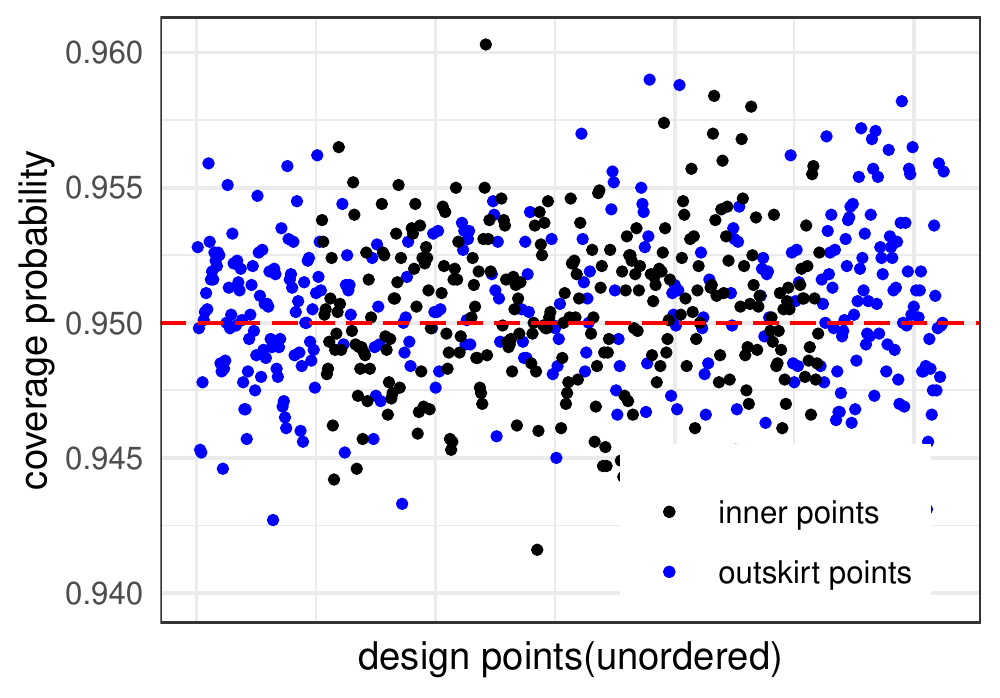}}
	\caption{Scatter plots for the coverage probabilities of the 95\% CIs in $d=2$, where $f_0(x) = e^{x_1 + x_2}$ in $25 \times 25$ lattice.}
	\label{fig:2d}
\end{figure}

\begin{figure}[hbt]
	\centering
	\includegraphics[width=0.85\textwidth]{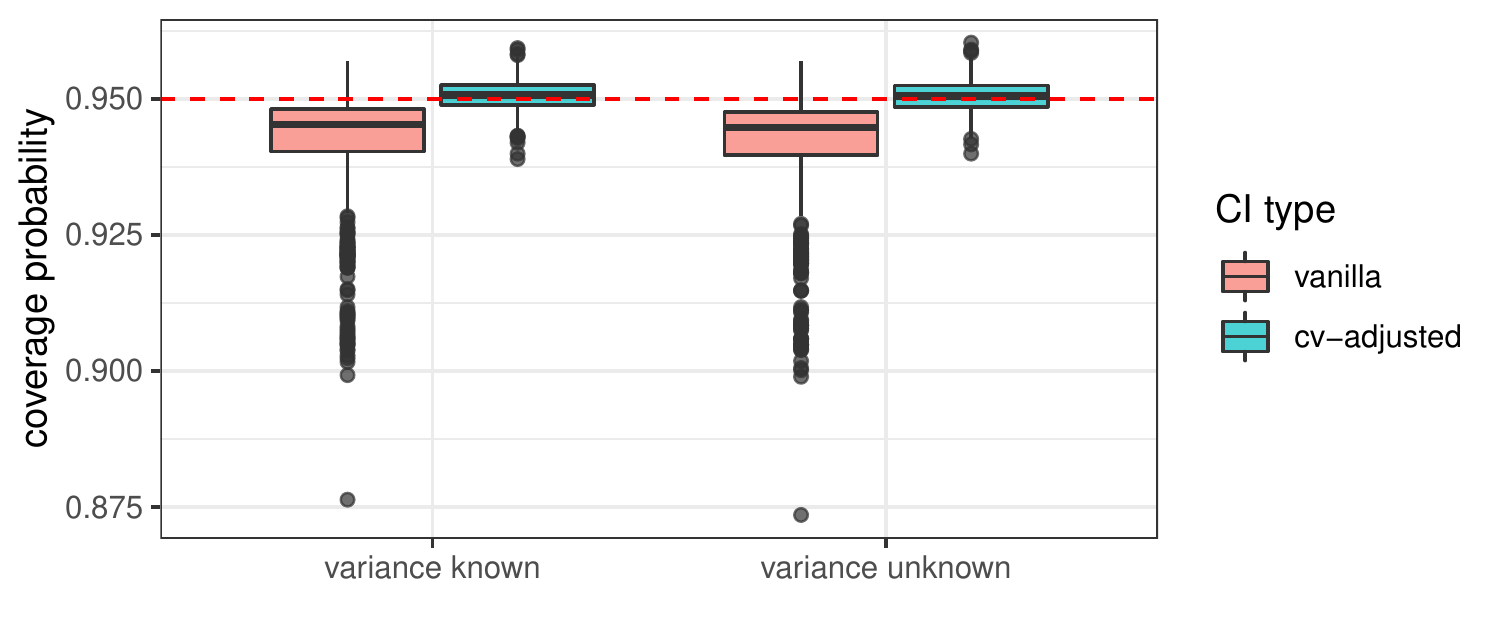}
	\caption{ Boxplots for the coverage probabilities of the 95\% CIs at all design points in $d=2$ under known and unknown variance $\sigma^2 =1$, 
		where $f_0(x) = e^{x_1 + x_2}$ on a $25\times 25$ lattice. 
		} 
	\label{fig:2d-bp}
\end{figure}

Our simulation results for $d=3$ exhibit similar phenomena to the case $d=2$. See the scatter plots in Figure \ref{fig:3d} and the boxplots in Figure \ref{fig:3d-bp}, where we designate the design points in the inner $5 \times 5 \times 5$ lattice as inner points and the rest outskirt points. For the vanilla CIs, we use the approximated universal critical value $1.63$ in Table \ref{tab:cv-final}. The lattice of size $9 \times 9 \times 9$ has too few points in each dimension, so distributional approximation to the pivot limit is less accurate, as shown in Figures \ref{fig:3d-known-vanilla} and \ref{fig:3d-unknown-vanilla}. 
On the other hand, the CV-adjusted CIs yield much better empirical results, 
as shown in Figures \ref{fig:3d-known-cvadj} and \ref{fig:3d-unknown-cvadj}.

\begin{figure}[hbt]
	\centering
	\subfigure[Vanilla CIs, $\sigma^2 =1$ is known]{
		\label{fig:3d-known-vanilla} %% label for first subfigure
		\includegraphics[width=0.48\textwidth]{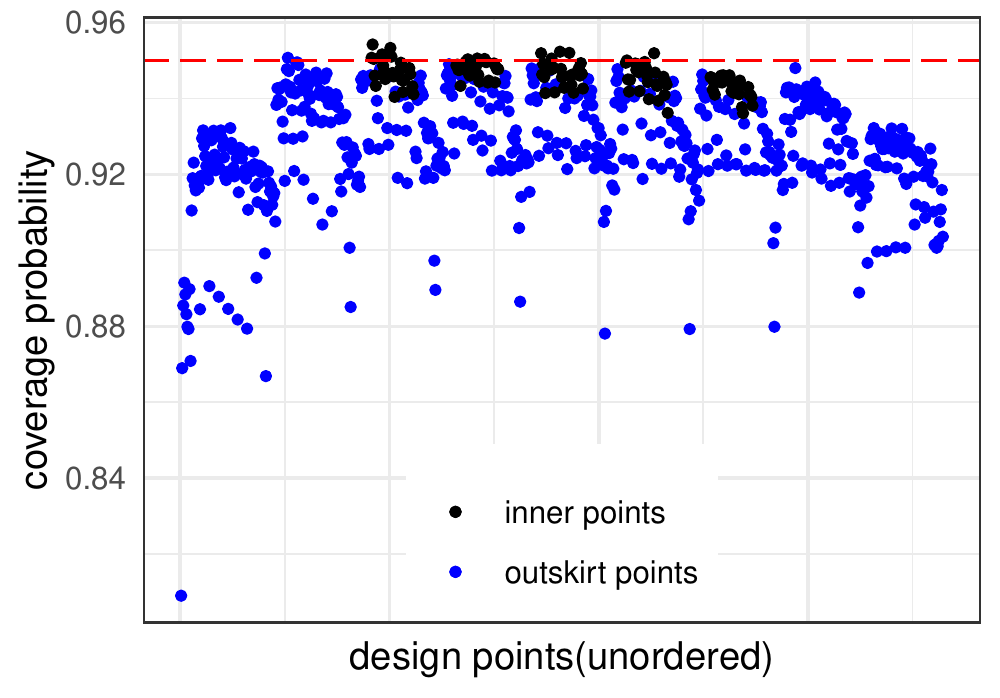}}
	\hspace{0\textwidth}
	\subfigure[CV-adjusted CIs, $\sigma^2 =1$ is known]{
		\label{fig:3d-known-cvadj} %% label for second subfigure
		\includegraphics[width=0.48\textwidth]{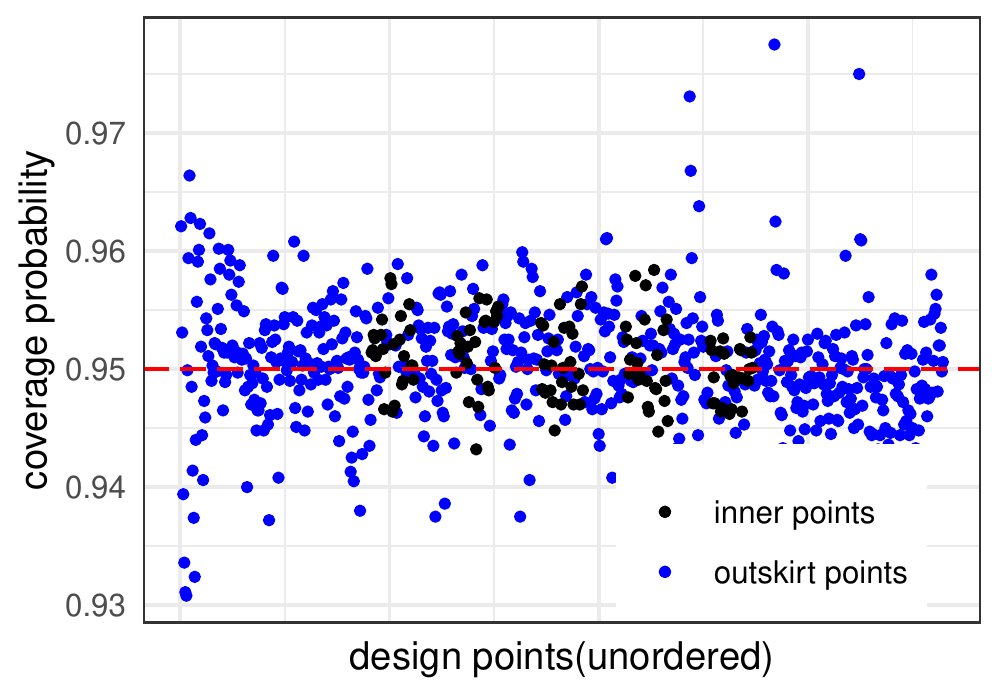}}
	\hspace{0\textwidth}
	\subfigure[Vanilla CIs, $\sigma^2 =1$ is unknown]{
		\label{fig:3d-unknown-vanilla} %% label for first subfigure
		\includegraphics[width=0.48\textwidth]{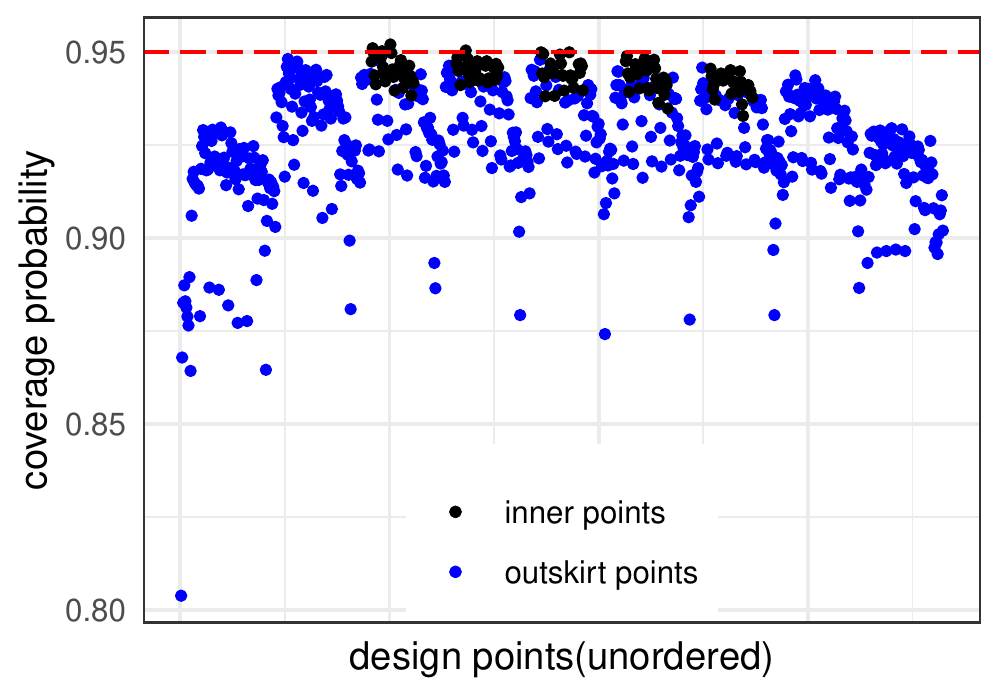}}
	\hspace{0\textwidth}
	\subfigure[CV-adjusted CIs, $\sigma^2 =1$ is unknown]{
		\label{fig:3d-unknown-cvadj} %% label for second subfigure
		\includegraphics[width=0.48\textwidth]{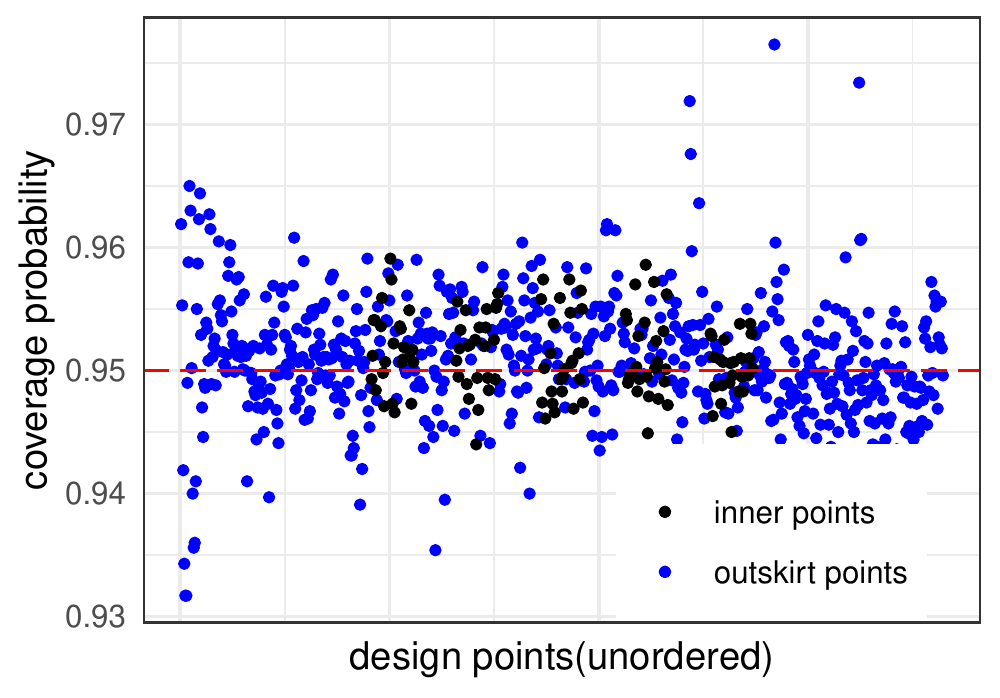}}
	
	\caption{Scatter plots for the coverage probabilities of the 95\% CIs in $d=3$, where $f_0(x) = e^{2(x_1+x_2 + x_3)/3}$ on a $9 \times 9 \times 9$ lattice.}
	\label{fig:3d}
\end{figure}

\begin{figure}[hbt]
	\centering
	\includegraphics[width=0.85\textwidth]{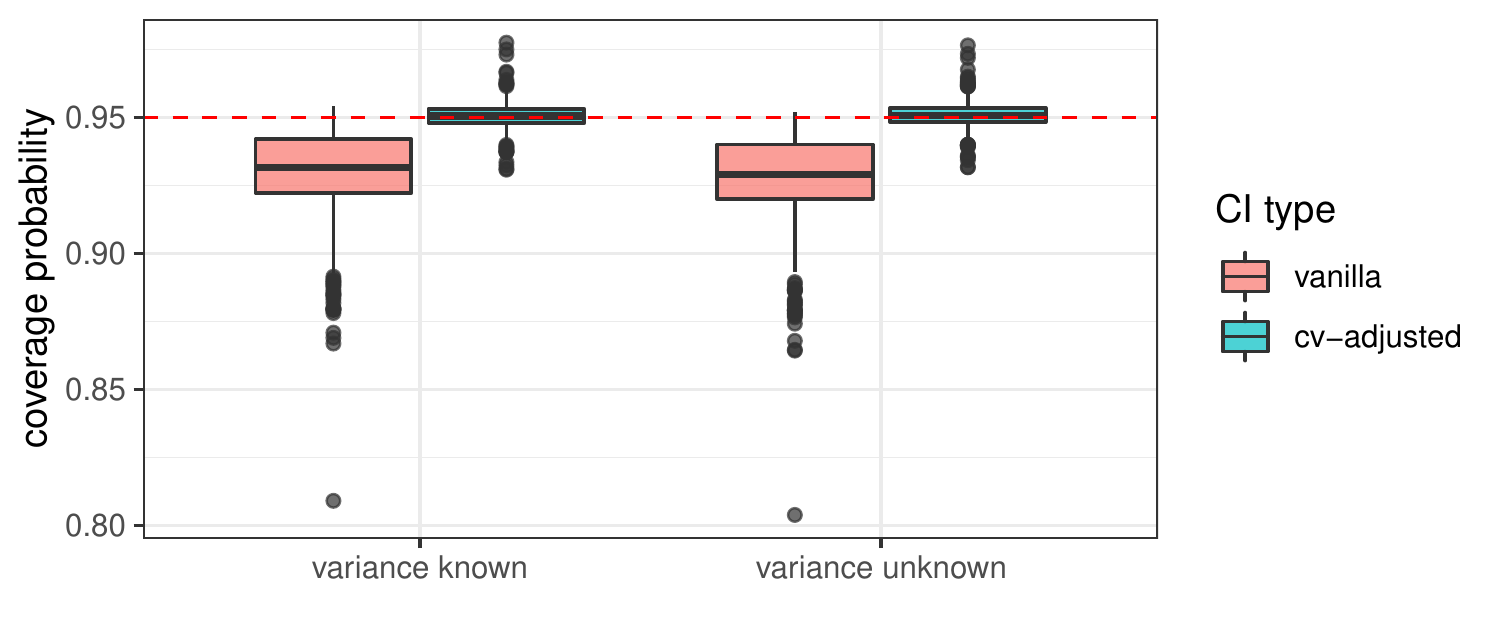}
	\caption{
	Boxplots for the coverage probabilities of the 95\% CIs at all design points in $d=3$ under known and unknown variance $\sigma^2 =1$, where $f_0(x) = e^{2(x_1+x_2 + x_3)/3}$ on a $9 \times 9 \times 9$ lattice.
		} 
	\label{fig:3d-bp}
\end{figure}

In conclusion, the above simulations support our proposed inference procedure via the pivotal limit distribution theory. In situations when sample size (in terms of each dimension) is relatively small, or the design points are on the outskirts, CV-adjusted CIs are shown to significantly improve the inference accuracy compared with the vanilla CIs that use universal critical values in Table \ref{tab:cv-final}.

\subsubsection{CI lengths} 
Another important theoretical property stated in Theorem \ref{thm:CI_exact} is that the proposed CI \eqref{def:CI} shrinks at the oracle rate. Suppose we know the partial derivatives $\{\partial_k f_0(x_0), 1\le k \le d\}$, the limit distribution theory in Corollary \ref{coro:simple} implies an oracle CI 
\begin{align}\label{def:oracle_CI}
\textstyle \Big[ \hat{f}_n(x_0) \pm c_{\delta} \cdot (n/\sigma^2)^{-1/(d+2)} \big\{ \prod_{k=1}^d (\partial_k f_0(x_0)/2)\big\}^{1/(2+d)} \Big],
\end{align}
where $c_{\delta}$ is the $1 - \delta$ quantile of $|\mathbb{D}| = |(\mathbb{D}_{\bm{1}_d}^- + \mathbb{D}_{\bm{1}_d}^+)/2|$ which can be similarly simulated as in Section 4.1.1. Recall $\mathbb{D}_1^{\mp}/2$ follows Chernoff distribution, so $c_{.05} = 1.9964$ in $d=1$; see e.g., \cite[Table 3.1]{groeneboom2014nonparametric}. Our simulation suggests that $c_{.05}$ is approximately $1.85$ in $d=2$ and $1.78$ in $d=3$. Then, Theorem \ref{thm:CI_exact} \eqref{eqn:length} asserts that the length of the proposed CI should shrink at the same rate as the length of the oracle CI in \eqref{def:oracle_CI}.

To see this in finite samples, we carry out a simulation that follows the same procedure as before but with varying sample sizes $n$. Only balanced fixed lattice design is considered in this simulation, so sample size $n$ indicates an $n^{1/d} \times \cdots \times n^{1/d}$ lattice. See Figure \ref{fig:length-bp} for the boxplots for the lengths of the proposed CI based on $10^4$ simulations, where the lengths of the oracle CIs in \eqref{def:oracle_CI} are given in red dashed lines. It clearly shows that the proposed CI indeed shrinks at the oracle length.

\begin{figure}[hbt]
	\centering
	\includegraphics[width=\textwidth]{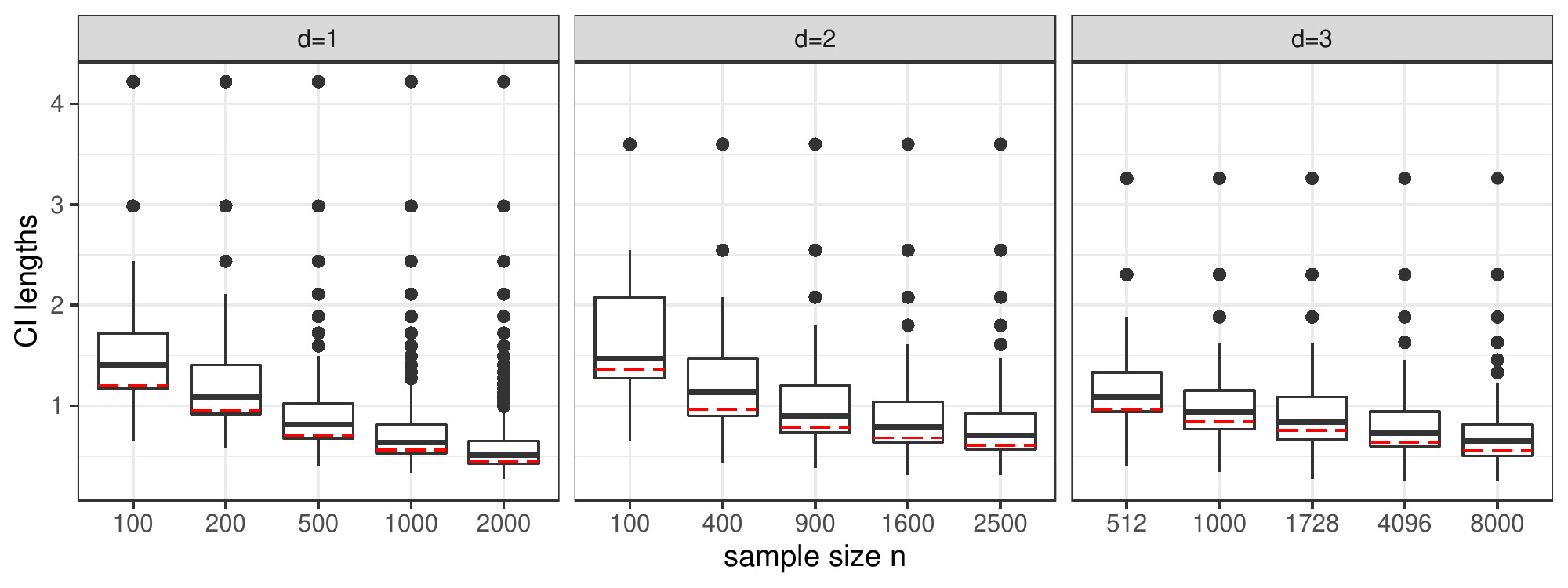}
	\caption{Boxplots for the lengths of the 95\% CIs of: (i) $f_0(x) = e^x$ at $x_0 = 0.5$ in $d=1$, (ii) $f_0(x) = e^{x_1+x_2}$ at $x_0 = (0.5,0.5)$ in $d=2$ and (iii) $f_0(x) = e^{2(x_1+x_2+x_3)/3}$ at $x_0 = (0.5,0.5,0.5)$ in $d=3$ under different sample sizes $n$. Red dashed lines represent the lengths of the oracle CIs. Here $\sigma=1$ is known.} 
	\label{fig:length-bp}
\end{figure}

\subsection{The merit of using the block average estimator}
\label{subsec:blocl-average}
In \eqref{def:CI} and \eqref{eqn:pivotal_limit_distribution}, we propose to use the block average 
estimator \eqref{def:block_avg} to carry out statistical inference about $f_0(x_0)$.
It is of natural interest to ask if using either the block max-min or min-max estimator alone in the proposed procedure is adequate as this would almost reduce the computational cost by half. In this subsection, we will show empirically the benefits of using \eqref{def:block_avg}. Specifically, while 
of the block average estimator improves upon the block max-min estimators only slightly in 
the mean squared error (numerical results omitted),
the proposed CI \eqref{def:CI} which uses the block average estimator 
outperforms the CIs that uses the block max-min estimator alone 
by a fairly considerable amount in terms of accuracy for the coverage of the CIs 
(the situation for block min-max estimator is analogous). 

More precisely, for inference based on the block max-min estimator alone,
we consider the following form of the CIs:
\begin{align}\label{def:CI_max_min}
\textstyle \mathcal{I}_n^{-}(c_\delta^{-})\equiv \Big[\hat{f}_n^{-}(x_0) \pm c_\delta^{-}\cdot \hat{\sigma} \big/ \textstyle \sqrt{n_{\hat{u},\hat{v}}^{-}(x_0)} \, \Big],
\end{align}
where $n_{\hat{u},\hat{v}}^{-}(x_0) \equiv \sum_{i} \bm{1}_{X_i \in [\hat{u}^{-}(x_0), \hat{v}^{-}(x_0)]}$ with $\hat{u}^{-}(x_0),\hat{v}^{-}(x_0)$ defined as any pair such that
\begin{align*}
\hat{f}_n^{-}(x_0) 
&\equiv \max_{u\leq x_0} \min_{ \substack{v\geq x_0\\ [u,v]\cap \{X_i\}\neq \emptyset}} \bar{Y}|_{[u,v]} 
=   \bar{Y}|_{[\hat{u}^{-}(x_0),\hat{v}^{-}(x_0)]}.
\end{align*}
We conjecture that
\begin{conjecture}\label{conj:pivotal_max_min}
Under the same settings as in Theorem \ref{thm:pivotal_limit_distribution}, for some finite random variable $\mathbb{L}_{\bm{\alpha}}^{-}$ (that does not depend on $K(f_0,x_0)$),
\begin{align*}
\sqrt{n_{\hat{u},\hat{v}}^{-}(x_0)}\big(\hat{f}_n^{-}(x_0)-f_0(x_0)\big)\rightsquigarrow \sigma\cdot \mathbb{L}_{\bm{\alpha}}^{-}.
\end{align*}
\end{conjecture}
Note that in $d=1$ the block max-min and min-max estimators are equivalent, so the above conjecture reduces to Theorem \ref{thm:pivotal_limit_distribution}. Below we consider $d=2,3$, and provide some numerical evidence that the CIs using the block average estimator could provide better probability coverage than using only the block max-min estimator based on Conjecture \ref{conj:pivotal_max_min}.

The summary of statistics for the coverage probabilities of the 95\% CIs in $d=2$ is listed in Table \ref{tab:2d-average-max-min}. The mean squared errors of the coverage probabilities of the 95\% CIs by the block average estimator are $1.67 \times 10^{-4}$ under known $\sigma^2$, and $1.84 \times 10^{-4}$ under unknown $\sigma^2$ for vanilla CIs, reducing about 18\% of those by the block max-min estimator which are $2.05\times 10^{-4}$ and $2.23\times 10^{-4}$ respectively. Similar conclusion can be made for CV-adjusted CIs.

The summary of statistics for the coverage probabilities of the 95\% CIs in $d=3$ is listed in Table \ref{tab:3d-average-max-min}. As the simulated critical values used in this simulation
% \hd{
% , $1.63$ for the block average estimator and $1.62$ for the block max-min estimator, 
% }
are less accurate due to the relatively small sample size in $d=3$, the vanilla CIs for both the block average estimator and the block max-min estimator suffer from slight under-coverage. However, when using CV-adjusted CIs with improved accuracy in the mean of the coverage probabilities, similar reduction in the mean squared errors of the coverage probabilities can be observed for the  block average estimators.

\begin{table}[htb]
    \centering
    \label{fig:2d-average-max-min-3}
    \resizebox{\textwidth}{!}{
    \begin{tabular}{c | c c | c c | c c | c c }\hline
         & \multicolumn{4}{c}{vanilla CIs} & \multicolumn{4}{c}{CV-adjusted CIs} \\
         \hline
         & \multicolumn{2}{c}{$\sigma$ known} & \multicolumn{2}{c}{$\sigma$ unknown} & \multicolumn{2}{c}{$\sigma$ known} & \multicolumn{2}{c}{$\sigma$ unknown} \\
        \hline
        & average & max-min & average & max-min & average & max-min & average & max-min \\ \hline
        mean & 0.9414 & 0.9405 & 0.9405 & 0.9397 & 0.9503 & 0.9503 & 0.9503 & 0.9503 \\
       median & 0.9444 & 0.9440 & 0.9438 & 0.9432 & 0.9503 & 0.9503 & 0.9503 & 0.9503 \\
        s.d. & 0.00961 & 0.01072 & 0.00973 & 0.01080 & 0.00146 & 0.00156 & 0.00140 & 0.00155 \\
        \hline
        \end{tabular}
        }
    \vspace{0.5ex}
   \caption{Summary of statistics for the coverage probabilities of the 95\% CIs in $d=2$ by the block average and the block max-min estimators.}
    \label{tab:2d-average-max-min}
\end{table}

\begin{table}[htb]
	\resizebox{\textwidth}{!}{
    \begin{tabular}{c | c c | c c | c c | c c }\hline
         & \multicolumn{4}{c}{vanilla CIs} & \multicolumn{4}{c}{CV-adjusted CIs} \\
         \hline
         & \multicolumn{2}{c}{$\sigma$ known} & \multicolumn{2}{c}{$\sigma$ unknown} & \multicolumn{2}{c}{$\sigma$ known} & \multicolumn{2}{c}{$\sigma$ unknown} \\
        \hline
        & average & max-min & average & max-min & average & max-min & average & max-min \\ \hline
        mean & 0.9273 & 0.9287 & 0.9232 & 0.9246 & 0.9513 & 0.9510 & 0.9512 & 0.9511 \\
       median & 0.9277 & 0.9293 & 0.9236 & 0.9247 & 0.9516 & 0.9513 & 0.9515 & 0.9515 \\
        s.d. & 0.01521 & 0.01450 & 0.01563 & 0.01497 & 0.00342 & 0.00428 & 0.00334 & 0.00402 \\
        \hline
        \end{tabular}
        }
	\vspace{0.5ex}
	\caption{Summary of statistics for the coverage probabilities of the 95\% CIs in $d=3$ by the block average and the block max-min estimators. }
	\label{tab:3d-average-max-min}
\end{table}

\subsection{Numerical comparison with Banerjee-Wellner (BW) CIs in $d=1$} 

\begin{figure}[hbt]
	\centering
	\subfigure[$f_0(x) = e^{2x}$]{
		\label{fig:comp-1} %% label for first subfigure
		\includegraphics[width=0.48\textwidth]{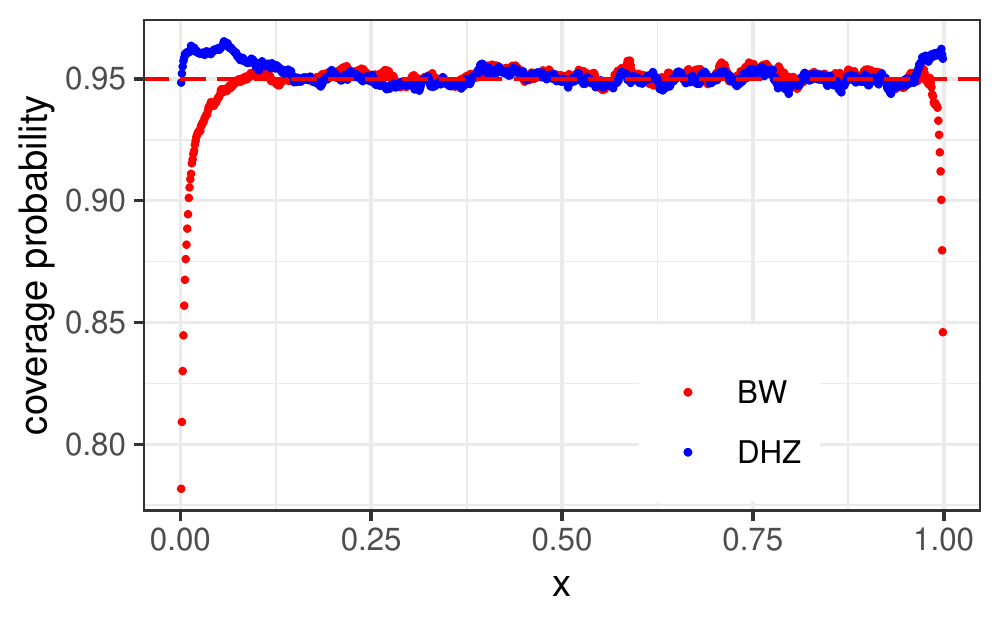}}
	\hspace{0\textwidth}
	\subfigure[$f_0(x) = 10x^5$]{
		\label{fig:comp-2} %% label for second subfigure
		\includegraphics[width=0.48\textwidth]{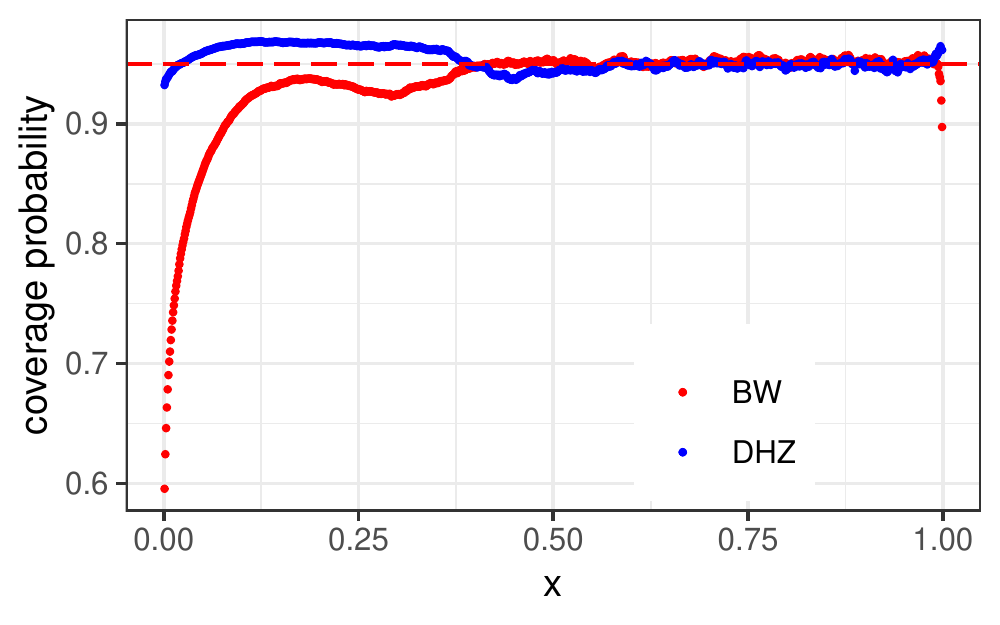}}
	\hspace{0\textwidth}
	\subfigure[$f_0(x) = (4 \pi x + \sin (4 \pi x))/2$]{
		\label{fig:comp-3} %% label for first subfigure
		\includegraphics[width=0.48\textwidth]{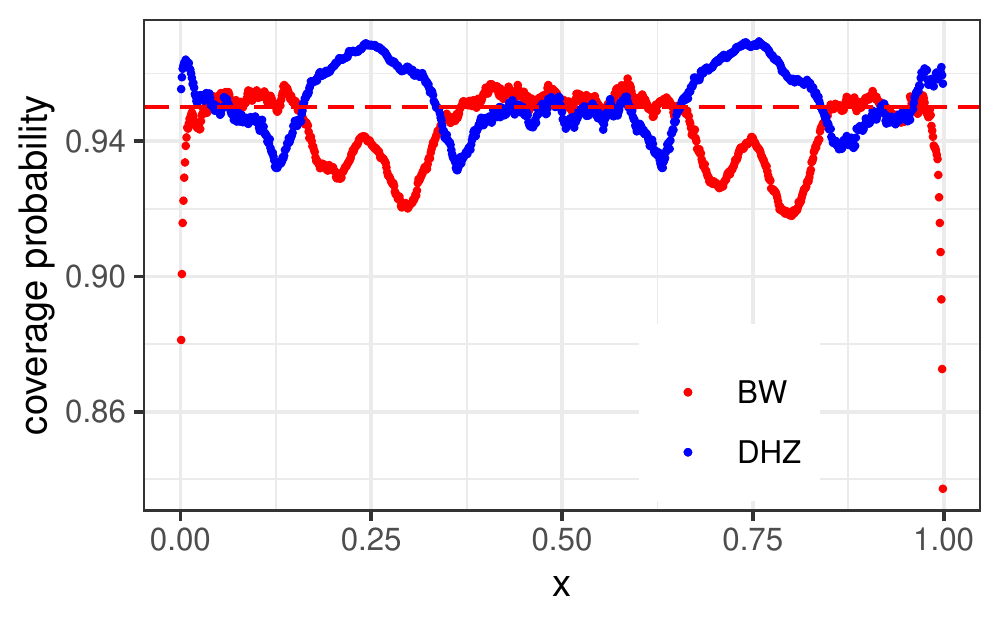}}
	\hspace{0\textwidth}
	\subfigure[$f_0(x) = \log(x)$]{
		\label{fig:comp-4} %% label for second subfigure
		\includegraphics[width=0.48\textwidth]{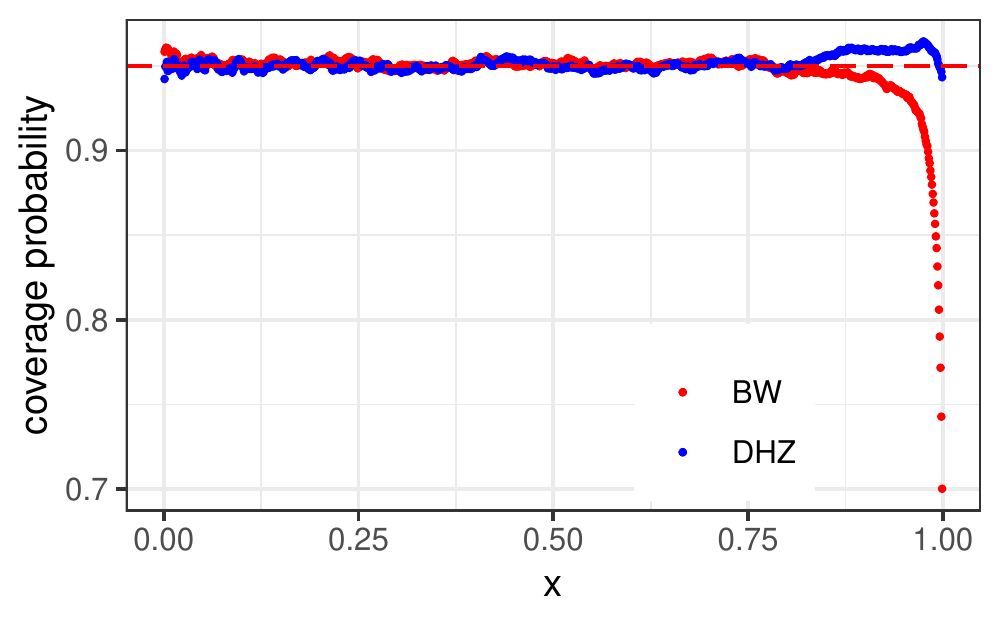}}
	
	\caption{Scatter plots for the coverage probabilities of the 95\% BW CIs and proposed CIs.}
	\label{fig:comp}
\end{figure}

\begin{figure}[hbt]
	\centering
	\includegraphics[width=0.9\textwidth]{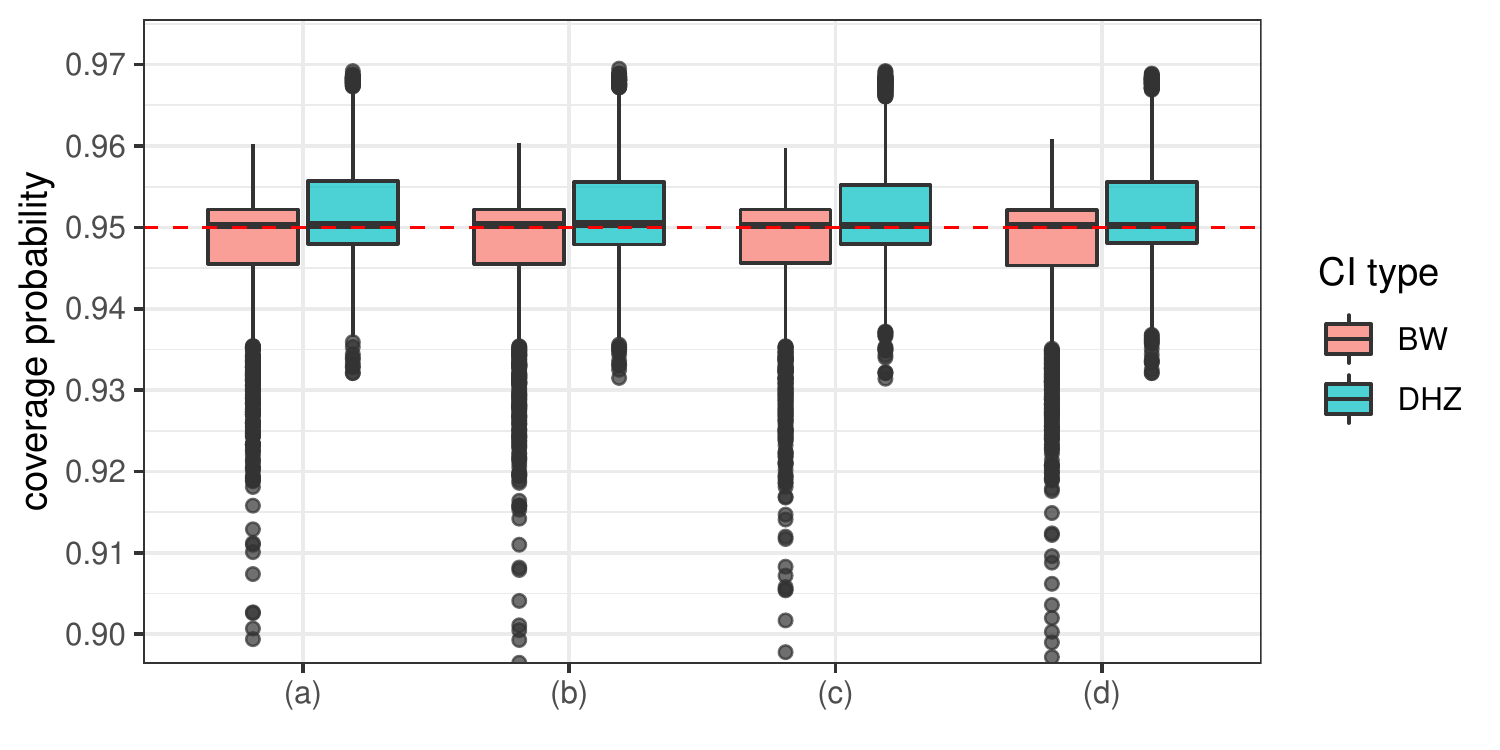}
	\caption{Boxplots for the coverage probabilities of the 95\% BW CIs and proposed CIs. (a), (b), (c) and (d) correspond to the functions in Figure \ref{fig:comp}. Here some outliers of the boxplots for BW CI are removed as they can be as small as $0.6$. }
	\label{fig:comp-bp}
\end{figure}

In this subsection, we compare numerically the coverage probabilities and the lengths of the CIs in \cite{banerjee2001likelihood} (cf. Section \ref{section:compare_BW_theory}), hereafter referred as \emph{BW CIs}, with the CIs proposed in this paper, labelled as \emph{DHZ} in the plots. 

\subsubsection{Coverage probability}\label{subsubsection:coverage} For the performance on coverage, we consider four different mean functions $f_0(x)$'s: $e^{2x}$, $10x^5$, $(4\pi x + \sin (4\pi x))/2$ and $\log (x + 0.001)$ on $[0,1]$. We let $n=10^3$ and $x \in \{1/n, \ldots, (n-1)/n, n/n\}$, and assume $\sigma=1$ is known. For fair comparison, we use the vanilla CIs with the universal critical value $2.11$ for the proposed CIs, and the recommended critical value 2.26916 in \cite[Method 2, Table 2]{banerjee2001likelihood} for the BW CIs, both for $95\%$ coverage. The scatter plots and boxplots for the coverage probabilities at design points $\{1/n, \ldots, (n-1)/n\}$ are given in Figure \ref{fig:comp} and Figure \ref{fig:comp-bp}. The coverage probabilities are approximated by the relative frequencies of the successful coverage of the corresponding CIs out of $B = 10^4$ simulations.

Both types of CIs have rather accurate coverage probabilities at design points that are far from the outskirts for functions with non-extreme derivatives; see Figure \ref{fig:comp-1} and Figure \ref{fig:comp-4}. 
Nevertheless, the proposed CIs appear to have two advantages. First, the edge effect in the BW CIs is much more severe than in the proposed CIs; many more outskirt points suffer under-coverage for the BW CIs. A similar phenomenon is observed in the related current status model in \cite[Figure 9.16, pp. 270]{groeneboom2014nonparametric}. Note that as the LSE is probably inconsistent near the edge, we do not expect the BW or the proposed CIs to provide accurate coverage in theory. However, it turns out the proposed CIs are able to give some reasonably good coverage for outskirt points numerically.

Second, the coverage probabilities of the proposed CIs over flat regions (where derivatives of the mean functions are small) are more biased towards over-coverage, while the BW CIs are likely to suffer again under-coverage; see Figure \ref{fig:comp-2} when $x \in [0.1, 0.5]$, and Figure \ref{fig:comp-3} when $x$ is around $0.25$ and $0.75$.

\subsubsection{CI lengths} 
	
We compare the lengths of the BW CIs and the proposed CIs. 

We may first have a glance at the BW and the proposed CIs. Let $f_0(x) = e^{2x}$ and continue the above setting in Section \ref{subsubsection:coverage} but with $n = 10^2$. For one observation $\{(X_i, Y_i), X_i = i/n, 1 \le i \le n\}$, we compute both CIs for the function values at design points $\{1/n, \ldots, (n-1)/n\}$ and plot them in Figure \ref{fig:comp_CI}. 

\begin{figure}[hbt]
	\centering
	\subfigure[BW CI]{
		\label{fig:comp_CI_BW} %% label for second subfigure
		\includegraphics[width=0.48\textwidth]{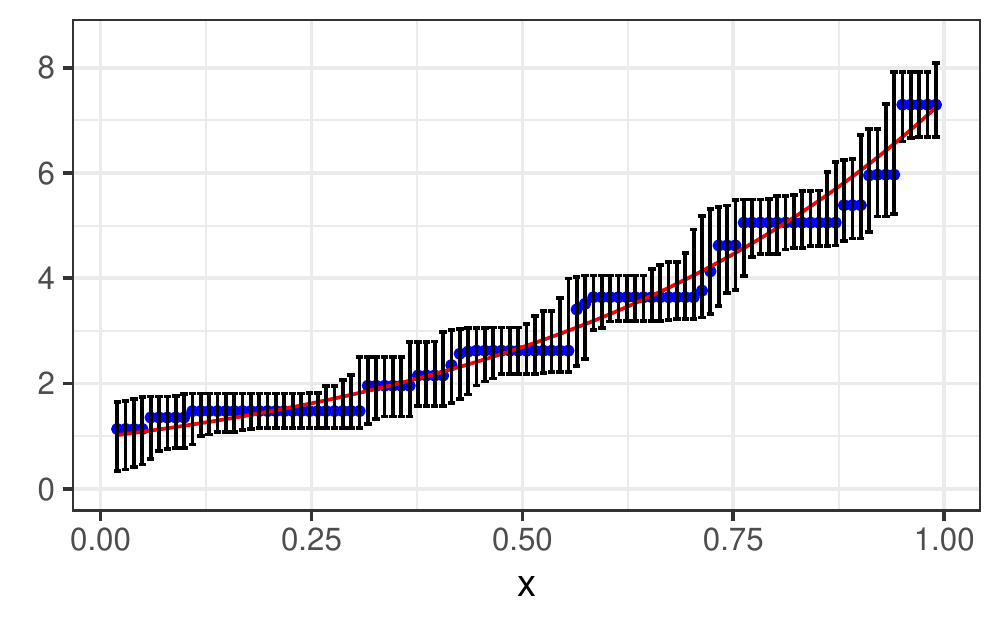}}
	\subfigure[proposed CI (DHZ)]{
		\label{fig:comp_CI_DHZ} %% label for first subfigure
		\includegraphics[width=0.48\textwidth]{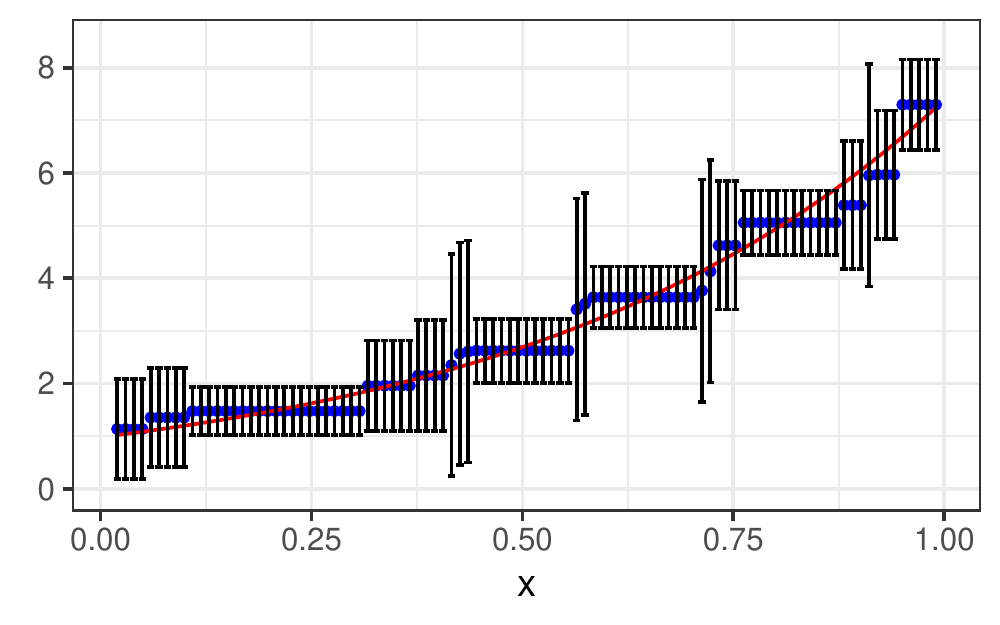}}
	\caption{95\% BW CIs and proposed CIs for $f_0(x) = e^{2x}$ at $\{0.01, \ldots, 0.99\}$. Here $n=10^2$ and $\sigma=1$ is known. Red line represents $f_0(x) = e^{2x}$ and blue dots are $\widehat{f}_n(X_i)$.}
	\label{fig:comp_CI}
\end{figure}

From Figure \ref{fig:comp_CI}, we notice an interesting difference between the BW and the proposed CIs: The lower and upper boundaries of the BW CIs seem to be monotone, but, because it is possible to have small $n_{\widehat{u}, \widehat{v}}(x_0)$ for any $x_0$, the proposed CIs at certain locations could be quite large.  We may employ some practical remedies for the proposed CI, e.g., adding an extra size constraint $n_{\widehat{u}, \widehat{v}}(x_0) \ge 5$ in the maximization and minimization of \eqref{def:block_estiamtors}, but the improvement may not be as  substantial as in Figure \ref{fig:comp_CI_BW}.

We also observe in this simulation that the BW CIs are narrower in Figure \ref{fig:comp_CI}. To investigate this phenomenon more carefully, we continue the setting in Section \ref{subsubsection:coverage} with $n=10^3$ and compute the lengths of both CIs for design points $\{1/n, \ldots, (n-1)/n\}$. We run $10^4$ simulations, so that, at each design point, we obtain $10^4$ BW CIs and the proposed CIs. The lengths of the CIs for $f_0(x) = e^{2x}$ and $f_0(x) = 10x^5$ are given in Figure \ref{fig:CI_length}. In each subfigure of Figure \ref{fig:CI_length}, the lower (resp. upper) boundary of the shaded area represents the 1st (resp. 3rd) quantile of the lengths of $10^4$ CIs, the black line is the median of the lengths, and the red line represents the length of the oracle CI defined in \eqref{def:oracle_CI}. Although the lengths of the proposed CI shrink at the oracle rate, which supports Theorem \ref{thm:CI_exact} \eqref{eqn:length} in $d=1$, it seems that the BW CIs based on LRT are usually narrower than the proposed CIs. 

\begin{figure}[hbt]
	\centering
	\subfigure[$f_0(x) = e^{2x}$, DHZ]{
		\label{fig:CI_length-1} %% label for first subfigure
		\includegraphics[width=0.48\textwidth]{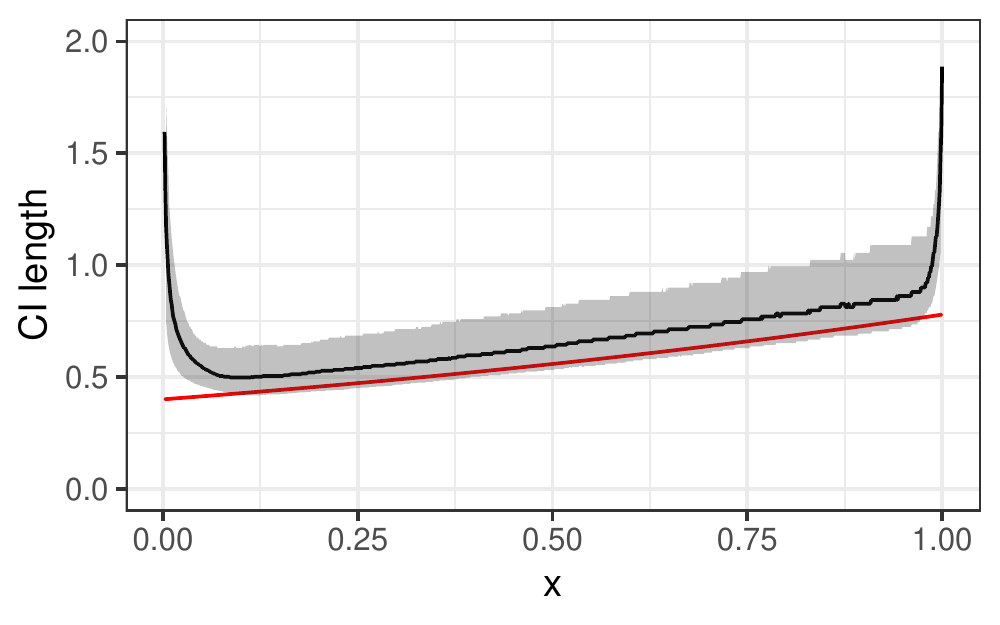}}
	\subfigure[$f_0(x) = e^{2x}$, BW]{
		\label{fig:CI_length-1} %% label for first subfigure
		\includegraphics[width=0.48\textwidth]{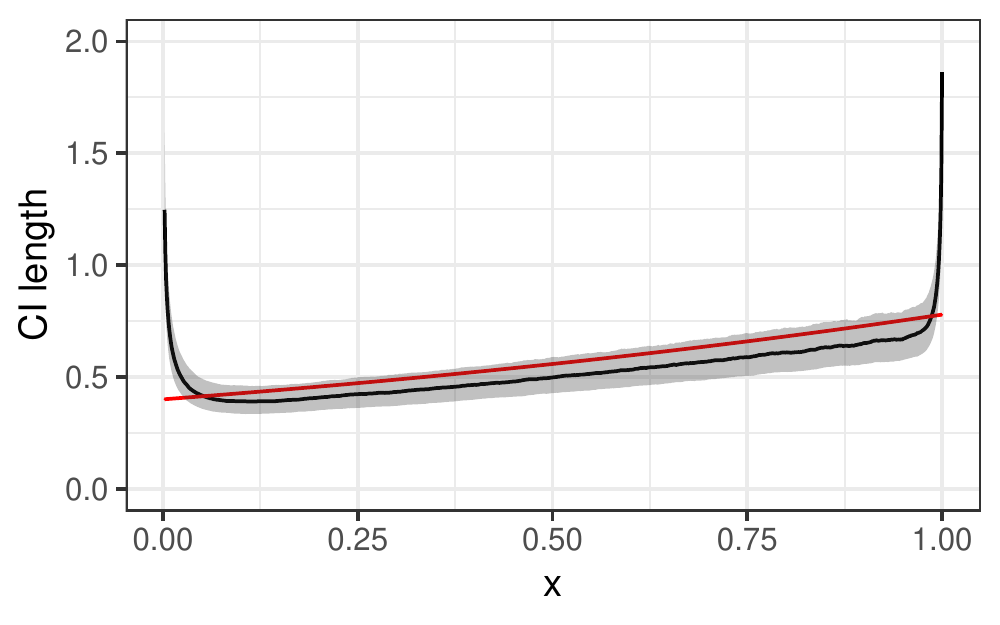}}
	\subfigure[$f_0(x) = 10x^5$, DHZ]{
		\label{fig:CI_length-2} %% label for second subfigure
		\includegraphics[width=0.48\textwidth]{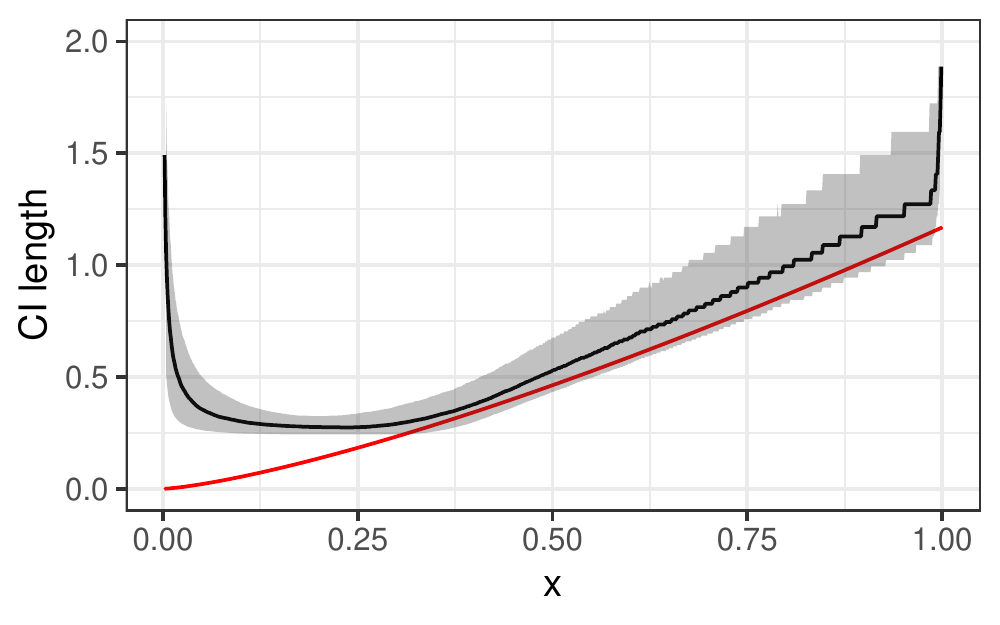}}
	\subfigure[$f_0(x) = 10x^5$, BW]{
		\label{fig:CI_length-2} %% label for second subfigure
		\includegraphics[width=0.48\textwidth]{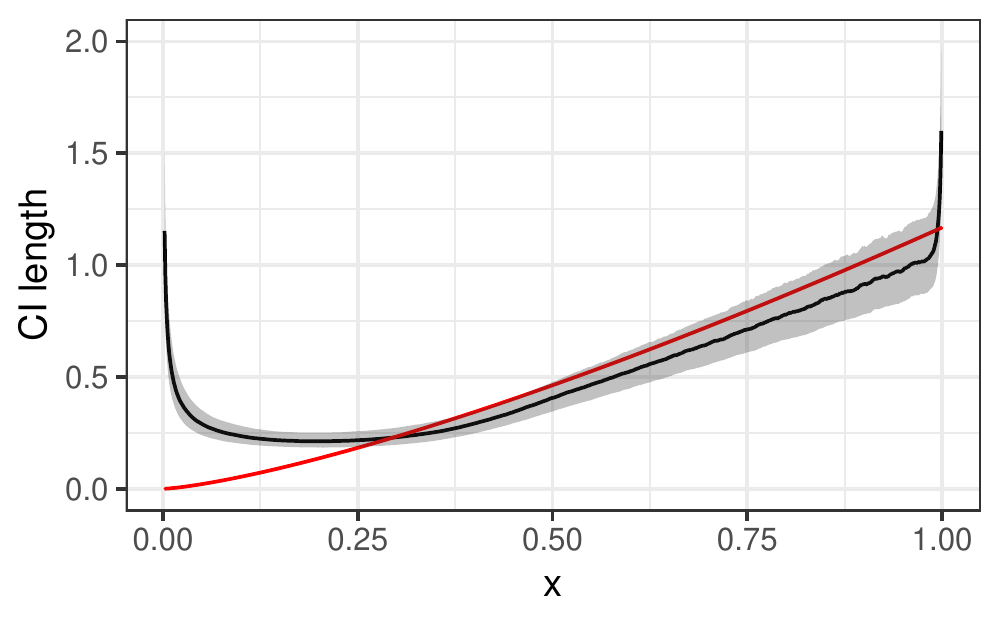}}
	\caption{Lengths of the 95\% BW CIs and proposed CIs.}
	\label{fig:CI_length}
\end{figure}

\appendix

\section{Proofs in Section \ref{section:CI_isotonic_regression}}

\subsection{Proof of Theorem \ref{thm:pivotal_limit_distribution}} 

\begin{lemma}\label{lem:uniqueness_supinf}
	Let $\{G(h_1,h_2): h_1\in T_1, h_2 \in T_2\}$ be a Gaussian process with continuous sample path defined on $T_1\times T_2$, where $T_j\subset \R^d (j=1,2)$ are compact. Further suppose for any $h_1\neq h_1^\prime$ in $T_1$ there exists $\epsilon>0$ such that
  \begin{align*}
  \inf_{ \substack{ k \in \N, \bm{a},\bm{a}' \in \Delta_k,\\ 
  h_{1,j}\in H_1, h_{1,j}^\prime \in H_1^\prime}}
  \mathrm{Var}\bigg(\sum_{j=1}^k a_j G(h_{1,j},h_{2,j})-\sum_{j=1}^k a_j' G(h_{1,j}',h_{2,j}')\bigg)>0,
  \end{align*}
  where $\Delta_k\equiv \{(a_1,\ldots,a_k) \in \R_{\geq 0}^k: \sum_{j=1}^k a_j =1\}$, 
  $H_1=\{\bar{h}_1\in T_1:\|\bar{h}_1-h_1\|\le \epsilon\}$, 
  $H_1^\prime=\{\bar{h}_1\in T_1:\|\bar{h}_1-h_1^\prime\|\le \epsilon\}$, 
  and $\pnorm{\cdot}{}$ is the Euclidean norm on $\R^d$. Then with probability one, 
	\begin{align*}
	\inf_{h_2 \in T_2} G(h_1^\ast,h_2)&= \sup_{h_1 \in T_1}\ \inf_{h_2 \in T_2} G(h_1,h_2)
	\end{align*}
	has exactly one solution $h_1^*$.
\end{lemma}

\begin{proof}
The set $\{(\bar{h}_1,\bar{h}_1^\prime): \bar{h}_1 \neq \bar{h}_1^\prime\}$ is covered 
by countably many product sets $H_1\times H_1^\prime$ such that the pair $\{H_1,H_1^\prime\}$ satisfies the 
non-degenerate variance assumption. Thus, it  
suffices to prove that for every such pair
\begin{align}\label{ineq:uniqueness_supinf_1}
\Prob\bigg(\sup_{\bar{h}_1 \in H_1} \inf_{\bar{h}_2 \in T_2}G(\bar{h}_1,\bar{h}_2) 
= \sup_{\bar{h}_1 \in H_1'} \inf_{\bar{h}_2 \in T_2} G(\bar{h}_1,\bar{h}_2)\bigg) = 0.
\end{align}
Any $(\bar{h}_1,\bar{h}_2) \in T_1\times T_2$ may be identified with $\delta_{(\bar{h}_1,\bar{h}_2)}(G)\equiv G(\bar{h}_1,\bar{h}_2)-\E G(\bar{h}_1,\bar{h}_2) \in L_2(\Omega,\mathcal{A},P)\equiv L_2$. Let $\bar{F}$ be smallest closed subspace of $L_2$ containing $\{\delta_{\bar{h}_1,\bar{h}_2}(G): (\bar{h}_1,\bar{h}_2)\in T_1\times  T_2\}$. Members of $\bar{F}$ are mean zero Gaussian variables. By the non-degenerate variance assumption, the closed convex hulls of the collections $\{\delta_{\bar{h}_1,\bar{h}_2}(G): (\bar{h}_1,\bar{h}_2)\in H_1\times  T_2\}$ and $\{\delta_{\bar{h}_1,\bar{h}_2}(G): (\bar{h}_1,\bar{h}_2)\in H_1'\times  T_2\}$ (which are compact) in $\bar{F}\subset L_2$ do not intersect, so by Hahn-Banach separation theorem (see \cite[Theorem 3.4(b)]{rudin1991functional}) and Riesz representation theorem, there exists some $\mathcal{N}(0,1)$ variable $\zeta \in  \bar{F}$ such that 
\begin{align}\label{ineq:uniqueness_supinf_2}
\sup_{(\bar{h}_1,\bar{h}_2)\in H_1\times  T_2} \iprod{\zeta}{\delta_{(\bar{h}_1,\bar{h}_2)}}_{L_2}<\inf_{(\bar{h}_1,\bar{h}_2) \in H_1'\times  T_2} \iprod{\zeta}{\delta_{(\bar{h}_1,\bar{h}_2)}}_{L_2}.
\end{align}
Let $Q(\cdot)$ and $Q'(\cdot)$ be stochastic processes defined by
\begin{align*}
Q(g)&\equiv \sup_{\bar{h}_1 \in H_1 } \inf_{\bar{h}_2 \in  T_2}\Big[G(\bar{h}_1,\bar{h}_2) - \iprod{\zeta}{\delta_{(\bar{h}_1,\bar{h}_2)}}_{L_2}(\zeta-g)\Big],\\ 
 Q^\prime(g)&\equiv \sup_{\bar{h}_1 \in H_1' } \inf_{\bar{h}_2 \in  T_2}
\Big[G(\bar{h}_1,\bar{h}_2) - \iprod{\zeta}{\delta_{(\bar{h}_1,\bar{h}_2)}}_{L_2}(\zeta-g)\Big].
\end{align*}
We aim to prove that $\Prob\{Q(\zeta)=Q^\prime(\zeta)|Q(\cdot),Q^\prime(\cdot)\}=0$ 
as \eqref{ineq:uniqueness_supinf_1} can be written as $\Prob\{Q(\zeta)=Q^\prime(\zeta)\}=0$. 
We note that $\zeta\sim \mathcal{N}(0,1)$ is independent of $Q(\cdot)$ and $Q^\prime(\cdot)$. 
Conditionally on $Q(\cdot)$ and $Q^\prime(\cdot)$ let $g_1\in \R$ be a solution of $Q(g_1)=Q^\prime(g_1)$ if exists. Then for $c>0$,
\begin{align*}
Q(g_1+c)&\leq   
Q(g_1) + c \sup_{\bar{h}_1 \in H_1  } \sup_{\bar{h}_2 \in  T_2} \iprod{\zeta}{\delta_{(\bar{h}_1,\bar{h}_2)}}_{L_2}\\
& < Q'(g_1) + c  \inf_{\bar{h}_1 \in H_1' } \inf_{\bar{h}_2 \in  T_2} \iprod{\zeta}{\delta_{(\bar{h}_1,\bar{h}_2)}}_{L_2} \leq Q'(g_1+c),
\end{align*}
where the first inequality follows from the definition of $Q(\cdot)$, the second from 
\eqref{ineq:uniqueness_supinf_2}, and the third from the definition of $Q^\prime(\cdot)$. Similarly one may show that $Q(g_1+c)>Q'(g_1+c)$ for $c<0$. Thus, $\Prob\{Q(\zeta)=Q^\prime(\zeta)\}=0$ 
as conditionally on $Q(\cdot)$ and $Q^\prime(\cdot)$, $\zeta\sim \mathcal{N}(0,1)$ and 
the set $\{g: Q(g)=Q^\prime(g)\}$ contains at most one point. This completes the proof of (\ref{ineq:uniqueness_supinf_1}).
\end{proof}

For notational convenience, let
\begin{align}\label{def:r_n}
\omega_n \equiv n_\ast^{-\frac{1}{2+\sum_{k=\kappa_\ast}^s \alpha_k^{-1}}},\quad r_n \equiv (\omega_n^{1/\alpha_1},\ldots,\omega_n^{1/\alpha_d})\bm{1}_{[\kappa_\ast:d]} \in \R^d.
\end{align}
Write $\hat{u}(x_0) = x_0 - \hat{h}_1(x_0) r_n$ and $\hat{v}(x_0)=x_0+\hat{h}_2(x_0)r_n$ for some $\hat{h}_1(x_0),\hat{h}_2(x_0) \in \R_{\geq 0}^d$. It is shown in \cite[Proposition 5]{han2019limit} that $\hat{h}_1(x_0),\hat{h}_2(x_0) \in \R_{\geq 0}^d$ are well-defined with high probability despite the fact that the first $\kappa_\ast-1$ coordinates of $r_n$ are $0$.

\begin{proof}[Proof of Theorem \ref{thm:pivotal_limit_distribution}]
For notational simplicity, we only prove the theorem in the case where $\bm{\alpha}=\bm{1}_d$, $\kappa_\ast=1$ and $s=d$; 
other cases follow from minor modifications. Let 
\begin{align*}
\mathbb{U}(h_1,h_2)\equiv \frac{\sigma\cdot \G(h_1,h_2)}{\prod_{k=1}^d (h_1+h_2)_k}+\frac{1}{2}\sum_{k=1}^d \partial_ k f_0(x_0) (h_2-h_1)_k.
\end{align*}
Then 
\begin{align*}
&\omega_n^{-1}\big(\hat{f}_n^{-}(x_0)-f_0(x_0)\big)\rightsquigarrow \sup_{h_1> 0}\inf_{h_2> 0} \mathbb{U}(h_1,h_2),\\
&\omega_n^{-1}\big(\hat{f}_n^{+}(x_0)-f_0(x_0)\big)\rightsquigarrow \inf_{h_2> 0}\sup_{h_1> 0} \mathbb{U}(h_1,h_2).
\end{align*}
Let $(h_1^\ast(x_0),h_2^\ast(x_0))$ be any pair defined by
\begin{align*}
\inf_{h_2> 0} \mathbb{U}(h_1^\ast(x_0), h_2) &= \sup_{h_1> 0}\inf_{h_2> 0} \mathbb{U}(h_1,h_2), \\
\sup_{h_1> 0} \mathbb{U}(h_1, h_2^\ast(x_0)) &= \inf_{h_2> 0}\sup_{h_1> 0} \mathbb{U}(h_1,h_2).
\end{align*} 
For notational convenience, we will omit the dependence on $x_0$ if no confusion could arise. Let 
\begin{align*}
\mathbb{U}_n(h_1,h_2) 
\equiv \omega_n^{-1}\big\{\bar{Y}|_{[x_0-r_nh_1,x_0+r_nh_2]} - f_0(x_0)\big\}. 
\end{align*}
By \cite{han2019limit}, for any $\epsilon>0$, we may find $c=c(\epsilon)>1$ such that for $n$ large enough, there exists an event $\Omega_c$ with $\Prob(\Omega_c)\ge 1-\epsilon$ in which
\begin{align*}
\sup_{h_1 \in [c^{-1}\bm{1},c\bm{1}]}\inf_{h_2 \in [c^{-1}\bm{1},c\bm{1}]} \mathbb{U}(h_1,h_2) &= \sup_{h_1>0}\inf_{h_2>0} \mathbb{U}(h_1,h_2),\\
\sup_{h_1 \in [c^{-1}\bm{1},c\bm{1}]}\inf_{h_2 \in [c^{-1}\bm{1},c\bm{1}]} \mathbb{U}_n(h_1,h_2) & =\sup_{h_1>0}\inf_{h_2>0} \mathbb{U}_n(h_1,h_2),\\
\inf_{h_2 \in [c^{-1}\bm{1},c\bm{1}]} \sup_{h_1 \in [c^{-1}\bm{1},c\bm{1}]}\mathbb{U}(h_1,h_2)& = \inf_{h_2>0}\sup_{h_1>0} \mathbb{U}(h_1,h_2),\\
\inf_{h_2 \in [c^{-1}\bm{1},c\bm{1}]} \sup_{h_1 \in [c^{-1}\bm{1},c\bm{1}]}\mathbb{U}_n(h_1,h_2) & = \inf_{h_2>0}\sup_{h_1>0} \mathbb{U}_n(h_1,h_2). 
\end{align*}
In particular, $\hat{h}_1,\hat{h}_2 \in [c^{-1}\bm{1},c\bm{1}]$ and $h_1^\ast,h_2^\ast  \in [c^{-1}\bm{1},c\bm{1}]$ hold on $\Omega_c$. The non-degenerate variance assumption in Lemma~\ref{lem:uniqueness_supinf} holds for  
the Gaussian process $\mathbb{U}(h_1,h_2)$ 
as the process $\mathbb{G}(h_1,h_2)$ 
has independent increments in each coordinate when other coordinates are held fixed. Thus, by  Lemma~\ref{lem:uniqueness_supinf}, we may also assume that 
$h_1^\ast$ and $h_2^\ast$ are uniquely defined on $\Omega_c$. On the event $\Omega_c$, we have
\begin{align*}
&\omega_n^{-1}\big(\hat{f}_n^{-}(x_0)-f_0(x_0)\big)\\
&\equiv \sup_{c^{-1}\bm{1}\leq h_1\leq c\bm{1} }\inf_{c^{-1}\bm{1}\leq h_2\leq c\bm{1}} \mathbb{U}_n(h_1,h_2)\\
& \equiv \sup_{c^{-1}\bm{1}\leq h_1\leq c\bm{1} }\inf_{c^{-1}\bm{1}\leq h_2\leq c\bm{1}} \bigg[\frac{\G_n(h_1,h_2)}{\prod_{k=1}^d \big((h_1)_k+(h_2)_k\big)}\cdot (1+\mathfrak{o}(1))\\
&\qquad\qquad\qquad\qquad 
+(1+\mathfrak{o}(1))\sum_{k=1}^d \frac{ \partial_k f_0(x_0)}{2}\big((h_2)_k-(h_1)_k\big)\bigg],
\end{align*}
where for any $h_1,h_2>0$, $
\G_n(h_1,h_2)\equiv \omega_n \sum_{i: x_0-h_1 r_n\leq X_i\leq x_0+h_2 r_n} \xi_i$. 
A similar identity holds for $\omega_n^{-1}\big(\hat{f}_n^{+}(x_0)-f_0(x_0)\big)$ by switching sup and inf. 
On the other hand, on the event $\Omega_c$, $\hat{h}_1\leq t \le c{\bf 1}$ if and only if 
\begin{align*}
\sup_{h_1 \in [c^{-1}\bm{1},t]}\inf_{h_2 \in [c^{-1}\bm{1},c\bm{1}]} \mathbb{U}_n(h_1,h_2) =  \sup_{h_1 \in [c^{-1}\bm{1},c\bm{1}]}\inf_{h_2 \in [c^{-1}\bm{1},c\bm{1}]} \mathbb{U}_n(h_1,h_2),
\end{align*}
and similarly $\hat{h}_2\leq s\le c{\bf 1}$ if and only if 
\begin{align*}
\inf_{h_2 \in [c^{-1}\bm{1},s]}\sup_{h_1 \in [c^{-1}\bm{1},c\bm{1}]} \mathbb{U}_n(h_1,h_2) 
=  \inf_{h_2 \in [c^{-1}\bm{1},c\bm{1}]}\sup_{h_1 \in [c^{-1}\bm{1},c\bm{1}]} \mathbb{U}_n(h_1,h_2).
\end{align*}
Using that $\mathbb{U}_n\rightsquigarrow \mathbb{U}$ in $\ell^\infty([c^{-1}\bm{1},c\bm{1}]\times [c^{-1}\bm{1},c\bm{1}])$, it follows from the continuous mapping theorem that for any measurable subset $A \subset [c^{-1}\bm{1},c\bm{1}]$,
we have $\sup_{h_1 \in A}\inf_{c^{-1}\bm{1}\leq h_2\leq c\bm{1}} \mathbb{U}_n(h_1,h_2) \rightsquigarrow \sup_{h_1 \in A}\inf_{c^{-1}\bm{1}\leq h_2\leq c\bm{1}} \mathbb{U}(h_1,h_2)$, and a similar conclusion holds by switching sup and inf. 
Hence for any $(u,v,t,s) \in \R\times \R\times  \R_{\geq 0}^d \times \R_{\geq 0}^d$, 
\begin{align*}
&\limsup_n\Prob\big( \big\{\omega_n^{-1}(\hat{f}_n^{-}(x_0)-f_0(x_0))\leq u,\\
&\qquad\qquad  \omega_n^{-1}(\hat{f}_n^{+}(x_0)-f_0(x_0))\leq v,\hat{h}_1\leq t,\hat{h}_2\leq s\big\}\cap \Omega_c\big)\\
&\leq \Prob\bigg(\bigg\{\sup_{h_1 \in [c^{-1}\bm{1},c\bm{1}]}\inf_{h_2 \in [c^{-1}\bm{1},c\bm{1}]} \mathbb{U}(h_1,h_2)\leq u, \inf_{h_2 \in [c^{-1}\bm{1},c\bm{1}]}\sup_{h_1 \in [c^{-1}\bm{1},c\bm{1}]} \mathbb{U}(h_1,h_2)\leq v\\
&\qquad \sup_{h_1 \in [c^{-1}\bm{1},t]}\inf_{h_2 \in [c^{-1}\bm{1},c\bm{1}]} \mathbb{U}(h_1,h_2) =  \sup_{h_1 \in [c^{-1}\bm{1},c\bm{1}]}\inf_{h_2 \in [c^{-1}\bm{1},c\bm{1}]} \mathbb{U}(h_1,h_2),\\
&\qquad \inf_{h_2 \in [c^{-1}\bm{1},s]}\sup_{h_1 \in [c^{-1}\bm{1},c\bm{1}]} \mathbb{U}_n(h_1,h_2) =  \inf_{h_2 \in [c^{-1}\bm{1},c\bm{1}]}\sup_{h_1 \in [c^{-1}\bm{1},c\bm{1}]} \mathbb{U}_n(h_1,h_2)\bigg\} \cap \Omega_c \bigg)\\
& \leq \Prob\bigg(\sup_{h_1\geq 0}\inf_{h_2\geq 0} \mathbb{U}(h_1,h_2)\leq u, \inf_{h_2\geq 0}\sup_{h_1\geq 0} \mathbb{U}(h_1,h_2)\leq v, h_1^\ast \leq t, h_2^\ast \leq s\bigg).
\end{align*}
This means that
\begin{align*}
& \limsup_n\Prob\big(\omega_n^{-1}(\hat{f}_n^{-}(x_0)-f_0(x_0))\leq u,\\
&\qquad\qquad \omega_n^{-1}(\hat{f}_n^{+}(x_0)-f_0(x_0))\leq v, \hat{h}_1\leq t,\hat{h}_2\leq s\big)\\
&\leq \Prob\bigg(\sup_{h_1\geq 0}\inf_{h_2\geq 0} \mathbb{U}(h_1,h_2)\leq u, \inf_{h_2\geq 0}\sup_{h_1\geq 0} \mathbb{U}(h_1,h_2)\leq v, h_1^\ast \leq t, h_2^\ast \leq s\bigg)+\epsilon.
\end{align*} 
Taking $\epsilon \downarrow 0$ we obtain the one-sided estimate. For the other side, we may proceed similarly, but only need to check the convergence for a continuity point $(u,v,t,s) \in \R\times \R\times \R_{\geq 0}^d \times \R_{\geq 0}^d$ of the map $(u,v,t,s)\mapsto \Prob\big(\sup_{h_1\geq 0}\inf_{h_2\geq 0} \mathbb{U}(h_1,h_2)\leq u, \inf_{h_2\geq 0}\sup_{h_1\geq 0} \mathbb{U}(h_1,h_2)\leq v, h_1^\ast \leq t, h_2^\ast \leq s\big)$. More concretely, 
\begin{align*}
&\liminf_n\Prob\big( \big\{\omega_n^{-1}(\hat{f}_n^{-}(x_0)-f_0(x_0))\leq u,\\
&\qquad\qquad  \omega_n^{-1}(\hat{f}_n^{+}(x_0)-f_0(x_0))\leq v,\hat{h}_1\leq t,\hat{h}_2\leq s\big\}\cap \Omega_c\big)\\
&\geq \Prob\bigg(\bigg\{\sup_{h_1 \in [c^{-1}\bm{1},c\bm{1}]}\inf_{h_2 \in [c^{-1}\bm{1},c\bm{1}]} \mathbb{U}(h_1,h_2)< u-\epsilon, \\
&\qquad \inf_{h_1 \in [c^{-1}\bm{1},c\bm{1}]}\sup_{h_2 \in [c^{-1}\bm{1},c\bm{1}]} \mathbb{U}(h_1,h_2)< v-\epsilon,\\
&\qquad \sup_{h_1 \in [c^{-1}\bm{1},c\bm{1}] \setminus [c^{-1}\bm{1},t-\epsilon)}\inf_{h_2 \in [c^{-1}\bm{1},c\bm{1}]} \mathbb{U}(h_1,h_2) <  \sup_{h_1 \in [c^{-1}\bm{1},c\bm{1}]}\inf_{h_2 \in [c^{-1}\bm{1},c\bm{1}]} \mathbb{U}(h_1,h_2),\\
&\qquad \inf_{h_2 \in [c^{-1}\bm{1},c\bm{1}]\setminus [c^{-1}\bm{1},t-\epsilon)}\sup_{h_1 \in [c^{-1}\bm{1},c\bm{1}]} \mathbb{U}(h_1,h_2) 
>  \inf_{h_2 \in [c^{-1}\bm{1},c\bm{1}]}\sup_{h_1 \in [c^{-1}\bm{1},c\bm{1}]} 
\mathbb{U}(h_1,h_2) \bigg\} \cap \Omega_c\bigg)\\
& \geq \Prob\bigg(\sup_{h_1\geq 0}\inf_{h_2\geq 0} \mathbb{U}(h_1,h_2)<u-\epsilon, \inf_{h_2\geq 0}\sup_{h_1\geq 0} \mathbb{U}(h_1,h_2)<v-\epsilon, \\
&\qquad\qquad h_1^\ast < t-\epsilon, h_2^\ast < s-\epsilon\bigg)-\epsilon,
\end{align*}
and taking $\epsilon \downarrow 0$ to conclude. Hence we have proved the joint convergence in distribution of
\begin{align}\label{ineq:local_pivotal_1}
&\big(\omega_n^{-1}(\hat{f}_n^{-}(x_0)-f_0(x_0)), \omega_n^{-1}(\hat{f}_n^{+}(x_0)-f_0(x_0)), \hat{h}_1,\hat{h}_2\big)\\
&\rightsquigarrow \big(\sup_{h_1\geq 0}\inf_{h_2\geq 0} \mathbb{U}(h_1,h_2), 
\inf_{h_2\geq 0}\sup_{h_1\geq 0} \mathbb{U}(h_1,h_2),h_1^\ast, h_2^\ast \big).\nonumber
\end{align}
 Now by continuous mapping, 
\begin{align*}
&\sqrt{n_{\hat{u},\hat{v}}(x_0)}\big(\hat{f}_n(x_0)-f_0(x_0)\big)\\
& = \sqrt{(1+\mathfrak{o}(1))\prod_{k=1}^d (\hat{h}_2+\hat{h}_1)_k}\cdot\frac{1}{2}\bigg\{ \omega_n^{-1}(\hat{f}_n^{-}(x_0)-f_0(x_0))+ \omega_n^{-1}(\hat{f}_n^{+}(x_0)-f_0(x_0)) \bigg\}\\
&\rightsquigarrow \frac{1}{2} \sqrt{\prod_{k=1}^d (h_2^\ast+h_1^\ast)_k} \cdot \bigg( \sup_{h_1\geq 0}\inf_{h_2\geq 0} \mathbb{U}(h_1,h_2)+\inf_{h_2\geq 0}\sup_{h_1\geq 0} \mathbb{U}(h_1,h_2)  \bigg).
\end{align*}
So the remaining task is to verify that the right hand side is free of nuisance parameters $\{\partial_k f_0(x_0): k=1,\ldots,d\}$. To this end, let
\begin{align*}
\mathbb{V}(g_1,g_2)\equiv \frac{ \G(g_1,g_2)}{\prod_{k=1}^d (g_1+g_2)_k}+\sum_{k=1}^d  (g_2-g_1)_k,
\end{align*}
and $(g_1^\ast,g_2^\ast)$ be (a.s. uniquely) defined by 
\begin{align*}
&\inf_{g_2\geq 0} \mathbb{V}(g_1^\ast,g_2) = \sup_{g_1\geq 0}\inf_{g_2\geq 0} \mathbb{V}(g_1,g_2), \quad \sup_{g_1\geq 0} \mathbb{V}(g_1,g_2^\ast) = \inf_{g_2\geq 0}\sup_{g_1\geq 0} \mathbb{V}(g_1,g_2).
\end{align*} 
We wish to relate $g_i^\ast$ to $h_i^\ast$. Let $\gamma_0,\gamma_1,\ldots,\gamma_d>0$ be such that
\begin{align*}
\gamma_0\bigg(\prod_{k=1}^d \gamma_k\bigg)^{-1/2}  = \sigma,\quad \gamma_0 \gamma_k = \frac{1}{2} \partial_k f_0(x_0),
\end{align*}
and let $g_i = \gamma h_i$ where $\gamma=(\gamma_1,\ldots,\gamma_d)$. Then
\begin{align*}
\gamma_0 \mathbb{V}(g_1,g_2) &=\gamma_0 \frac{\G(\gamma h_1,\gamma h_2)}{\prod_{k=1}^d \gamma_k(h_1+h_2)_k}+ \sum_{k=1}^d \gamma_0 \gamma_k (h_2-h_1)_k\\
& =_d \gamma_0\bigg(\prod_{k=1}^d \gamma_k\bigg)^{-1/2} \frac{\G(h_1, h_2)}{\prod_{k=1}^d (h_1+h_2)_k}+\sum_{k=1}^d  \gamma_0\gamma_k(h_2-h_1)_k\\
& = \mathbb{U}(h_1,h_2).
\end{align*}
Therefore $(g_i^\ast)_k = \gamma_k (h_i^\ast)_k$ for $k=1,\ldots,d$, and the limit distribution becomes
\begin{align*}
&\frac{1}{2}\sqrt{\prod_{k=1}^d (g_2^\ast+g_1^\ast)_k}\cdot \bigg(\prod_{k=1}^d \gamma_k\bigg)^{-1/2} \cdot \gamma_0 \bigg( \sup_{g_1\geq 0}\inf_{g_2\geq 0} \mathbb{V}(g_1,g_2)+\inf_{g_2\geq 0}\sup_{g_1\geq 0} \mathbb{V}(g_1,g_2) \bigg)\\
& = \frac{\sigma}{2} \sqrt{\prod_{k=1}^d (g_2^\ast+g_1^\ast)_k}\cdot \bigg( \sup_{g_1\geq 0}\inf_{g_2\geq 0} \mathbb{V}(g_1,g_2)+\inf_{g_2\geq 0}\sup_{g_1\geq 0} \mathbb{V}(g_1,g_2) \bigg)\equiv \sigma\cdot \mathbb{L}_{\bm{1}_d}.
\end{align*}
The distribution of $\mathbb{L}_{\bm{1}_d}$ now is free of $K(f_0,x_0)$, as desired. 
\end{proof}

\subsection{Proof of Theorems \ref{thm:CI_exact} and \ref{thm:adaptive_confidence_interval}}

\begin{proof}[Proof of Theorems \ref{thm:CI_exact} and \ref{thm:adaptive_confidence_interval}]
We only prove Theorem \ref{thm:CI_exact}; Theorem \ref{thm:adaptive_confidence_interval} follows easily. The claim that $\Prob_{f_0}\big(f_0(x_0) \in \mathcal{I}_n(\delta)\big) \to 1-\delta$ follows immediately. For the other claim, note by (\ref{ineq:local_pivotal_1}), we have by continuous mapping that
\begin{align*}
\omega_n \sqrt{n_{\hat{u},\hat{v}}(x_0)} &= \sqrt{(1+\mathfrak{o}(1))\prod_{k=1}^d (\hat{h}_2+\hat{h}_1)_k}\\
&\rightsquigarrow \sqrt{\prod_{k=1}^d (h_2^\ast+h_1^\ast)_k} = \sqrt{\prod_{k=1}^d (g_2^\ast+g_1^\ast)_k}\cdot \bigg(\prod_{k=1}^d \gamma_k\bigg)^{-1/2}
\end{align*}
where $\gamma_k$ is as in the proof of Theorem \ref{thm:pivotal_limit_distribution}.
Equivalently, for any $\hat{\sigma}^2\to_p \sigma^2$,
\begin{align*}
\bigg(\sigma^2 \prod_{k=1}^d \frac{\partial_k f_0(x_0)}{2}\bigg)^{1/(d+2)} \omega_n  \bigg(\sqrt{n_{\hat{u},\hat{v}}(x_0)}\big/\hat{\sigma}\bigg) \rightsquigarrow \sqrt{\prod_{k=1}^d (g_2^\ast+g_1^\ast)_k} = \mathbb{S}_{\bm{1}_d}.
\end{align*} 
Combined with $
\abs{\mathcal{I}_n(c_\delta)} = 2 c_\delta{\hat{\sigma}}/{ \sqrt{n_{\hat{u},\hat{v}}(x_0)} }$, it follows that
\begin{align*}
{\abs{\mathcal{I}_n(c_\delta)} } \bigg/{2c_\delta \omega_n \bigg(\sigma^2 \prod_{k=1}^d \frac{\partial_k f_0(x_0)}{2}\bigg)^{1/(d+2)} }\rightsquigarrow \mathbb{S}_{\bm{1}_d}^{-1}.
\end{align*}
The fact that $\mathbb{S}_{\bm{1}_d}^{-1} = \mathcal{O}_{\mathbf{P}}(1)$ follows from the proof of \cite[Proposition 7]{han2019limit}. 
\end{proof}

\subsection{Proof of Proposition \ref{prop:consistent_var_est}}

\begin{lemma}\label{lem:rate_estimator}
	Assume the same conditions as in Theorem \ref{thm:limit_distribution_pointwise}. Then for any $c>1$, $\sup_{\pnorm{h}{}\leq c} \abs{
		\hat{f}_n(x_0)-f_0(x_0+hr_n)} = \mathcal{O}_{\mathbf{P}}(\omega_n)$.
\end{lemma}
\begin{proof}
	This is a strengthened version of \cite[Proposition 4]{han2019limit}, and the claim follows by controlling $\abs{f_0(x_0)-f_0(x_0+hr_n)}$ uniformly in $\pnorm{h}{}\leq c$ by a Taylor expansion. Details are omitted.
\end{proof}

\begin{proof}[Proof of Proposition \ref{prop:consistent_var_est}]
	Note that
	\begin{align*}
	\sigma_{\hat{u},\hat{v}}^2 &= \frac{1}{n_{\hat{u},\hat{v}}(x_0)} \sum_{X_i \in [\hat{u}(x_0),\hat{v}(x_0)]}  \xi_i^2 +  \frac{1}{n_{\hat{u},\hat{v}}(x_0)} \sum_{X_i \in [\hat{u}(x_0),\hat{v}(x_0)]}  \xi_i (f_0(X_i)-\hat{f}_n(x_0))\\
	&\qquad+  \frac{1}{n_{\hat{u},\hat{v}}(x_0)} \sum_{X_i \in [\hat{u}(x_0),\hat{v}(x_0)]}  (\hat{f}_n(x_0)-f_0(X_i))^2\\
	& = (I)+(II)+(III).
	\end{align*}
	By \cite[Proposition 5]{han2019limit}, for any fixed $\epsilon>0$, there exists $c=c(\epsilon)>1$ such that for $n$ large enough, with probability at least $1-\epsilon$, $\big(\hat{h}_i(x_0)\big)_k \in [c^{-1},c]$ for $\kappa_\ast\leq k\leq d$ and $\big(\hat{h}_i(x_0)\big)_k = 0$ for $1\leq k\leq \kappa_\ast-1$ for $i=1,2$. Denote this event $\mathcal{E}$. On $\mathcal{E}$, in the fixed lattice design, with $d_\ast \equiv d-\kappa_\ast+1$,
	\begin{align*}
	n_{\hat{u},\hat{v}}(x_0) &= n\cdot  \omega_n^{\sum_{k=\kappa_\ast}^s \alpha_k^{-1}} \prod_{k=\kappa_\ast}^d \big(\hat{h}_2(x_0)+\hat{h}_1(x_0)\big)_k\cdot (1+\mathfrak{o}_{\mathbf{P}}(1))\\
	& \in n\cdot  \omega_n^{\sum_{k=\kappa_\ast}^s \alpha_k^{-1}}[(2c^{-1})^{d_\ast},(2c)^{d_\ast}]\cdot (1+\mathfrak{o}_{\mathbf{P}}(1)).
	\end{align*}
	For $(I)$, note that for any $\delta>0$, using (essentially) \cite[Lemma 2]{deng2018isotonic},
	\begin{align*}
	&\Prob\bigg(\biggabs{\frac{1}{n_{\hat{u},\hat{v}}(x_0)} \sum_{X_i \in [\hat{u}(x_0),\hat{v}(x_0)]}  (\xi_i^2-\sigma^2)}>\delta
	\bigg)\\
	&\leq \Prob(\mathcal{E}^c)+ \Prob\bigg(\bigg\{ \biggabs{\frac{1}{n_{\hat{u},\hat{v}}(x_0)} \sum_{X_i \in [\hat{u}(x_0),\hat{v}(x_0)]}  (\xi_i^2-\sigma^2)}>\delta\bigg\}\cap \mathcal{E}\bigg)\\
	&\leq \epsilon + \Prob\bigg(\sup_{ \substack{\big(\hat{h}_i(x_0)\big)_k =0 ,1\leq k\leq \kappa_\ast-1\\ \big(\hat{h}_i(x_0)\big)_k \in [c^{-1},c], \kappa_\ast\leq k\leq d\\ i=1,2} } \biggabs{ \sum_{X_i \in [x_0-\hat{h}_1(x_0)r_n, x_0+\hat{h}_2(x_0)r_n]} \big(\xi_i^2-\sigma^2) }\\
	&\qquad\qquad \gtrsim \big(n\cdot  \omega_n^{\sum_{k=\kappa_\ast}^s \alpha_k^{-1}}\big)\delta \bigg)\\
	&\leq \epsilon + \Prob\bigg(\biggabs{ \sum_{X_i \in [x_0-c r_n, x_0+c r_n]} \big(\xi_i^2-\sigma^2) }\gtrsim \big(n\cdot  \omega_n^{\sum_{k=\kappa_\ast}^s \alpha_k^{-1}}\big)\delta \bigg).
	\end{align*}
	Taking $n \to \infty$ followed by $\epsilon \to 0$ to see that $(I)\to_p \sigma^2$. 
	
	Next we handle $(III)$. On the event $\mathcal{E}$, by Lemma \ref{lem:rate_estimator},
	\begin{align*}
	(III)&\lesssim \big(n\cdot  \omega_n^{\sum_{k=\kappa_\ast}^s \alpha_k^{-1}}\big)^{-1} \sum_{X_i \in [x_0-cr_n,x_0+cr_n]} \big(\hat{f}_n(x_0)-f_0(X_i)\big)^2\\
	&= \mathcal{O}_{\mathbf{P}}(\omega_n^2) = \mathfrak{o}_{\mathbf{P}}(1).
	\end{align*}
	The second term $(II)$ can be handled easily using Cauchy-Schwarz inequality combined with $(I)$ and $(III)$. The random design case can be handled similarly using large deviation inequality (cf. Bernstein's inequality). We omit the details.
\end{proof}

\subsection{Proof of Proposition \ref{prop:uniform_tail_L}}

We need the following Dudley's entropy integral bound and Gaussian concentration inequality; see e.g., \cite[Theorem 2.3.7]{gine2015mathematical}.

\begin{lemma}\label{lem:dudley_entropy_integral}
	Let $(T,d)$ be a pseudo metric space, and $(X_t)_{t \in T}$ be a sub-Gaussian process. Then for any $t_0 \in T$,
	\begin{align*}
	\E \sup_{t \in T} \abs{X_t} \leq \E \abs{X_{t_0}}+ 4\sqrt{2}\int_0^{\mathrm{diam}(T)/2} \sqrt{\log \mathcal{N}(\epsilon,T,d)}\ \d{\epsilon}.
	\end{align*}
	Here $C>0$ is a universal constant.
\end{lemma}

The following Gaussian concentration inequality will also be useful; see e.g., \cite[Theorem 2.2.7]{gine2015mathematical} and the comments after the statement of that theorem.

\begin{lemma}[Gaussian concentration inequality]\label{lem:Gaussian_concentration}
	Let $(T,d)$ be a pseudo metric space, and $(X_t)_{t \in T}$ be a mean-zero separable Gaussian process. Then with $\sigma^2\equiv \sup_{t \in T} \mathrm{Var}(X_t)$, for any $u >0$,
	\begin{align*}
	\Prob\big( \bigabs{\sup_{t \in T} \abs{X_t} - \E \sup_{t \in T} \abs{X_t}}>u\big)\le 2\exp\big(-u^2/2\sigma^2\big).
	\end{align*}
\end{lemma}

\begin{proof}[Proof of Proposition \ref{prop:uniform_tail_L}]
	We drop the dependence of $\mathbb{S}_{\bm{\alpha}}, \mathbb{V}_{\bm{\alpha}}, g_{i,\bm{\alpha}}^\ast$ on $\bm{\alpha}$ if no confusion arises. Let $(\bm{1}_1)_k \equiv \bm{1}_{1\leq k\leq s}+ (x_0)_k \bm{1}_{s+1\leq k\leq d}$ and $(\bm{1}_2)_k \equiv \bm{1}_{1\leq k\leq s} + (1-x_0)_k \bm{1}_{s+1\leq k\leq d}$, and $d_\ast \equiv d-\kappa_\ast+1, s_\ast \equiv s-\kappa_\ast+1$. Let
   \begin{align*}
   \mathbb{V}_{\bm{\alpha}}^{-}&\equiv \sup_{g_1 \in \mathscr{G}_1}\inf_{g_2 \in \mathscr{G}_2} \mathbb{V}_{\bm{\alpha}}(g_{1},g_{2}), \quad \mathbb{V}_{\bm{\alpha}}^{+} \equiv \inf_{g_2 \in \mathscr{G}_2} \sup_{g_1 \in \mathscr{G}_1}\mathbb{V}_{\bm{\alpha}}(g_{1},g_{2}).
   \end{align*}
	On the event $\{\max_{\kappa_\ast \leq k\leq s} (g_2^\ast)_k>u\}$ where $u\geq 1$, it holds that
	\begin{align*}
	\mathbb{V}_{\bm{\alpha}}^{+} &\geq \frac{\G(\bm{1}_1,g_2^\ast)}{\prod_{k=\kappa_\ast}^d \big((g_2^\ast)_k+(\bm{1}_1)_k)} + \sum_{k=\kappa_\ast}^s  \frac{(g_2^\ast)_k^{\alpha_k+1}-1}{(g_2^\ast)_k+1}\\
	&\geq  -\bigg(\prod_{k=s+1}^d (x_0)_k\bigg)^{-1}\sup_{g_2\geq 0}\frac{ \abs{\G(\bm{1}_1,g_2)}}{\prod_{k=\kappa_\ast}^s \big((g_2)_k+1\big)}+ (u-1)-d.
	\end{align*}
	On the other hand, 
	\begin{align*}
	\mathbb{V}_{\bm{\alpha}}^{-}\vee \mathbb{V}_{\bm{\alpha}}^{+}&\leq \frac{\G(g_1^\ast,\bm{1}_2)}{\prod_{k=\kappa_\ast}^d \big((\bm{1}_2)_k+(g_1^\ast)_k)} + \sum_{k=\kappa_\ast}^s \frac{1-(g_1^\ast)_k^{\alpha_k+1}}{(g_1^\ast)_k+1}\\
	&\leq \bigg(\prod_{k=s+1}^d (1-x_0)_k\bigg)^{-1} \sup_{g_1\geq 0}\frac{ \abs{\G(g_1,\bm{1}_2)}}{\prod_{k=\kappa_\ast}^s \big((g_1)_k+1\big)}+d.
	\end{align*}
	Hence
	\begin{align*}
	&\Prob \big(\max_{\kappa_\ast \leq k\leq s} (g_2^\ast)_k>u\big) 
	\vee \Prob\big(\mathbb{V}_{\bm{\alpha}}^{-} >u\big)\\\
	&\leq \Prob \bigg(\sup_{g_2\geq 0}\frac{ \abs{\G(\bm{1}_1,g_2)}}{\prod_{k=\kappa_\ast}^s \big((g_2)_k+1\big)}> \bigg(\prod_{k=s+1}^d (x_0)_k\bigg)(u-1-2d)_+/2\bigg)\\
	&\qquad + \Prob\bigg(\sup_{g_1\geq 0}\frac{ \abs{\G(g_1,\bm{1}_2)}}{\prod_{k=\kappa_\ast}^s \big((g_1)_k+1\big)}> \bigg(\prod_{k=s+1}^d (1-x_0)_k\bigg)(u-1-2d)_+/2\bigg)\\
	&\equiv (I)+(II).
	\end{align*}
	We only handle the first term above. To this end, let $u_d \equiv \big(\prod_{k=s+1}^d (x_0)_k\big)(u-1-2d)_+/2$, and let $G(g)\equiv \G(\bm{1}_1,g)$. It follows by symmetry and the peeling device that
	\begin{align*}
	(I)&\leq \sum_{\ell_k\geq 0,\kappa_\ast\leq k\leq s} \Prob \bigg(\sup_{ \substack{2^{\ell_k}-1\leq g_k\leq 2^{\ell_k+1}-1, \kappa_\ast \leq k\leq s\\ 0\leq g_k\leq (1-x_0)_k, s+1\leq k\leq d} } \frac{\abs{G(g)}}{ \prod_{k=\kappa_\ast}^s \big(g_k+1\big) }>u_d\bigg)\\
	&\leq \sum_{\ell_k\geq 0,\kappa_\ast\leq k\leq s} \Prob\bigg(\sup_{ \substack{2^{\ell_k}-1\leq g_k\leq 2^{\ell_k+1}-1, \kappa_\ast \leq k\leq s\\ 0\leq g_k\leq (1-x_0)_k, s+1\leq k\leq d} }  \abs{G(g)}> \prod_{k=\kappa_\ast}^s 2^{\ell_k} \cdot u_d\bigg)\\
	&\leq \sum_{\ell_k \geq 1, \kappa_\ast\leq k \leq s} \Prob\bigg(\sup_{ \substack{0\leq g_k\leq 2^{\ell_k}-1, \kappa_\ast \leq k\leq s\\ 0\leq g_k\leq (1-x_0)_k, s+1\leq k\leq d} }  \abs{G(g)}> 2^{-s_\ast}\prod_{k=\kappa_\ast}^s 2^{\ell_k} \cdot u_d\bigg).
	\end{align*}
	Note that for any $0\leq (g_i)_k\leq 2^{\ell_k}-1, \kappa_\ast \leq k\leq s$ and $0\leq (g_i)_k\leq (1-x_0)_k, s+1\leq k\leq d$ with $i=1,2$, the natural induced metric of $G$  satisfies
	\begin{align*}
	d_G^2(g_1,g_2) &= \E \big(G(g_1)-G(g_2)\big)^2\\
	&= \biggabs{\prod_{k=\kappa_\ast}^s \big(1+(g_1)_k\big) \prod_{k=s+1}^d \big(x_0+g_1)_k-\prod_{k=\kappa_\ast}^s \big(1+(g_2)_k\big) \prod_{k=s+1}^d (x_0+g_2)_k }\\
	&\leq  \prod_{k=s+1}^d (x_0+g_1)_k \biggabs{\prod_{k=\kappa_\ast}^s \big(1+(g_1)_k\big) -\prod_{k=\kappa_\ast}^s \big(1+(g_2)_k\big) } \\
	&\qquad + \biggabs{\prod_{k=s+1}^d (x_0+g_1)_k-\prod_{k=s+1}^d (x_0+g_2)_k}\prod_{k=\kappa_\ast}^s \big(1+(g_2)_k\big)\\
	&\leq 2\prod_{k=\kappa_\ast}^s 2^{\ell_k} \cdot  \sum_{k=\kappa_\ast}^d \abs{(g_1-g_2)_k}.
	\end{align*}
	Here the last inequality follows essentially by \cite[Lemma 15]{han2019limit}. On the other hand, $
	\sup_{g_1,g_2} d_G^2(g_1,g_2) \leq  \prod_{k=\kappa_\ast}^s 2^{\ell_k}\equiv D^2$, where the supremum is taken over $0\leq (g_i)_k\leq 2^{\ell_k}-1, \kappa_\ast \leq k\leq s$ and $0\leq (g_i)_k\leq (1-x_0)_k, s+1\leq k\leq d$ with $i=1,2$. Using that $
	\E \abs{G(0)}\leq \sqrt{\mathrm{Var}(G(0))}\leq \prod_{k=s+1}^{d}(x_0)_k\leq 1$, 
	it follows from Dudley's entropy integral (cf. Lemma \ref{lem:dudley_entropy_integral}) that 
	\begin{align*}
	&\E \sup_{ \substack{0\leq g_k\leq 2^{\ell_k}-1, \kappa_\ast \leq k\leq s\\ 0\leq g_k\leq (1-x_0)_k, s+1\leq k\leq d} }  \abs{G(g)}\\
	&\leq  \E \abs{G(0)}+ 4\sqrt{2} \int_{0}^{D/2}\sqrt{\log \mathcal{N}\bigg(\epsilon, \prod_{k=\kappa_\ast}^s [0,2^{\ell_k}]\times \prod_{k=s+1}^d [0,(1-x_0)_k], d_G \bigg)}\ \d{\epsilon}\\
	&\leq 1+4\sqrt{2}\int_0^{D/2} \sqrt{\log \prod_{k=\kappa_\ast}^s \frac{2^{\ell_k}}{\epsilon^2/ \big(d_\ast 2 \prod_{k=\kappa_\ast}^s 2^{\ell_k}\big) }\cdot \prod_{k=s+1}^d \frac{1}{ \epsilon^2/ \big(d_\ast 2 \prod_{k=\kappa_\ast}^s 2^{\ell_k}\big) }  }\ \d{\epsilon}\\
%	&\leq 1+ 4\sqrt{2} \sqrt{d_\ast \log (2d_\ast )}(D/2)+4\sqrt{2} \sqrt{d_\ast+1} \int_0^{D/2} \sqrt{\log \bigg(\frac{D^2}{\epsilon \wedge \epsilon^2 }\bigg)}\ \d{\epsilon}\\
	&\leq C_1\cdot  D\sqrt{\log D}.
	\end{align*}
	Hence if $u_d\geq C_2$, then $2^{-s_\ast}\prod_{k=\kappa_\ast}^s 2^{\ell_k}\cdot u_d = 2^{-s_\ast}u_d \cdot D^2 \geq 2 C_1\cdot D\sqrt{\log D}$,	and therefore by Gaussian concentration inequality (cf. Lemma \ref{lem:Gaussian_concentration}),
	\begin{align*}
	&\Prob\bigg(\sup_{ \substack{0\leq g_k\leq 2^{\ell_k}-1, \kappa_\ast \leq k\leq s\\ 0\leq g_k\leq (1-x_0)_k, s+1\leq k\leq d} }  \abs{G(g)}> 2^{-s_\ast}\prod_{k=\kappa_\ast}^s 2^{\ell_k} \cdot u_d\bigg)\\
	&\leq \Prob\bigg(\sup_{\cdots }  \abs{G(g)} - \E \sup_{\cdots }  \abs{G(g)}> 2^{-s_\ast-1}\prod_{k=\kappa_\ast}^s 2^{\ell_k} \cdot u_d\bigg)\\
	&\leq 2e^{-{2^{-2s_\ast-2}\prod_{k=\kappa_\ast}^s 2^{2\ell_k} \cdot u_d^2}\big/{2\prod_{k=\kappa_\ast}^s 2^{\ell_k}}} \leq 2  e^{-\prod_{k=\kappa_\ast}^s 2^{\ell_k} \cdot u_d^2/C_3}.
	\end{align*}
	Now we may assemble all the estimates together to get
	\begin{align*}
	(I)&\leq  2\sum_{\ell_k \geq 1, \kappa_\ast\leq k \leq s}  e^{-\prod_{k=\kappa_\ast}^s 2^{\ell_k} \cdot u_d^2/C_3}\leq C_4 e^{- u_d^2/C_4 }.
	\end{align*}
	Hence for $
	u \geq C_5$,
	we have for some $C_6=C_6(d,x_0)$
	\begin{align*}
	\Prob \big(\max_{\kappa_\ast \leq k\leq s} (g_2^\ast)_k>u\big)\leq C_6 e^{- u^2/C_6 }.
	\end{align*}
	A similar bound holds for $\Prob\big(\max_{\kappa_\ast \leq k\leq s} (g_1^\ast)_k>u\big)$. So if $s_\ast\geq 1$, for $
	t\geq C_7$,
	we have for some $C_8=C_8(d,x_0)$
	\begin{align*}
	\Prob\big(\mathbb{S}(g_1^\ast,g_2^\ast)>t\big)& = \Prob\bigg(\sqrt{\prod_{k=\kappa_\ast}^s (g_2^\ast+g_1^\ast)_k} >t\bigg)\\
	&\leq \Prob\big(\max_{\kappa_\ast \leq k\leq s} (g_1^\ast)_k>t^{2/s_\ast}/2\big)+ \Prob\big(\max_{\kappa_\ast \leq k\leq s} (g_2^\ast)_k>t^{2/s_\ast}/2\big)\\
	&\leq C_8 e^{- t^{4/s_\ast}/C_8 }.
	\end{align*}
	Similarly, for $
	t'\geq C_9$,
	we have for some $C_{10}=C_{10}(d,x_0)$
	\begin{align*}
	\Prob\big(\abs{\mathbb{V}_{\bm{\alpha}}^{-} } >t'\big)\vee \Prob\big(\abs{\mathbb{V}_{\bm{\alpha}}^{+} } >t'\big)\leq C_{10} e^{-(t')^{2}/C_{10} }.
	\end{align*}
	This means that for $s_\ast\geq 1$ and $v\geq C_{11}$, we have for some $C_{12}=C_{12}(d,x_0)$,
	\begin{align*}
	\Prob\big(\abs{\mathbb{L}(g_1^\ast,g_2^\ast)}>v\big)&\leq  \Prob\big(\mathbb{S}(g_1^\ast,g_2^\ast)>v^{s_\ast/(s_\ast+2)}\big)+\Prob\big( \abs{(\mathbb{V}_{\bm{\alpha}}^{-}+\mathbb{V}_{\bm{\alpha}}^{+})/2 }>v^{2/(s_\ast+2)}\big)\\
	&\leq C_{12} e^{- v^{4/(s_\ast+2)}/C_{12} }.
	\end{align*}
	The above display is also valid for $s_\ast = 0$. 
\end{proof}

\subsection{Proof of Theorem \ref{thm:LRT}}

\begin{proof}[Proof of Theorem \ref{thm:LRT}]
	The main idea of the proof is contained in \cite{banerjee2007likelihood} so we only give a sketch. First,
	\begin{align*}
	\binom{n^{\frac{1}{2+\alpha^{-1}}} \big(\hat{f}_n(x_0+h\cdot n^{-\frac{\alpha^{-1}}{2+\alpha^{-1}}})-f_0(x_0)\big) }{n^{\frac{1}{2+\alpha^{-1}}} \big(\hat{f}_n^0(x_0+h\cdot n^{-\frac{\alpha^{-1}}{2+\alpha^{-1}}})-f_0(x_0)\big)} \rightsquigarrow \binom{g_{1,\partial^\alpha f_0(x_0)/(\alpha+1)!;\alpha}(h)}{g^0_{1,\partial^\alpha f_0(x_0)/(\alpha+1)!;\alpha}(h)}
	\end{align*}
	where the convergence is in $\big(\ell^\infty([-c,c])\big)^2$ for any $c>0$. This follows from the standard approach (as outlined in the proof of Theorem 2.1 of \cite{banerjee2007likelihood}). Let $J_n$ be the set of all indices $i$ for which $\hat{f}_n^0(X_i)\neq \hat{f}_n(X_i)$. By the characterization of $\hat{f}_n^0$ (see \cite[pp. 939]{banerjee2007likelihood}), we may write $J_n = \cup_{k=1}^\ell B_k$, where $B_k$'s are all constant pieces of $\hat{f}_n^0$. Note that
	\begin{align*}
	2\log \lambda_n(m_0)& = -\sum_{i \in J_n} (Y_i-\hat{f}_n(X_i))^2 + \sum_{i \in J_n} (Y_i-\hat{f}_n^0(X_i))^2\\
	& = -2 \sum_{i \in J_n} (Y_i-f_0(x_0))(f_0(x_0)-\hat{f}_n(X_i)) \\
	&\qquad + 2 \sum_{i \in J_n} (Y_i-f_0(x_0))(f_0(x_0)-\hat{f}_n^0(X_i)) \\
	&\qquad - \sum_{i \in J_n} (f_0(x_0)-\hat{f}_n(X_i))^2 + \sum_{i \in J_n} (f_0(x_0)-\hat{f}_n^0(X_i))^2\\
	& = \sum_{i \in J_n} (f_0(x_0)-\hat{f}_n(X_i))^2 - \sum_{i \in J_n} (f_0(x_0)-\hat{f}_n^0(X_i))^2,
	\end{align*}
	where the last equality follows, e.g., by using \cite[(2.6)]{banerjee2007likelihood}: With $w_k \equiv \hat{f}_n^0|_{B_k} = \bar{Y}|_{B_k}$ where $\hat{f}_n^0|_{B_k}\neq m_0=f_0(x_0)$,
	\begin{align*}
	2 \sum_{i \in J_n} (Y_i-f_0(x_0))(f_0(x_0)-\hat{f}_n^0(X_i))&= 2 \sum_{k} \sum_{i \in B_k} (Y_i-f_0(x_0))(f_0(x_0)-w_k)\\
	& = -2 \sum_k \sum_{i \in B_k} (f_0(x_0)-w_k)^2.
	\end{align*}
	Hence, with $D_n$ be the smallest interval containing $\{X_i: i \in J_n\}$, by standard empirical process techniques,
	\begin{align*}
	2\log \lambda_n(m_0) & = \sum_{i \in J_n} (f_0(x_0)-\hat{f}_n(X_i))^2 - \sum_{i \in J_n} (f_0(x_0)-\hat{f}_n^0(X_i))^2\\
	& = n\cdot \Prob_n \big[(f_0(x_0)-\hat{f}_n(X))^2- (f_0(x_0)-\hat{f}_n^0(X))^2\big]\bm{1}_{X \in D_n}\\
	& =  n\cdot P\big[(f_0(x_0)-\hat{f}_n(X))^2- (f_0(x_0)-\hat{f}_n^0(X))^2\big]\bm{1}_{X \in D_n}+ \mathfrak{o}_{\mathbf{P}}(1)\\
	& =n \int \big[(f_0(x_0)-\hat{f}_n(x))^2- (f_0(x_0)-\hat{f}_n^0(x))^2\big]\bm{1}_{x \in D_n}\ \d{x} + \mathfrak{o}_{\mathbf{P}}(1).
	\end{align*}
	Change the variable $x= x_0+h\cdot n^{-\frac{\alpha^{-1}}{2+\alpha^{-1}}}$. Let $\tilde{D}_n\equiv n^{\frac{\alpha^{-1}}{2+\alpha^{-1}}}(D_n-x_0)$. Then the above display equals, up to an $\mathfrak{o}_{\mathbf{P}}(1)$ term,
	\begin{align*}
	&\int_{\tilde{D}_n} \big[ \big(n^{\frac{1}{2+\alpha^{-1}}} (f_0(x_0)-\hat{f}_n(x_0+h \cdot n^{-\frac{\alpha^{-1}}{2+\alpha^{-1}}}))\big)^2\\
	&\qquad - \big(n^{\frac{1}{2+\alpha^{-1}}} (f_0(x_0)-\hat{f}_n^0(x_0+ h \cdot n^{-\frac{\alpha^{-1}}{2+\alpha^{-1}}}))\big)^2\big]\ \d{h} \\
	&\rightsquigarrow \int_{\R} \big\{ \big(g_{1,\partial^\alpha f_0(x_0)/(\alpha+1)!;\alpha}(h)\big)^2-\big(g^0_{1,\partial^\alpha f_0(x_0)/(\alpha+1)!;\alpha}(h)\big)^2\big\}\ \d{h},
	\end{align*}
	where the weak convergence follows from tightness arguments. Now we only need to rescale the process to conclude. More specifically, $X_{1,b;\alpha}(h) =_d b^{-1/(2+\alpha)} X_{1,1;\alpha}(b^{2/(2+\alpha)}h)$, and thus $g_{1,b;\alpha}(h) =_d b^{1/(2+\alpha)} g_{1,1;\alpha}(b^{2/(2+\alpha)}h)$, $g^0_{1,b;\alpha}(h) =_d b^{1/(2+\alpha)} g^0_{1,1;\alpha}(b^{2/(2+\alpha)}h)$ (by taking derivative). Therefore with $b\equiv \partial^\alpha f_0(x_0)/(\alpha+1)!$, and $t\equiv b^{2/(2+\alpha) }h$, the integral in the last display is equal in distribution to
	\begin{align*}
	&\int \big\{(g_{1,b;\alpha}(h))^2- (g^0_{1,b;\alpha}(h))^2\big\}\ \d{h}\\
	& = b^{2/(2+\alpha)} \int \big\{(g_{1,1;\alpha}(b^{2/(2+\alpha)} h ))^2- (g^0_{1,1;\alpha}(b^{2/(2+\alpha)} h))^2\big\}\ \d{h}\\
	& = \int \big\{(g_{1,1;\alpha}(t))^2- (g^0_{1,1;\alpha}(t))^2\big\}\ \d{t} = \mathbb{K}_{\alpha},
	\end{align*}
	as desired.
\end{proof}

\section{Proofs in Section \ref{section:more_examples}}

\subsection{Proof of Theorem \ref{thm:CI_grenander}}
	
\begin{proof}[Proof of Theorem \ref{thm:CI_grenander}]
	We sketch the proof here. Let $r_n \equiv n^{-1/3}$ and $\hat{h}_1(x_0),\hat{h}_2(x_0)$ be such that $\hat{u}(x_0) = x_0 - \hat{h}_1(x_0) r_n$, $\hat{v}(x_0) = x_0 + \hat{h}_2(x_0) r_n$. We drop the dependence on $x_0$ in the notation from now on for simplicity. We may then write $\hat{f}_n$ as
	\begin{align*}
	\hat{f}_n(x_0) &= \inf_{h_1>0}\sup_{h_2\geq 0} \frac{\mathbb{F}_n(x_0+h_2 r_n)-\mathbb{F}_n(x_0-h_1 r_n)}{(h_2+h_1)r_n} = \frac{\mathbb{F}_n( x_0+\hat{h}_2 r_n)-\mathbb{F}_n(x_0-\hat{h}_1 r_n)}{(\hat{h}_2+\hat{h}_1)r_n}.
	\end{align*}
	Then it suffices to prove that $\hat{h}_i (i=1,2)$ are bounded away from $\infty$ and $0$ with high probability, in order to adapt the proof of Theorem \ref{thm:pivotal_limit_distribution}. These are referred as `large deviation' and `small deviation' problems in the sequel.
	
	First it is well-known (see e.g., \cite[pp. 297]{van1996weak}) that with $\omega_n \equiv n^{-1/3}$,
	\begin{align}\label{ineq:grenander_1}
	\abs{\hat{f}_n(x_0)-f_0(x_0)} = \mathcal{O}_{\mathbf{P}}(\omega_n).
	\end{align}
	This claim can also be proved directly using the max-min formula by similar arguments as in the regression case in \cite{han2019limit}. 
	
	Next consider the large deviation problem. We only prove that for large enough $c$, $\{\hat{h}_1\leq c\}$ holds with high probability for $n$ large. The case for $\hat{h}_2$ is similar. To see this, with $\bar{\G}_n \equiv \sqrt{n}(\mathbb{F}_n-F_0)$, on the event $\{\hat{h}_1>c\}$, 
	\begin{align*}
	&\omega_n^{-1}\big(\hat{f}_n(x_0)-f_0(x_0)\big)\geq  \omega_n^{-1}\bigg[ \frac{\mathbb{F}_n(x_0+ r_n)-\mathbb{F}_n(x_0-\hat{h}_1 r_n)}{(1+\hat{h}_1)r_n} - f_0(x_0)\bigg]\\
	&\geq \frac{n^{1/6}\bar{\G}_n \bm{1}_{[x_0-\hat{h}_1 n^{-1/3},x_0+n^{-1/3}]}   }{1+\hat{h}_1} + \omega_n^{-1}\bigg[ \frac{F_0(x_0+r_n)-F_0(x_0-\hat{h}_1r_n)}{(1+\hat{h}_1)r_n}-f_0(x_0)\bigg]\\
	&\geq -\sup_{h\geq 0} \frac{n^{1/6} \abs{\bar{\G}_n \bm{1}_{[x_0-h n^{-1/3},x_0+n^{-1/3}]}}   }{1+h}+ \mathcal{O}\bigg(\frac{c^2-1}{c+1}\bigg).
	\end{align*}
	The last inequality follows since the second term of the second last display is non-decreasing in $\hat{h}_1$, so is bounded below by
	\begin{align*}
	&\omega_n^{-1}\bigg[ \frac{F_0(x_0+r_n)-F_0(x_0-cr_n)}{(1+c)r_n}-f_0(x_0)\bigg]\\
	& = \big[(1+c)r_n^2\big]^{-1} \cdot \bigg[F_0(x_0+r_n)-F_0(x_0)-f_0(x_0) r_n\\
	&\qquad\qquad\qquad\qquad -\big(F_0(x_0-cr_n)-F_0(x_0)+f_0(x_0)cr_n\big)\bigg]\\
	& = \big[(1+c)r_n^2\big]^{-1}\cdot \frac{1}{2} f_0'(x_0)(1+\mathfrak{o}(1))(1-c^2)r_n^2 = \mathcal{O}\bigg(\frac{c^2-1}{c+1}\bigg).
	\end{align*}
	By a standard peeling method and empirical process techniques, we have $
	\sup_{h\geq 0} {n^{1/6} \abs{\bar{\G}_n \bm{1}_{[x_0-h n^{-1/3},x_0+n^{-1/3}]}}   }/(1+h) = \mathcal{O}_{\mathbf{P}}(1)$. 
	Hence on the event $\{\hat{h}_1> c\}$ for $c$ large enough,
	\begin{align*}
	\omega_n^{-1}\big(\hat{f}_n(x_0)-f_0(x_0)\big) \geq -\abs{\mathcal{O}_{\mathbf{P}}(1)} + \mathcal{O}(c),
	\end{align*}
	which can only occur with small probability due to (\ref{ineq:grenander_1}).
	
	Finally we consider the small deviation problem. Without loss of generality we work on the event that $\Omega_0\equiv \{\hat{h}_1\leq c\}$. We want to prove that $\Omega_1 \equiv \{\hat{h}_1\leq c^{-\gamma}\}$ occurs with small probability for $c,n$ large. By empirical process techniques, it can be shown that 
	\begin{align*}
	&\G_n(h_1,h_2)\equiv n^{1/6} \bar{\G}_n \bm{1}_{[x_0-h_1 n^{-1/3},x_0+h_2 n^{-1/3}]} \\
	&\qquad\qquad \rightsquigarrow \sqrt{f_0(x_0)} \cdot \G(h_1,h_2) \textrm{ in } \ell^\infty([0,c]\times [0,c]).
	\end{align*}
	Then by standard estimates for the Gaussian process, it follows that, for fixed $\epsilon>0$ and $b<\gamma$, for $c, n$ large enough, the event
	\begin{align*}
	\Omega_2\equiv \bigg\{\sup_{0\leq h_1\leq c^{-\gamma}, 0\leq h_2\leq c^{-b}} \abs{\G_n(h_1,h_2)-\G_n(0,h_2)}\leq (C/\epsilon) \sqrt{c^{-\gamma}\log c}\bigg\}
	\end{align*}
	occurs with probability $1-\epsilon$. Furthermore, since $\mathbb{G}(0,h_2)=_d \mathbb{B}(h_2)$ where $\mathbb{B}$ is a standard Brownian motion started from $0$, we have by reflection principle, for some $\rho_\epsilon>0$, the event $
	\Omega_3 \equiv \big\{\sup_{0\leq h_2\leq c^{-b}} \G_n(0,h_2)>c^{-b/2} \rho_\epsilon\big\}$
	occurs with probability $1-\epsilon$ for $n$ large. Hence on the event $\Omega_0\cap \Omega_1\cap \Omega_2\cap \Omega_3$,
	\begin{align*}
	&\omega_n^{-1}\big(\hat{f}_n(x_0)-f_0(x_0)\big)\geq \sup_{0\leq h_2\leq c^{-b}}  \frac{\G_n(\hat{h}_1,h_2) }{h_2+\hat{h}_1} \\
	&\qquad + \inf_{0\leq h_1\leq c, 0\leq h_2\leq c}\omega_n^{-1}\bigg[ \frac{F_0(x_0+h_2r_n)-F_0(x_0-h_1r_n)}{(h_2+h_1)r_n}-f_0(x_0)\bigg]\\
	&\geq \sup_{0\leq h_2\leq c^{-b}}  \frac{\G_n(0,h_2)-(C/\epsilon)\sqrt{c^{-\gamma}\log c}}{h_2+\hat{h}_1}- \mathcal{O}(c)\\
	&\geq \frac{c^{-b/2}\rho_\epsilon-(C/\epsilon)\sqrt{c^{-\gamma}\log c}}{c^{-b}+c^{-\gamma}}- \mathcal{O}(c) \geq \mathcal{O}(c^{b/2})- \mathcal{O}(c) \geq \mathcal{O}(c),
	\end{align*}
	if we choose $2<b<\gamma$. This concludes the small deviation problem in view of (\ref{ineq:grenander_1}). The rest part of the proof follows closely with that of Theorem \ref{thm:pivotal_limit_distribution}. In particular, we may show 
	\begin{align*}
	&\big(\omega_n^{-1}(\hat{f}_n(x_0)-f_0(x_0)), \hat{h}_1,\hat{h}_2\big)\\
	&\rightsquigarrow \bigg(\sup_{h_1>0}\inf_{h_2>0}\bigg[\sqrt{f_0(x_0)}\cdot \frac{\G(h_1,h_2)}{h_1+h_2}+\frac{1}{2}\abs{f'_0(x_0)} (h_2-h_1)\bigg], h_1^\ast, h_2^\ast \bigg),
	\end{align*}
	where $h_1^\ast,h_2^\ast$ are almost uniquely defined via
	\begin{align*}
	&\sup_{h_1>0}\inf_{h_2>0}\bigg[\sqrt{f_0(x_0)}\cdot \frac{\G(h_1,h_2)}{h_1+h_2}+\frac{1}{2}\abs{f'_0(x_0)} (h_2-h_1)\bigg]\\
	& = \sqrt{f_0(x_0)}\cdot \frac{\G(h_1^\ast,h_2^\ast)}{h_1^\ast+h_2^\ast}+\frac{1}{2}\abs{f'_0(x_0)} (h_2^\ast-h_1^\ast).
	\end{align*}
	By similar arguments as in the proof of Theorem \ref{thm:pivotal_limit_distribution}, we conclude that
	\begin{align*}
	\sqrt{n(\hat{v}(x_0)-\hat{u}(x_0))}\big(\hat{f}_n(x_0)-f_0(x_0)\big)\rightsquigarrow \sqrt{f_0(x_0)} \cdot \mathbb{L}_1,
	\end{align*}
	and the claim of the theorem follows by noting that $\hat{f}_n(x_0)\to_p f_0(x_0)$.
\end{proof}

\subsection{Proof of Theorem \ref{thm:CI_interval_censoring}}

\begin{proof}[Proof of Theorem \ref{thm:CI_interval_censoring}]
	Follow essentially the same lines as in the proof of Theorem \ref{thm:CI_grenander}, we may show that
	\begin{align*}
	\sqrt{n(\hat{v}(t_0)-\hat{u}(t_0))}\big(\hat{F}_n(t_0)-F_0(t_0)\big)\rightsquigarrow \sqrt{F_0(t_0)(1-F_0(t_0))/g_0(t_0)}\cdot  \mathbb{L}_1.
	\end{align*}
	On the other hand, $\hat{F}_n(t_0)\to_p F_0(t_0)$, so it remains to show that $\hat{g}_n(t_0)\to_p g_0(t_0)$. This is a uniform law of large number upon observing that $\hat{u}(t_0),\hat{v}(t_0)$ stabilizes at $r_n=n^{-1/3}$ rate. More specifically, for any $\epsilon>0$, we may find some $c=c(\epsilon)>0$ such that $c^{-1}\leq \hat{h}_1(t_0),\hat{h}_2(t_0)\leq c$, where $\hat{u}(t_0) = t_0 - \hat{h}_1(t_0) r_n$, $\hat{v}(t_0) = t_0 + \hat{h}_2(t_0) r_n$. On this event,
	\begin{align*}
	\abs{\hat{g}_n(t_0)-g_0(t_0)}& =  \biggabs{\frac{1}{nr_n (\hat{h}_2+\hat{h}_1) } \sum_i  \bm{1}_{T_i \in [t_0-\hat{h}_1 r_n,t_0+\hat{h}_2 r_n]} -g_0(t_0)}\\
	& \leq \sup_{c^{-1}\leq h_i\leq c, i=1,2} \biggabs{\frac{1}{r_n(h_2+h_1)} \Prob_n \bm{1}_{T \in [t_0-h_1 r_n,t_0+h_2 r_n]} -g_0(t_0) }\\ 
	&\lesssim_c \sup_{c^{-1}\leq h_i\leq c, i=1,2}\bigabs{r_n^{-1}(\Prob_n-P) \bm{1}_{T \in [t_0-h_1 r_n,t_0+h_2 r_n]} }\\
	&\qquad + \sup_{c^{-1}\leq h_i\leq c, i=1,2} \biggabs{\frac{P \bm{1}_{T \in [t_0-h_1 r_n,t_0+h_2 r_n]}}{r_n(h_2+h_1) }-g_0(t_0)}.
	\end{align*}
	To handle the first term, let $\mathcal{F}_n\equiv \{r_n^{-1} \bm{1}_{T \in [t_0-h_1 r_n,t_0+h_2r_n]}: c^{-1}\leq h_i\leq c, i=1,2\}$. Take a $\delta$-bracketing of $\mathcal{F}_n$ under $L_1(P)$, namely $[\underline{f}_1,\bar{f}_1],\ldots,[\underline{f}_N,\bar{f}_N]$. Then $N_\delta \equiv \log \mathcal{N}_{[\,]}(\delta, \mathcal{F}_n, L_1(P))\lesssim \log(c/\delta)$ does not grow with $n$. Since for any $f \in \mathcal{F}_n$, there exists some $\underline{f}_k\leq f\leq \bar{f}_k$, we have
	\begin{align*}
	(\Prob_n-P)(f) \leq \Prob_n \bar{f}_k - Pf = (\Prob_n-P)(\bar{f}_k) +P(\bar{f}_k-f)\leq (\Prob_n-P)(\bar{f}_k)+\delta.
	\end{align*}
	A reversed inequality can be established similarly. Hence
	\begin{align*}
	\E \sup_{f \in \mathcal{F}_n}\abs{(\Prob_n-P)(f)}\leq \delta+ \E \max_{1\leq k\leq N_\delta} \abs{(\Prob_n-P)(\underline{f}_k)}+ \E \max_{1\leq k\leq N_\delta} \abs{(\Prob_n-P)(\bar{f}_k)}\to 0
	\end{align*}
	by letting $n \to \infty$ followed by $\delta \to 0$. For the second term, 
	\begin{align*}
	&\biggabs{ \frac{P \bm{1}_{T \in [t_0-h_1 r_n,t_0+h_2 r_n]}}{r_n(h_2+h_1) } -g_0(t_0)}=\biggabs{ \frac{1}{r_n(h_2+h_1)} \int_{t_0-h_1r_n}^{t_0+h_2r_n} \big(g_0(t)-g_0(t_0)\big)\ \d{t}}\\
	&\leq \sup_{t \in [t_0-h_1r_n,t_0+h_2r_n]} \abs{g_0(t)-g_0(t_0)}\to  0
	\end{align*}
	uniformly in $c^{-1}\leq h_i\leq c, i=1,2$.
\end{proof}

\subsection{Proof of Theorem \ref{thm:CI_panel_count}}

\begin{proof}[Proof of Theorem \ref{thm:CI_panel_count}]
	Following essentially the same lines as in the proof of Theorem \ref{thm:CI_grenander}, we may show that
	\begin{align*}
	\sqrt{n(\hat{v}(t_0)-\hat{u}(t_0))}\big(\hat{\Lambda}_n(t_0)-\Lambda_0(t_0)\big)\rightsquigarrow \sqrt{\sigma^2(t_0)/g(t_0)}\cdot  \mathbb{L}_1.
	\end{align*}
	It therefore remains to show that $\hat{\sigma}_n^2(t_0)\to_p \sigma^2(t_0)$ and $\hat{g}_n(t_0)\to_p g(t_0)$. We use similar ideas as in the proof of Theorem \ref{thm:CI_interval_censoring}. Let $r_n = n^{-1/3}$. For any $\epsilon>0$, we may find some $c=c(\epsilon)>0$ such that $c^{-1}\leq \hat{h}_1(t_0),\hat{h}_2(t_0)\leq c$, where $\hat{u}(t_0) = t_0 - \hat{h}_1(t_0) r_n$, $\hat{v}(t_0) = t_0 + \hat{h}_2(t_0) r_n$. On this event,
	\begin{align*}
	\abs{\hat{g}_n(t_0)-g(t_0)}&\lesssim \sup_{c^{-1}\leq h_i\leq c, i=1,2}\biggabs{r_n^{-1}(\Prob_n-P) \bigg(\sum_{j=1}^K \bm{1}_{T_{K,j} \in [t_0-h_1 r_n,t_0+h_2 r_n]}\bigg) }\\
	&\qquad + \sup_{c^{-1}\leq h_i\leq c, i=1,2}\biggabs{\frac{P \big(\sum_{j=1}^K \bm{1}_{T_{K,j} \in [t_0-h_1 r_n,t_0+h_2 r_n]}\big)}{r_n(h_2+h_1)}-g(t_0) }.
	\end{align*}
	The first term converges to $0$ in probability, while the second term can be handled through
	\begin{align*}
	& \sup_{c^{-1}\leq h_i\leq c, i=1,2}\biggabs{\frac{P \big(\sum_{j=1}^K \bm{1}_{T_{K,j} \in [t_0-h_1 r_n,t_0+h_2 r_n]}\big)}{r_n(h_2+h_1)}-g(t_0) }\\
	&\leq \sup_{c^{-1}\leq h_i\leq c, i=1,2}\biggabs{  \sum_{k=1}^\infty \Prob(K=k) \sum_{j=1}^k \bigg[\frac{1}{r_n(h_2+h_1)}\int_{t_0-h_1 r_n}^{t_0+h_2r_n} \big(g_{k,j}(t)-g_{k,j}(t_0)\big)\ \d{t} \bigg]}\\
	&\leq  \biggabs{\sum_{k=1}^\infty k \cdot \Prob(K=k)} \sup_{1\leq j\leq k, k\geq 1}\sup_{t \in [t_0-c r_n, t_0+c r_n]} \abs{g_{k,j}(t)-g_{k,j}(t_0)}\to 0
	\end{align*}
	by \cite[conditions (E1)-(E2)]{wellner2000two}.
	
	On the other hand, on the same event,
	\begin{align*}
	&\abs{\hat{\sigma}_n^2(t_0)\hat{g}_n(t_0)-\sigma^2(t_0)g(t_0)}\\
	&\lesssim \biggabs{\frac{1}{n(\hat{v}(t_0)-\hat{u}(t_0))}\sum_{i=1}^n \sum_{j=1}^{K_i} \big(N_{K_i,j}^{(i)}-\Lambda_0(T_{K_i,j}^{(i)})\big)^2 \bm{1}_{T_{K_i,j}^{(i)} \in [t_0-h_1 r_n,t_0+h_2r_n] }-\sigma^2(t_0)g(t_0)}\\
	&\qquad + \hat{g}_n(t_0)\cdot \bigg[\big(\hat{\Lambda}_n(t_0)-\Lambda_0(t_0)\big)^2+\sup_{t \in [t_0-h_1r_n,t_0+h_2r_n]}\big(\Lambda_0(t)-\Lambda_0(t_0)\big)^2\bigg] \\
	&\lesssim \sup_{c^{-1}\leq h_i\leq c, i=1,2}\biggabs{r_n^{-1}(\Prob_n-P) \bigg(\sum_{j=1}^K (N_{K,j}-\Lambda_0(T_{K,j}))^2\bm{1}_{T_{K,j} \in [t_0-h_1 r_n,t_0+h_2 r_n]}\bigg) }\\
	&\qquad +  \sup_{c^{-1}\leq h_i\leq c, i=1,2}\bigg\lvert \frac{P \big(\sum_{j=1}^K (N_{K,j}-\Lambda_0(T_{K,j}))^2 \bm{1}_{T_{K,j} \in [t_0-h_1 r_n,t_0+h_2 r_n]}\big)}{r_n(h_2+h_1)}\\
	&\qquad\qquad\qquad\qquad-\sigma^2(t_0)g(t_0) \bigg\lvert+ \mathfrak{o}_{\mathbf{P}}(1).
	\end{align*}
	The first term above can be handled similarly as in the proof in Theorem \ref{thm:CI_interval_censoring}, which converges to $0$ in probability. For the second term, by independence of $N$ and $(K,T)$, the second term equals
	\begin{align*}
	&\sup_{c^{-1}\leq h_i\leq c, i=1,2}\biggabs{\frac{P \big(\sum_{j=1}^K \sigma^2(T_{K,j}) \bm{1}_{T_{K,j} \in [t_0-h_1 r_n,t_0+h_2 r_n]}\big)}{r_n(h_2+h_1)}-\sigma^2(t_0)g(t_0) }\\
	& =\sup_{c^{-1}\leq h_i\leq c, i=1,2}\biggabs{ \sum_{k=1}^\infty \Prob(K=k)\sum_{j=1}^k \frac{\int_{t_0-h_1r_n}^{t_0+h_2 r_n} \big(\sigma^2(t) g_{k,j}(t)-\sigma^2(t_0)g_{k,j}(t_0)\big)\ \d{t} }{r_n(h_2+h_1)} }\\
	& \leq \biggabs{\sum_{k=1}^\infty k \cdot \Prob(K=k)} \sup_{1\leq j\leq k, k\geq 1}\sup_{t \in [t_0-c r_n, t_0+c r_n]} \abs{\sigma^2(t)g_{k,j}(t)-\sigma^2(t_0)g_{k,j}(t_0)}\to 0,
	\end{align*}
	where we used \cite[condition (E5)]{wellner2000two} to complete the proof.
\end{proof}

\subsection{Proof of Theorems \ref{thm:CI_GLM} and \ref{thm:CI_GLM_1}}

\begin{proof}[Proof of Theorems \ref{thm:CI_GLM} and \ref{thm:CI_GLM_1}]
	
We first prove Theorem \ref{thm:CI_GLM}.
	
(1) By definition, $\omega_n|S_{n,0,1}|^{1/2}=\{n\omega_n^{2+1/\alpha}\}^{1/2} + \mathfrak{o}(1) = 1+\mathfrak{o}(1)$ 
and $|S_{n,h_1,h_2}|/|S_{n,0,1}| = h_1+h_2+\mathfrak{o}(1)$ uniformly in $(h_1,h_2)\in [1/c,c]^2$ for every $c >1$. 
Thus, by the Donsker-Prokhorov invariance principle (or one may use \cite[Theorem 2.11.9]{van1996weak}), the second and third lines of \eqref{inid-cond} imply 
\begin{align*}
\sum_{i\in S_{n,h_1,h_2}} \frac{Y_i - \theta_0(x_i)}{\omega_n|S_{n,0,1}|} 
\rightsquigarrow \sigma(x_0) \cdot \G(h_1,h_2)
\end{align*}
in $(h_1,h_2)\in [1/c,c]^2$ for every $c >1$. 
Here $\G(h_1,h_2)\equiv\mathbb{B}(-h_1) -\mathbb{B}(h_2)$ for a standard two-sided Brownian motion $\mathbb{B}$ starting from 0, is defined in Theorem \ref{thm:limit_distribution_pointwise} with $d=1$. This functional CLT and the first line of \eqref{inid-cond} yield part (i) by the proof for i.i.d.\, errors. 

(2) Let $u_n \equiv x_0-(h_1\omega_n^{1/\alpha}/\pi_0)^{1/\beta}$ and $v_n\equiv x_0+(h_2\omega_n^{1/\alpha}/\pi_0)^{1/\beta}$. 
As $\int_{u_n}^{x_0} \pi(x)\ \d{x} =(1+\mathfrak{o}(1)) h_1\omega_n^{1/\alpha}$ 
and $\int_{x_0}^{v_n} \pi(x)\ \d{x} = (1+\mathfrak{o}(1)) h_2\omega_n^{1/\alpha}$, 
\begin{align*}%\label{inid-cond}
& \sum_{i\in S_{n,h_1,h_2}} \frac{\theta_0(x_i)-\theta_0(x_0)}{\omega_n|S_{n,h_1,h_2}|} \\
&= \frac{\int_{u_n}^{v_n} \partial^{\alpha\beta} \theta_0(x_0)(x-x_0)^{\alpha\beta} \pi_0 \beta |x-x_0|^{\beta-1}\ \d{x}}
{(\alpha\beta)! \omega_n^{1+1/\alpha}(h_1+h_2)} + \frac{ \mathfrak{o}_{\mathbf{P}}(1)  (v_n-u_n)^{(\alpha+1)\beta}}{\omega_n^{1+1/\alpha}}\\
&= \frac{\pi_0\partial^{\alpha\beta} \theta_0(x_0)}{(\alpha\beta)!(\alpha+1)} \frac{(v_n-x_0)^{(\alpha+1)\beta} - (x_0-u_n)^{(\alpha+1)\beta}}
{\omega_n^{1+1/\alpha}(h_1+h_2)} + \mathfrak{o}_{\mathbf{P}}(1) \\
  &= g_0(x_0)\frac{h_2^{\alpha+1}-h_1^{\alpha+1}}{h_1+h_2} +  \mathfrak{o}_{\mathbf{P}}(1)  
\end{align*}
with $g_0(x_0) \equiv \partial^{\alpha\beta}\theta_0(x_0)/\{(\alpha\beta)!(\alpha+1)\pi_0^{\alpha}\}$. 
This gives the first line of \eqref{inid-cond}. 

For the exponential family, $\sigma^2(x)=1/p^\prime(\theta_0(x))$ is continuous in $\theta_0(x)$ 
so that the continuity of $\theta_0(\cdot)$ implies the second line of \eqref{inid-cond}. 
As the variance of $f(y;\theta)$ is finite at $\theta = \theta_0(0)$ and $\theta = \theta_0(1)$, 
$\{(Y_i-\theta_0(x_i))^2, 1\le i\le n\}, n\ge 1$, are uniformly integrable, so that 
the Lindeberg condition also holds. 

Theorem \ref{thm:CI_GLM_1} can be proved along similar lines as in the proof of Theorem \ref{thm:CI_grenander}. Details are omitted.
\end{proof}

\section*{Acknowledgments}
The authors would like to thank three referees for their helpful comments and suggestions that improved the quality of the paper.

\bibliographystyle{amsalpha}
\bibliography{mybib}

\providecommand{\bysame}{\leavevmode\hbox to3em{\hrulefill}\thinspace}
\providecommand{\MR}{\relax\ifhmode\unskip\space\fi MR }
\providecommand{\MRhref}[2]{%
  \href{http://www.ams.org/mathscinet-getitem?mr=#1}{#2}
}
\providecommand{\href}[2]{#2}
\begin{thebibliography}{MMTW01}

\bibitem[Ban07]{banerjee2007likelihood}
Moulinath Banerjee, \emph{Likelihood based inference for monotone response
  models}, Ann. Statist. \textbf{35} (2007), no.~3, 931--956. \MR{2341693}

\bibitem[Ban08]{banerjee2008estimating}
\bysame, \emph{Estimating monotone, unimodal and {U}-shaped failure rates using
  asymptotic pivots}, Statist. Sinica \textbf{18} (2008), no.~2, 467--492.
  \MR{2411614}

\bibitem[BBBB72]{barlow1972statistical}
R.~E. Barlow, D.~J. Bartholomew, J.~M. Bremner, and H.~D. Brunk,
  \emph{Statistical inference under order restrictions. {T}he theory and
  application of isotonic regression}, John Wiley \& Sons, London-New
  York-Sydney, 1972, Wiley Series in Probability and Mathematical Statistics.
  \MR{0326887}

\bibitem[Bel18]{bellec2018sharp}
Pierre~C. Bellec, \emph{Sharp oracle inequalities for {L}east {S}quares
  estimators in shape restricted regression}, Ann. Statist. \textbf{46} (2018),
  no.~2, 745--780. \MR{3782383}

\bibitem[BM07]{banerjee2007confidence}
Moulinath Banerjee and Ian~W. McKeague, \emph{Confidence sets for split points
  in decision trees}, Ann. Statist. \textbf{35} (2007), no.~2, 543--574.
  \MR{2336859}

\bibitem[Bru70]{brunk1970estimation}
H.~D. Brunk, \emph{Estimation of isotonic regression}, Nonparametric
  {T}echniques in {S}tatistical {I}nference ({P}roc. {S}ympos., {I}ndiana
  {U}niv., {B}loomington, {I}nd., 1969), Cambridge Univ. Press, London, 1970,
  pp.~177--197. \MR{0277070}

\bibitem[BRW09]{balabdaoui2009limit}
Fadoua Balabdaoui, Kaspar Rufibach, and Jon~A. Wellner, \emph{Limit
  distribution theory for maximum likelihood estimation of a log-concave
  density}, Ann. Statist. \textbf{37} (2009), no.~3, 1299--1331. \MR{2509075
  (2010h:62290)}

\bibitem[BW01]{banerjee2001likelihood}
Moulinath Banerjee and Jon~A. Wellner, \emph{Likelihood ratio tests for
  monotone functions}, Ann. Statist. \textbf{29} (2001), no.~6, 1699--1731.
  \MR{1891743}

\bibitem[BW07]{balabdaoui2007estimation}
Fadoua Balabdaoui and Jon~A. Wellner, \emph{Estimation of a {$k$}-monotone
  density: limit distribution theory and the spline connection}, Ann. Statist.
  \textbf{35} (2007), no.~6, 2536--2564. \MR{2382657}

\bibitem[CGS15]{chatterjee2015risk}
Sabyasachi Chatterjee, Adityanand Guntuboyina, and Bodhisattva Sen, \emph{On
  risk bounds in isotonic and other shape restricted regression problems}, Ann.
  Statist. \textbf{43} (2015), no.~4, 1774--1800. \MR{3357878}

\bibitem[CGS18]{chatterjee2018matrix}
\bysame, \emph{On matrix estimation under monotonicity constraints}, Bernoulli
  \textbf{24} (2018), no.~2, 1072--1100. \MR{3706788}

\bibitem[CL19]{chatterjee2015adaptive}
Sabyasachi Chatterjee and John Lafferty, \emph{Adaptive risk bounds in unimodal
  regression}, Bernoulli \textbf{25} (2019), no.~1, 1--25. \MR{3892309}

\bibitem[CS16]{chen2016generalized}
Yining Chen and Richard~J. Samworth, \emph{Generalized additive and index
  models with shape constraints}, J. R. Stat. Soc. Ser. B. Stat. Methodol.
  \textbf{78} (2016), no.~4, 729--754. \MR{3534348}

\bibitem[Dos19]{doss2019concave}
Charles~R. Doss, \emph{Concave regression: value-constrained estimation and
  likelihood ratio-based inference}, Math. Program. \textbf{174} (2019),
  no.~1-2, Ser. B, 5--39. \MR{3935071}

\bibitem[DR09]{dumbgen2009maximum}
Lutz D{\"u}mbgen and Kaspar Rufibach, \emph{Maximum likelihood estimation of a
  log-concave density and its distribution function: basic properties and
  uniform consistency}, Bernoulli \textbf{15} (2009), no.~1, 40--68.
  \MR{2546798 (2011b:62096)}

\bibitem[DSS11]{dumbgen2011approximation}
Lutz D{\"u}mbgen, Richard Samworth, and Dominic Schuhmacher,
  \emph{Approximation by log-concave distributions, with applications to
  regression}, Ann. Statist. \textbf{39} (2011), no.~2, 702--730. \MR{2816336
  (2012e:62039)}

\bibitem[DW16]{doss2013global}
Charles~R. Doss and Jon~A. Wellner, \emph{Global rates of convergence of the
  {MLE}s of log-concave and {$s$}-concave densities}, Ann. Statist. \textbf{44}
  (2016), no.~3, 954--981. \MR{3485950}

\bibitem[DW19]{doss2016inference}
\bysame, \emph{Inference for the mode of a log-concave density}, Ann. Statist.
  \textbf{47} (2019), no.~5, 2950--2976. \MR{3988778}

\bibitem[DZ20]{deng2018isotonic}
Hang Deng and Cun-Hui Zhang, \emph{Isotonic regression in multi-dimensional
  spaces and graphs}, Ann. Statist. (to appear). Available at arXiv:1812.08944
  (2020+).

\bibitem[FGKS18]{feng2018adaptation}
Oliver~Y Feng, Adityanand Guntuboyina, Arlene~KH Kim, and Richard~J Samworth,
  \emph{Adaptation in multivariate log-concave density estimation}, Ann.
  Statist. (to appear). Available at arXiv:1812.11634 (2018).

\bibitem[FGS19]{fang2019multivariate}
Billy Fang, Adityanand Guntuboyina, and Bodhisattva Sen, \emph{Multivariate
  extensions of isotonic regression and total variation denoising via entire
  monotonicity and {H}ardy-{K}rause variation}, arXiv preprint arXiv:1903.01395
  (2019).

\bibitem[FLN17]{fokianos2017integrated}
Konstantinos Fokianos, Anne Leucht, and Michael~H Neumann, \emph{On integrated
  $l_1$ convergence rate of an isotonic regression estimator for multivariate
  observations}, arXiv preprint arXiv:1710.04813 (2017).

\bibitem[GJ95]{groeneboom1995isotonic}
Piet Groeneboom and Geurt Jongbloed, \emph{Isotonic estimation and rates of
  convergence in {W}icksell's problem}, Ann. Statist. \textbf{23} (1995),
  no.~5, 1518--1542. \MR{1370294}

\bibitem[GJ14]{groeneboom2014nonparametric}
\bysame, \emph{Nonparametric estimation under shape constraints}, Cambridge
  Series in Statistical and Probabilistic Mathematics, vol.~38, Cambridge
  University Press, New York, 2014. \MR{3445293}

\bibitem[GJ15]{groeneboom2015nonparametric}
\bysame, \emph{Nonparametric confidence intervals for monotone functions}, Ann.
  Statist. \textbf{43} (2015), no.~5, 2019--2054. \MR{3375875}

\bibitem[GJW01a]{groeneboom2001canonical}
Piet Groeneboom, Geurt Jongbloed, and Jon~A. Wellner, \emph{A canonical process
  for estimation of convex functions: the ``invelope'' of integrated {B}rownian
  motion {$+t^4$}}, Ann. Statist. \textbf{29} (2001), no.~6, 1620--1652.
  \MR{1891741}

\bibitem[GJW01b]{groeneboom2001estimation}
\bysame, \emph{Estimation of a convex function: characterizations and
  asymptotic theory}, Ann. Statist. \textbf{29} (2001), no.~6, 1653--1698.
  \MR{1891742 (2003a:62047)}

\bibitem[GN16]{gine2015mathematical}
Evarist Gin\'{e} and Richard Nickl, \emph{Mathematical foundations of
  infinite-dimensional statistical models}, Cambridge Series in Statistical and
  Probabilistic Mathematics, [40], Cambridge University Press, New York, 2016.
  \MR{3588285}

\bibitem[Gre56]{grenander1956theory}
Ulf Grenander, \emph{On the theory of mortality measurement. {II}}, Skand.
  Aktuarietidskr. \textbf{39} (1956), 125--153 (1957). \MR{0093415}

\bibitem[Gro85]{groeneboom1985estimating}
Piet Groeneboom, \emph{Estimating a monotone density}, Proceedings of the
  {B}erkeley conference in honor of {J}erzy {N}eyman and {J}ack {K}iefer,
  {V}ol. {II} ({B}erkeley, {C}alif., 1983), Wadsworth Statist./Probab. Ser.,
  Wadsworth, Belmont, CA, 1985, pp.~539--555. \MR{822052}

\bibitem[Gro89]{groeneboom1989brownian}
\bysame, \emph{Brownian motion with a parabolic drift and {A}iry functions},
  Probab. Theory Related Fields \textbf{81} (1989), no.~1, 79--109. \MR{981568}

\bibitem[Gro15]{groeneboom2015rcpp}
\bysame, \emph{Rcpp scripts},
  \url{https://github.com/pietg/book/tree/master/Rcpp_scripts/}, 2015.

\bibitem[GS15]{guntuboyina2013global}
Adityanand Guntuboyina and Bodhisattva Sen, \emph{Global risk bounds and
  adaptation in univariate convex regression}, Probab. Theory Related Fields
  \textbf{163} (2015), no.~1-2, 379--411. \MR{3405621}

\bibitem[GS17]{ghosal2017univariate}
Promit Ghosal and Bodhisattva Sen, \emph{On univariate convex regression},
  Sankhya A \textbf{79} (2017), no.~2, 215--253. \MR{3707421}

\bibitem[GW92]{groeneboom1992information}
Piet Groeneboom and Jon~A. Wellner, \emph{Information bounds and nonparametric
  maximum likelihood estimation}, DMV Seminar, vol.~19, Birkh\"{a}user Verlag,
  Basel, 1992. \MR{1180321}

\bibitem[Han19]{han2019}
Qiyang Han, \emph{Global empirical risk minimizers with ``shape constraints''
  are rate optimal in general dimensions}, arXiv preprint arXiv:1905.12823
  (2019).

\bibitem[Hil54]{hildreth1954point}
Clifford Hildreth, \emph{Point estimates of ordinates of concave functions}, J.
  Amer. Statist. Assoc. \textbf{49} (1954), 598--619. \MR{0065093 (16,382f)}

\bibitem[HK19]{han2019berry}
Qiyang Han and Kengo Kato, \emph{Berry-{E}sseen bounds for {C}hernoff-type
  non-standard asymptotics in isotonic regression}, arXiv preprint
  arXiv:1910.09662 (2019).

\bibitem[HKT91]{hall1991estimation}
Peter Hall, J.~W. Kay, and D.~M. Titterington, \emph{On estimation of noise
  variance in two-dimensional signal processing}, Advances in {A}pplied
  {P}robability \textbf{23} (1991), no.~3, 476--495.

\bibitem[HP76]{hanson1976consistency}
D.~L. Hanson and Gordon Pledger, \emph{Consistency in concave regression}, Ann.
  Statist. \textbf{4} (1976), no.~6, 1038--1050. \MR{0426273 (54 \#14219)}

\bibitem[HW16a]{han2015approximation}
Qiyang Han and Jon~A. Wellner, \emph{Approximation and estimation of
  {$s$}-concave densities via {R}\'{e}nyi divergences}, Ann. Statist.
  \textbf{44} (2016), no.~3, 1332--1359. \MR{3485962}

\bibitem[HW16b]{han2016multivariate}
\bysame, \emph{Multivariate convex regression: global risk bounds and
  adaptation}, arXiv preprint arXiv:1601.06844 (2016).

\bibitem[HWCS19]{han2017isotonic}
Qiyang Han, Tengyao Wang, Sabyasachi Chatterjee, and Richard~J. Samworth,
  \emph{Isotonic regression in general dimensions}, Ann. Statist. \textbf{47}
  (2019), no.~5, 2440--2471. \MR{3988762}

\bibitem[HZ19]{han2019limit}
Qiyang Han and Cun-Hui Zhang, \emph{Limit distribution theory for block
  estimators in multiple isotonic regression}, Ann. Statist. (to appear).
  Available at arXiv:1905.12825 (2019+).

\bibitem[JW09]{jankowski2009nonparametric}
Hanna~K. Jankowski and Jon~A. Wellner, \emph{Nonparametric estimation of a
  convex bathtub-shaped hazard function}, Bernoulli \textbf{15} (2009), no.~4,
  1010--1035. \MR{2597581}

\bibitem[KGS18]{kim2016adaptation}
Arlene K.~H. Kim, Adityanand Guntuboyina, and Richard~J. Samworth,
  \emph{Adaptation in log-concave density estimation}, Ann. Statist.
  \textbf{46} (2018), no.~5, 2279--2306. \MR{3845018}

\bibitem[KM10]{koenker2010quasi}
Roger Koenker and Ivan Mizera, \emph{Quasi-concave density estimation}, Ann.
  Statist. \textbf{38} (2010), no.~5, 2998--3027. \MR{2722462 (2011j:62108)}

\bibitem[Kos08]{kosorok2008bootstrapping}
Michael~R. Kosorok, \emph{Bootstrapping in {G}renander estimator}, Beyond
  parametrics in interdisciplinary research: {F}estschrift in honor of
  {P}rofessor {P}ranab {K}. {S}en, Inst. Math. Stat. (IMS) Collect., vol.~1,
  Inst. Math. Statist., Beachwood, OH, 2008, pp.~282--292. \MR{2462212}

\bibitem[KS16]{kim2016global}
Arlene K.~H. Kim and Richard~J. Samworth, \emph{Global rates of convergence in
  log-concave density estimation}, Ann. Statist. \textbf{44} (2016), no.~6,
  2756--2779. \MR{3576560}

\bibitem[Kuo08]{kuosmanen2008representation}
Timo Kuosmanen, \emph{Representation theorem for convex nonparametric least
  squares}, The Econometrics Journal \textbf{11} (2008), no.~2, 308--325.

\bibitem[LG12]{lim2012consistency}
Eunji Lim and Peter~W. Glynn, \emph{Consistency of multidimensional convex
  regression}, Oper. Res. \textbf{60} (2012), no.~1, 196--208. \MR{2911667}

\bibitem[Mam91]{mammen1991nonparametric}
Enno Mammen, \emph{Nonparametric regression under qualitative smoothness
  assumptions}, Ann. Statist. \textbf{19} (1991), no.~2, 741--759. \MR{1105842
  (92j:62051)}

\bibitem[MBWF05]{munk2005difference}
Axel Munk, Nicolai Bissantz, Thorsten Wagner, and Gudrun Freitag, \emph{On
  difference-based variance estimation in nonparametric regression when the
  covariate is high dimensional}, J. R. Stat. Soc. Ser. B Stat. Methodol.
  \textbf{67} (2005), no.~1, 19--41. \MR{2136637}

\bibitem[MMTW01]{mammen2001general}
E.~Mammen, J.~S. Marron, B.~A. Turlach, and M.~P. Wand, \emph{A general
  projection framework for constrained smoothing}, Statist. Sci. \textbf{16}
  (2001), no.~3, 232--248. \MR{1874153}

\bibitem[PR69]{rao1969estimation}
B.~L.~S. Prakasa~Rao, \emph{Estimation of a unimodal density}, Sankhy\=a Ser. A
  \textbf{31} (1969), 23--36. \MR{0267677}

\bibitem[PR70]{rao1970estimation}
\bysame, \emph{Estimation for distributions with monotone failure rate}, Ann.
  Math. Statist. \textbf{41} (1970), 507--519. \MR{0260133}

\bibitem[Ric84]{rice1984bandwidth}
John Rice, \emph{Bandwidth choice for nonparametric regression}, Ann. Statist.
  \textbf{12} (1984), no.~4, 1215--1230. \MR{760684}

\bibitem[Rud91]{rudin1991functional}
Walter Rudin, \emph{Functional analysis}, second ed., International Series in
  Pure and Applied Mathematics, McGraw-Hill, Inc., New York, 1991. \MR{1157815}

\bibitem[RWD88]{robertson1988order}
Tim Robertson, F.~T. Wright, and R.~L. Dykstra, \emph{Order restricted
  statistical inference}, Wiley Series in Probability and Mathematical
  Statistics: Probability and Mathematical Statistics, John Wiley \& Sons,
  Ltd., Chichester, 1988. \MR{961262}

\bibitem[SB07]{sen2008pseudo}
Bodhisattva Sen and Moulinath Banerjee, \emph{A pseudo-likelihood method for
  analyzing interval censored data}, Biometrika \textbf{94} (2007), no.~1,
  71--86. \MR{2307901}

\bibitem[SBW10]{sen2010inconsistency}
Bodhisattva Sen, Moulinath Banerjee, and Michael Woodroofe, \emph{Inconsistency
  of bootstrap: the {G}renander estimator}, Ann. Statist. \textbf{38} (2010),
  no.~4, 1953--1977. \MR{2676880}

\bibitem[SK95]{sun1995estimation}
J.~Sun and J.~D. Kalbfleisch, \emph{Estimation of the mean function of point
  processes based on panel count data}, Statist. Sinica \textbf{5} (1995),
  no.~1, 279--289. \MR{1329298}

\bibitem[SS11a]{seijo2011change}
Emilio Seijo and Bodhisattva Sen, \emph{Change-point in stochastic design
  regression and the bootstrap}, Ann. Statist. \textbf{39} (2011), no.~3,
  1580--1607. \MR{2850213}

\bibitem[SS11b]{seijo2011nonparametric}
\bysame, \emph{Nonparametric least squares estimation of a multivariate convex
  regression function}, Ann. Statist. \textbf{39} (2011), no.~3, 1633--1657.
  \MR{2850215 (2012j:62119)}

\bibitem[SW10]{seregin2010nonparametric}
Arseni Seregin and Jon~A. Wellner, \emph{Nonparametric estimation of
  multivariate convex-transformed densities}, Ann. Statist. \textbf{38} (2010),
  no.~6, 3751--3781. \MR{2766867 (2012b:62126)}

\bibitem[vdVW96]{van1996weak}
Aad van~der Vaart and Jon~A. Wellner, \emph{Weak {C}onvergence and {E}mpirical
  {P}rocesses}, Springer Series in Statistics, Springer-Verlag, New York, 1996.
  \MR{1385671 (97g:60035)}

\bibitem[Wri81]{wright1981asymptotic}
F.~T. Wright, \emph{The asymptotic behavior of monotone regression estimates},
  Ann. Statist. \textbf{9} (1981), no.~2, 443--448. \MR{606630}

\bibitem[WZ00]{wellner2000two}
Jon~A. Wellner and Ying Zhang, \emph{Two estimators of the mean of a counting
  process with panel count data}, Ann. Statist. \textbf{28} (2000), no.~3,
  779--814. \MR{1792787}

\bibitem[XS20]{xu2019high}
Min Xu and Richard~J Samworth, \emph{High-dimensional nonparametric density
  estimation via symmetry and shape constraints}, Ann. Statist. (to appear).
  Available at arXiv:1903.06092 (2020+).

\bibitem[Zha02]{zhang2002risk}
Cun-Hui Zhang, \emph{Risk bounds in isotonic regression}, Ann. Statist.
  \textbf{30} (2002), no.~2, 528--555. \MR{1902898 (2003e:62084)}

\end{thebibliography}

\end{document}